\documentclass{amsart}

\usepackage{amsmath,amssymb}
\usepackage{mathrsfs}
\usepackage{mathtools}
\usepackage{stmaryrd}
\usepackage{enumerate}
\usepackage{tikz-cd}
\usepackage[all]{xy}
\usepackage{aliascnt}
\usepackage[colorlinks=true, linkcolor=blue, citecolor=blue]{hyperref}

\usepackage{enumitem}

\usepackage{combelow}

\newtheorem{theorem}{Theorem}

\newaliascnt{lemma}{theorem}
\newtheorem{lemma}[lemma]{Lemma}
\aliascntresetthe{lemma}

\newaliascnt{corollary}{theorem}
\newtheorem{corollary}[corollary]{Corollary}
\aliascntresetthe{corollary}

\newaliascnt{proposition}{theorem}
\newtheorem{proposition}[proposition]{Proposition}
\aliascntresetthe{proposition}

\newaliascnt{conjecture}{theorem}
\newtheorem{conjecture}[conjecture]{Conjecture}
\aliascntresetthe{conjecture}

\newaliascnt{question}{theorem}

\aliascntresetthe{question}

\theoremstyle{definition}

\newaliascnt{definition}{theorem}
\newtheorem{definition}[definition]{Definition}
\aliascntresetthe{definition}

\newaliascnt{remark}{theorem}
\newtheorem{remark}[remark]{Remark}
\aliascntresetthe{remark}

\newaliascnt{example}{theorem}

\aliascntresetthe{example}

\newaliascnt{notation}{theorem}
\newtheorem{notation}[notation]{Notation}
\aliascntresetthe{notation}

\newtheorem*{acknowledgements}{Acknowledgements}

\newif\ifhascomments \hascommentstrue
\ifhascomments
\else
\fi

\renewcommand{\setminus}{\smallsetminus}

\newcommand{\Z}{\mathbb{Z}}
\newcommand{\bZ}{\mathbb{Z}}

\newcommand{\Q}{\mathbb{Q}}

\newcommand{\R}{\mathbb{R}}

\newcommand{\cT}{\mathscr{T}}
\newcommand{\cX}{\mathcal{X}}
\newcommand{\cY}{\mathcal{Y}}
\newcommand{\cZ}{\mathcal{Z}}
\newcommand{\cI}{\mathcal{I}}
\newcommand{\cA}{\mathcal{A}}
\newcommand{\cC}{\mathcal{C}}
\newcommand{\cD}{\mathcal{D}}
\newcommand{\cU}{\mathcal{U}}
\newcommand{\cO}{\mathcal{O}}
\newcommand{\cV}{\mathcal{V}}
\newcommand{\cF}{\mathcal{F}}
\newcommand{\cG}{\mathcal{G}}
\newcommand{\cH}{\mathcal{H}}
\newcommand{\cJ}{\mathscr{J}}
\newcommand{\mcJ}{\mathcal{J}}
\newcommand{\cK}{\mathcal{K}}
\newcommand{\cE}{\mathcal{E}}
\newcommand{\cP}{\mathcal{P}}

\newcommand{\cL}{\mathcal{L}}

\newcommand{\cR}{\mathcal{R}}

\newcommand{\sL}{\mathscr{L}}
\newcommand{\sM}{\mathscr{M}}
\newcommand{\sI}{\mathscr{I}}
\newcommand{\sJ}{\mathscr{J}}

\newcommand{\bL}{\mathbb{L}}
\newcommand{\bH}{\mathbb{H}}
\newcommand{\bG}{\mathbb{G}}
\newcommand{\bA}{\mathbb{A}}

\newcommand{\Var}{\mathbf{Var}}

\newcommand{\diff}{\mathrm{d}}
\newcommand{\red}{\mathrm{red}}
\newcommand{\id}{\mathrm{id}}
\newcommand{\Gor}{\mathrm{Gor}}

\newcommand{\str}{\mathrm{str}}

\DeclareMathOperator{\rep}{rep}
\DeclareMathOperator{\synt}{synt}

\DeclareMathOperator{\Res}{Res}

\DeclareMathOperator{\st}{st}
\DeclareMathOperator{\het}{ht}

\DeclareMathOperator{\e}{e}
\DeclareMathOperator{\Spec}{Spec}
\DeclareMathOperator{\Ext}{Ext}
\DeclareMathOperator{\ord}{ord}
\DeclareMathOperator{\Hom}{Hom}
\DeclareMathOperator{\Isom}{Isom}
\DeclareMathOperator{\Aut}{Aut}
\DeclareMathOperator{\GL}{GL}
\DeclareMathOperator{\len}{len}
\DeclareMathOperator{\tors}{tors}

\DeclareMathOperator{\coh}{coh}

\DeclareMathOperator{\coker}{coker}

\DeclareMathOperator{\rank}{rank}

\DeclareMathOperator{\uHom}{\underline{\Hom}}
\DeclareMathOperator{\uIsom}{\underline{\Isom}}
\DeclareMathOperator{\uAut}{\underline{\Aut}}

\DeclareMathOperator{\HD}{HD}

\newcommand{\sW}{\mathscr{W}}

\newcommand{\sHom}{\mathscr{H}om}

\DeclareMathOperator{\wt}{wt}

\tikzset{cong/.style={draw=none,edge node={node [sloped, allow upside down, auto=false]{$\cong$}}},
         Isom/.style={above,every to/.append style={edge node={node [sloped, allow upside down, auto=false]{$\sim$}}}}}

\title{Stringy Hodge numbers via crepant resolutions by Artin stacks}

\author{Matthew Satriano and Jeremy Usatine}

\thanks{MS was partially supported by a Discovery Grant from the National Science and Engineering Research Council of Canada as well as a Mathematics Faculty Research Chair from the University of Waterloo}

\address{Matthew Satriano, Department of Pure Mathematics, University of Waterloo}
\email{msatriano@uwaterloo.ca}

\address{Jeremy Usatine, Department of Mathematics, Florida State University}
\email{jusatine@fsu.edu}

\begin{document}

\begin{abstract}
In a previous paper we showed that any variety with log-terminal singularities admits a crepant resolution by a smooth Artin stack. In this paper we prove the converse, thereby proving that a variety admits a crepant resolution by a smooth Artin stack if and only if it has log-terminal singularities. Furthermore if $\cX \to Y$ is such a resolution, we obtain a formula for the stringy Hodge numbers of $Y$ in terms of (motivically) integrating an explicit weight function over twisted arcs of $\cX$. That weight function takes only finitely many values, so we believe this result provides a plausible avenue for finding a long-sought cohomological interpretation for stringy Hodge numbers. Using that the resulting integral is defined intrinsically in terms of $\cX$, we also obtain a notion of stringy Hodge numbers for smooth Artin stacks, that in particular, recovers Chen and Ruan's notion of orbifold Hodge numbers.
\end{abstract}

\maketitle

\setcounter{tocdepth}{1}

\tableofcontents

\numberwithin{theorem}{section}
\numberwithin{lemma}{section}
\numberwithin{corollary}{section}
\numberwithin{proposition}{section}
\numberwithin{conjecture}{section}
\numberwithin{question}{section}
\numberwithin{remark}{section}
\numberwithin{definition}{section}
\numberwithin{example}{section}
\numberwithin{notation}{section}

\section{Introduction}

Let $V$ be a complex projective variety. It is natural to ask if there is a nice notion of Hodge numbers for $V$ even in the case where $V$ has singularities. Although we will not attempt to summarize the various successful approaches to this question that have appeared throughout the literature, we note that one naive first attempt at such a question would be to take a resolution of singularities $X \to V$ and consider the Hodge numbers of $X$. Of course, this recipe depends on the choice of resolution $X \to V$ and therefore fails to provide a well defined notion for $V$. Therefore one might hope to modify this recipe by considering only a special class of resolutions. From the perspective of mirror symmetry, it is natural to consider only \emph{crepant} resolutions $X \to V$, i.e., those whose relative canonical divisor $K_{X/V}$ is trivial. This approach has many favorable features. For example if $X \to V$ and $X' \to V$ are both crepant resolutons, then $X$ and $X'$ have the same Hodge numbers\footnote{Deep questions remain about which other invariants $X$ and $X'$ share. For example, entire swaths of the literature have been inspired by a famous open conjecture of Bondal and Orlov, which predicts that $X$ and $X'$ are derived equivalent.} \cite{Kontsevich} (or see e.g., \cite[Theorem 3.4]{Batyrev1998}). Unfortunately, this approach also suffers from a major weakness: the condition that $V$ admit a crepant resolution is a rather restrictive one. One remedy for this weakness is to allow more general resolutions $X \to V$, but instead of considering the Hodge numbers of $X$, one instead considers a recipe depending on the pair $(X, K_{X/V})$. In \cite{Batyrev1998} Batyrev used this approach to define the \emph{stringy Hodge numbers} of $V$ when $V$ has log-terminal singularities. Log-terminal singularities are a broad class that, for example, arise in the minimal model program, and Batyrev's stringy Hodge numbers have seen many beautiful applications across a diverse range of areas including mirror symmetry and the representation theory of finite groups. But as successful as this approach has been, it produces a downside. Because of the way in which $K_{X/V}$ is cooked into the definition, the stringy Hodge numbers of $V$ have no known interpretation as dimensions of cohomology groups. For example, this is why the following well-known conjecture of Batyrev's remains wide-open.

\begin{conjecture}[{\cite[Conjecture 3.10]{Batyrev1998}}]\label{conjectureBatyrevNonnegativity}
The stringy Hodge numbers of a projective variety are nonnegative.
\end{conjecture}

From the perspective of e.g., \autoref{conjectureBatyrevNonnegativity}, it would be useful to rescue the approach of studying stringy Hodge numbers via crepant resolutions. More precisely, we could ask which log-terminal varieties admit crepant resolutions if we are willing to also consider resolutions by stacks. In recent work, we showed that the answer to this question is \emph{all of them}.

\begin{theorem}[{\cite[Theorem 1.1\footnote{Although \cite[Theorem 1.1]{SatrianoUsatine3} is stated with the assumption that $Y$ is a variety, i.e., separated, the proof never uses that assumption.}]{SatrianoUsatine3}}]\label{citedTheoremLogTerminalImpliesCrepantExists}
Let $k$ be an algebraically closed field of characteristic 0, and let $Y$ be an irreducible finite type scheme over $k$ with log-terminal singularities. Then there exists a smooth irreducible finite type Artin stack $\cX$ over $k$ and a crepant resolution $\cX \to Y$ with affine diagonal.
\end{theorem}

In light of this existence result, we turn our focus in this paper to the task of \emph{applying} such a stacky resolution. In particular we show the converse to \autoref{citedTheoremLogTerminalImpliesCrepantExists}, thereby obtaining a new criterion for varieties to have log-terminal singularities. We also obtain a new formula for stringy Hodge numbers in terms of these stacky resolutions. To our knowledge, this formula is the first known description for stringy Hodge numbers in terms of motivically integrating a function \emph{that takes only finitely many values}. As we will discuss below, we believe such an integral provides a plausible avenue for finding a long-sought cohomological interpretation for stringy Hodge numbers.

\subsection{Statement of main results}

Throughout this paper, $k$ is an algebraically closed field of characteristic 0. 

We begin by recalling our definition of a crepant resolution $\cX \to Y$ by an Artin stack $\cX$. In the special case where the stack $\cX$ is a scheme, the following definition coincides with the usual definition of crepant resolution.

\begin{definition}
Let $Y$ be a $\Q$-Gorenstein irreducible finite type scheme over $k$, let $\cX$ be a smooth irreducible finite type Artin stack over $k$, and let $\cX \to Y$ be a morphism over $k$. We call $\cX \to Y$ \emph{weakly birational} if there exists a dense open substack $\cU$ of $\cX$ such that the composition $\cU \hookrightarrow \cX \to Y$ is an open immersion. In that case we may consider the \emph{relative canonical divisor} $K_{\cX/Y}$ as defined in \cite[Definition 1.6]{SatrianoUsatine3}, which in particular is a $\Q$-Cartier divisor on $\cX$. We call $\cX \to Y$ \emph{crepant} if it is weakly birational and $K_{\cX/Y} = 0$. We call $\cX \to Y$ \emph{strongly birational} if there exists a dense open subscheme $V$ of $Y$ such that $V \times_Y \cX \to V$ is an isomorphism. We call $\cX \to Y$ a \emph{resolution} if it is strongly birational and factors as a good moduli space morphism followed by a tame\footnote{Although the definition in \cite{SatrianoUsatine3} did not include this tameness hypothesis, the crepant resolutions constructed there satisfy this constraint, see e.g., \cite[Theorem 8.6]{SatrianoUsatine3}. The tameness assumption makes the definition equivalent to $\cX$ admitting a good moduli space $X$ and the induced map $X \to Y$ being proper.} proper morphism.
\end{definition}

Our first main result is the converse to \autoref{citedTheoremLogTerminalImpliesCrepantExists}. Together we obtain the following new characterization of log-terminal singularities.

\begin{theorem}\label{mainTheoremCrepantResolutionImpliesLogTerminal}
Let $Y$ be a $\Q$-Gorenstein irreducible finite type scheme over $k$. Then the following are equivalent:
\begin{enumerate}
\item There exists a smooth irreducible finite type Artin stack $\cX$ over $k$ and a crepant resolution $\cX \to Y$.
\item $Y$ has log-terminal singularities.
\end{enumerate}
\end{theorem}

Our next main result provides a formula for stringy Hodge numbers in terms of crepant resolutions by smooth Artin stacks. In what follows, $\mu^\Gor_Y$ denotes the \emph{Gorenstein measure} on the arc scheme $\sL(Y)$ as introduced by Denef and Loeser \cite{DenefLoeser2002}, and $\nu_\cX$ denotes the motivic measure on the stack $\sJ(\cX)$ of twisted arcs of $\cX$. See \autoref{sectionMotivicIntegrationTwistedArcs} below for the definitions of $\sJ(\cX)$, $\nu_\cX$, $\wt_\cX$, and the associated integrals.

\begin{notation}
If $\cX$ is a finite type Artin stack over $k$ with affine diagonal and $\cX$ has a stable good moduli space, we define a subset $\cA_\cX \subset |\sJ(\cX)|$ as follows. Let $\cX' \to \cX$ be the canonical reduction of stabilizers of Edidin--Rydh \cite[Theorem 2.11]{EdidinRydh}. The representable map $\cX' \to \cX$ induces in the obvious way a map $\sJ(\cX') \to \sJ(\cX)$, and we let $\cA_\cX$ denote the image of the associated map on topological spaces $|\sJ(\cX')| \to |\sJ(\cX)|$.
\end{notation}

\begin{theorem}\label{maintheorem}
Let $Y$ be a $\Q$-Gorenstein irreducible finite type scheme over $k$, let $\cX$ be a smooth irreducible finite type Artin stack over $k$, let $\cX \to Y$ be a crepant resolution, and assume that $\cX$ has affine diagonal. Then $\bL^{-\wt_\cX}$ is integrable on $\cA_\cX$, and
\[
	\int_{\sL(Y)}\diff\mu^\Gor_Y = \int_{\cA_\cX} \bL^{-\wt_\cX}\diff\nu_\cX.
\]
Furthermore, the function $\wt_\cX$ takes only finitely many values on $\cA_\cX$.
\end{theorem}

\begin{remark}
Note that in \autoref{maintheorem}, $\mu^\Gor_Y$ is well defined by \autoref{mainTheoremCrepantResolutionImpliesLogTerminal}, and $\cA_\cX$ is well defined by \autoref{propositionCrepantResolutionImpliesStableGMS} below. The equality in \autoref{maintheorem} occurs in the modified Grothendieck ring $\widehat{\sM}_k[\bL^{1/m}]$ where $m$ is such that the $m$th canonical sheaf $\omega_{Y,m}$ is invertible. The definition of $\widehat{\sM}_k$ is recalled in \autoref{sectionPreliminaries}.
\end{remark}

A key feature of the Gorenstein measure $\mu_Y^\Gor$ is that it specializes to the stringy Hodge numbers of $Y$ in the following way. The notion of Hodge--Deligne polynomial induces a ring morphism $\HD: \widehat{\sM}_k[\bL^{1/m}] \to \Q\llbracket u^{-1/m}, v^{-1/m} \rrbracket [u,v]$, and \cite[Section 5]{DenefLoeser2002} (or see \cite[Chapter 7 Section 3.4]{ChambertLoirNicaiseSebag}) implies that the \emph{stringy Hodge--Deligne invariant}\footnote{also often called the \emph{stringy E-function}} $\HD_{\str}(Y)$ of $Y$ is related to $\mu_Y^\Gor$ via
\[
	\HD_{\str}(Y) = \HD\left(\bL^{\dim Y}\int_{\sL(Y)}\diff\mu^\Gor_Y\right).
\]
If $\HD_{\str}(Y) \in \Z\llbracket u^{-1}, v^{-1}\rrbracket[u,v]$, which for example always occurs if $Y$ is Gorenstein, we set $h^{p,q}_{\str}(Y) \in \Z$ to be such that
\[
	\HD_{\str}(Y) = \sum_{p,q} (-1)^{p+q} h^{p,q}_{\str}(Y) u^p v^q.
\]
If $\HD_{\str}(Y) \in \Z[u,v]$, then the stringy Hodge numbers of $Y$ are defined to be the numbers $h_{\str}^{p,q}(Y)$.

\begin{remark}
Although Batyrev defined the power series $\HD_{\str}(Y)$ for all $Y$ with log-terminal singularities, he only called the (sign-corrected) coefficients stringy Hodge numbers in the case where $\HD_{\str}(Y) \in \Z[u,v]$, and it is under that assumption that he formulated \autoref{conjectureBatyrevNonnegativity}. In fact, \cite[Remark 3.3(i)]{SchepersVeys} gives an example of a projective variety $Y$  with Gorenstein canonical singularities where not all $h_{\str}^{p,q}(Y)$ are nonnegative. This example does not contradict \autoref{conjectureBatyrevNonnegativity} because it does not satisfy $\HD_{\str}(Y) \in \Z[u,v]$, so one is not supposed to refer to the $h^{p,q}_{\st}(Y)$ as stringy Hodge numbers in this case.
\end{remark}

Therefore \autoref{maintheorem} immediately implies the following new formula for stringy Hodge numbers, which we emphasize can be applied to any log-terminal variety $Y$ by \autoref{citedTheoremLogTerminalImpliesCrepantExists}. We also emphasize that the right hand side of the following equality depends only on the stack $\cX$ and not on $Y$ or the map $\cX \to Y$.

\begin{corollary}\label{corollaryStringyHodgeDeligne}
Let $Y$ be a $\Q$-Gorenstein irreducible finite type scheme over $k$, let $\cX$ be a smooth irreducible finite type Artin stack over $k$, let $\cX \to Y$ be a crepant resolution, and assume that $\cX$ has affine diagonal. Then $\bL^{-\wt_\cX}$ is integrable on $\cA_\cX$, and
\[
	\HD_{\str}(Y) = \HD\left( \bL^{\dim\cX} \int_{\cA_\cX} \bL^{-\wt_\cX}\diff\nu_\cX \right).
\]
\end{corollary}

\subsection{Dreaming of a cohomology theory}

\autoref{corollaryStringyHodgeDeligne} suggests we should make the following definition.

\begin{definition}\label{definitionStringyHodgeDeligneStack}
Let $\cX$ be a smooth irreducible finite type Artin stack over $k$ with affine diagonal and stable good moduli space, and assume that $\bL^{-\wt_\cX}$ is integrable on $\cA_\cX$. Then we define the \emph{stringy Hodge--Deligne invariant} of $\cX$ to be
\[
	\HD_{\str}(\cX) = \HD\left( \bL^{\dim\cX} \int_{\cA_\cX} \bL^{-\wt_\cX}\diff\nu_\cX \right).
\]
\end{definition}

In the special case where $\cX$ is Deligne--Mumford, the orbifold cohomology of Chen and Ruan \cite{ChenRuan} provides a cohomological interpretation for $\HD_{\str}(\cX)$. It is natural to ask for a cohomological interpretation for $\HD_{\str}(\cX)$ in the level of generality of \autoref{definitionStringyHodgeDeligneStack}, as by \autoref{citedTheoremLogTerminalImpliesCrepantExists} and \autoref{corollaryStringyHodgeDeligne}, this would provide a cohomological interpretation for stringy Hodge numbers of all log-terminal varieties. We briefly discuss, in a quite speculative manner, why we believe such a goal is plausible.

One reason is that, unlike any previously appearing formula for $\HD_{\str}(Y)$ of which we are aware, \autoref{corollaryStringyHodgeDeligne} expresses $\HD_{\str}(\cX)$ and therefore $\HD_{\str}(Y)$ in terms of integrating a function taking only finitely many values. In other words the right hand side of the defining equation
\[
	\int_{\cA_\cX} \bL^{-\wt_\cX}\diff\nu_\cX = \sum_{w \in \Z}\bL^{-w} \nu_{\cX}(\cA_\cX \cap \wt_\cX^{-1}(w))
\]
is a finite sum. Therefore to find a cohomological interpretation for stringy Hodge numbers, it would be sufficient to find a cohomological interpretation for each
\[
	\HD\left(\nu_{\cX}(\cA_\cX \cap \wt_\cX^{-1}(w))\right).
\]
If $\cA_\cX \cap \wt_\cX^{-1}(w)$ were a cylinder in the sense of \autoref{definitionCylinderForUnion}, Laszlo and Olsson's compactly supported cohomology of Artin stacks \cite{LaszloOlsson} would allow us to build the desired cohomological interpretation due to \cite[Theorem 2.5]{Ekedahl} and the nature of our definition of the motivic measure $\nu_\cX$. In fact this holds when $\cX$ is Deligne--Mumford, in which case we recover orbifold cohomology. But more generally, $\cA_\cX \cap \wt_\cX^{-1}(w)$ is not always a cylinder, as can be checked for example when $\cX = [\bA^2_k / \bG_m]$ with $\bG_m$ acting on $\bA^2_k$ with weights 1 and -1. Yet we still expect that $\cA_\cX \cap \wt_\cX^{-1}(w)$ is better behaved than an arbitrary measurable set: since $\cA_\cX$ is the image of $|\sJ(\cX')|$ where $\cX' \to \cX$ is the canonical reduction of stabilizers of Edidin-Rydh, we expect $\cA_\cX$ behaves like a semi-algebraic set. Indeed for a map of varieties $X' \to X$, the image of $\sL(X')$ is a semi-algebraic set in $\sL(X)$ by Pas's quantifier elimination (see the discussion immediately preceeding Proposition 2.3 in \cite{DenefLoeser1999}). Following this line of thinking, one could hope that $\cA_\cX \cap \wt_\cX^{-1}(w)$ corresponds to something like a non-archimedean analytic stack over $k\llparenthesis t \rrparenthesis$. In the case of schemes and at the level of Euler characteristics, one can compare motivic integration and cohomology of non-archimedean analytic spaces \cite[Section 3.9]{NicaisePayneSchroeter}. Exploring this circle of ideas in order to find a cohomological interpretation for $\HD_{\str}(\cX)$ and thus $\HD_{\str}(Y)$ is the subject of ongoing work.

\subsection{Methods and the warping stack}

Most of this paper is devoted to proving two intermediate results, which themselves are the key inputs to proving \autoref{mainTheoremCrepantResolutionImpliesLogTerminal} and \autoref{maintheorem}. One is  \autoref{corollaryTwistedChangeOfVariables}, which is a general change of variables formula in the setting of twisted arcs of Artin stacks. The other is \autoref{theoremThinSubsetOfTwistedJets}, which shows that the twisted arcs that factor through a fixed closed substack of $\cX$ form a measurable set of measure 0. These two results should be of independent interest, as they equip the twisted/Artin setting with the same fundamental toolkit that makes motivic integration useful in the scheme setting.

Our proofs of \autoref{corollaryTwistedChangeOfVariables} and \autoref{theoremThinSubsetOfTwistedJets} rely heavily on the warping stack $\sW(\cX)$ that we introduced in \cite{SatrianoUsatine4}. An essential observation is that twisted arcs of $\cX$ can also be considered as untwisted arcs of $\sW(\cX)$. Thus following the bootstrapping philosophy proposed in \cite{SatrianoUsatine4}, the warping stack $\sW(\cX)$ allows us to reduce our study of twisted arcs to the case of untwisted arcs. More specifically \autoref{theoremTwistedToWarped} provides a comparison between the motivic measure on twisted arcs of $\cX$ and the motivic measure on untwisted arcs of $\sW(\cX)$, so as a result of the untwisted change of variables formula proved in \cite{SatrianoUsatine2}, we immediately obtain a change of variables formula in the twisted setting as well. In combination with \autoref{theoremHeightWarpingAndHeightWeight}, which allows us to rewrite the correction term in that change of variables formula, we obtain the desired twisted change of variables formula \autoref{corollaryTwistedChangeOfVariables}. Similarly \autoref{theoremTwistedToWarped} also reduces \autoref{theoremThinSubsetOfTwistedJets} to the untwisted case, which we prove in \autoref{theoremThinSubsetsAreNegligible}.

\begin{notation}
Throughout this paper, if $\cX$ is a stack we will let $\cX(A)$ denote its category of $A$-valued points, we will let $\overline{\cX}(A)$ denote the set of isomorphism classes of this category, and we will let $|\cX|$ denote the topological space associated to $\cX$. Furthermore if $\cC \subset |\cX|$ and $k'$ is a field, we will let $\cC(k')$ denote the category of $k'$-valued points of $\cX$ whose class in $|\cX|$ is contained in $\cC$, and we will let $\overline{\cC}(k')$ denote the set of isomorphism classes of the category $\cC(k')$.
\end{notation}

\begin{acknowledgements}
This paper benefited greatly from conversations with Piotr Achinger, Dori Bejleri, Jason Bell, Bhargav Bhatt, Ron Cherny, Dan Edidin, Changho Han, Elana Kalashnikov, Johannes Nicaise, Wim Veys, and Dimitri Wyss. We thank Aji Dhillon and Brett Nasserden for conversations that led to the proof of \autoref{prop:auts-Dnell}. We thank Takehiko Yasuda for pointing out \cite[Remark 15.6]{Yasuda2024}.
\end{acknowledgements}

\section{Preliminaries}\label{sectionPreliminaries}

We begin by setting some notation that will be used throughout this paper. If $\cX$ is an Artin stack over $k$, we will let $\sL_n(\cX)$ denote its $n$th jet stack, we will let $\sL(\cX)$ denote its arc stack, and we will let $\theta_n: \sL(\cX) \to \sL_n(\cX)$ and $\theta^n_m: \sL_n(\cX) \to \sL_m(\cX)$ denote the associated truncation morphisms. We refer to \cite[Section 3]{SatrianoUsatine} for basic properties of these notions. For any ring $A$ we will let $D_A = \Spec(A\llbracket t \rrbracket)$, so in particular, for any field extension $k'$ of $k$ the category $\sL(\cX)(k')$ is the category $\cX(D_{k'})$ of ($k'$-valued) arcs of $\cX$. If it is clear from context, we may suppress the ring $A$ and use only the notation $D$.

\begin{notation}
Let $\cX$ be a locally finite type Artin stack over $k$. If $\cI$ is a quasi-coherent ideal sheaf on $\cX$, then we let $\ord_{\cI}: |\sL(\cX)| \to \Z \cup \{\infty\}$ denote its associated order function, see e.g., \cite[Remark 2.6]{SatrianoUsatine2}. This order function evaluated at an arc of $\cX$ is the $t$-vanishing order of the pullback of $\cI$ along that arc. Similarly if $\cD$ is a Cartier divisor on $\cX$, we let $\ord_{\cD}: |\sL(\cX)| \to \Z \cup \{\infty\}$ denote its associated order function, see e.g., \cite[Paragraph after Remark 1.9]{SatrianoUsatine3}. Extending linearly, we can also consider $\ord_{\cD}: |\sL(\cX)| \to \Q \cup \{\infty\}$ when $\cD$ is a $\Q$-Cartier $\Q$-divisor on $\cX$. If $\cX \to \cY$ is a morphism of locally finite type Artin stacks, we let $\het_{\cX/\cY}: |\sL(\cX)| \to \Z \cup \{\infty\}$ denote the height function defined in \cite[Definition 6.2]{SatrianoUsatine5}. In particular if $\cX$ is smooth over $k$,
\[
	\het_{\cX/Y} = \het^{(0)}_{L_{\cX/Y}} -  \het^{(1)}_{L_{\cX/Y}}
\]
where the height functions on the right hand side are were defined in \cite[Definition 3.1]{SatrianoUsatine2} as follows:~for any $E\in D^-_{\coh}(\cX)$ any arc $\varphi: D_{k'}\to \cX$, we let
\[
	\het_E^{(i)}(\varphi) = \dim_{k'} \bH^i (L\varphi^*E).
\]
\end{notation}

Let $K_0(\Var_k)$ denote the Grothendieck ring of $k$-varieties, and let ${\widehat{\sM}}_k$ denote the ring obtained by inverting the class of $\bA^1_k$ in $K_0(\Var_k)$ then completing with respect to the dimension filtration. If $\cY$ is a finite type Artin stack over $k$ with affine geometric stabilizers, we will let $\e(\cY)$ denote the class of $\cY$ in $\widehat{\sM}_k$, and more generally if $\cC$ is a constructible subset of $|\cY|$, we will let $\e(\cC)$ denote the class of $\cC$ in $\widehat{\sM}_k$. We refer to \cite[Subsection 2.2]{SatrianoUsatine} for details on the definition of these classes. We also set $\bL = \e(\bA^1_k)$.

A major tool used throughout this paper is the warping stack $\sW(\cX)$ introduced in \cite{SatrianoUsatine4}, which we now recall. A \emph{warp} of a scheme $T$ is a flat finitely presented good moduli space map $\cT \to T$ with affine diagonal, and a \emph{warped map} from $T$ to a stack $\cX$ is a pair $(\cT \to T, \cT \to \cX)$ consisting of a warp $\cT \to T$ and a representable map $\cT \to \cX$. The \emph{warping stack} $\sW(\cX)$ of a stack $\cX$ over $k$ is a category fibered in groupoids (over the category of $k$-schemes) such that for each $k$-scheme $T$, $\sW(\cX)(T)$ is the category of warped maps from $T$ to $\cX$. See \cite[Remark 1.4]{SatrianoUsatine4} for details on defining $\sW(\cX)$ as a fibered category. If $\cX$ is a locally finite type Artin stack with affine diagonal over $k$, then \cite[Theorem 1.9(1)]{SatrianoUsatine4} implies that $\sW(\cX)$ is a locally finite type Artin stack over $k$. Furthermore in this case, the map $\tau: \cX \to \sW(\cX)$ defined by $(T \to \cX) \mapsto (T \to T, T \to \cX)$ is an open immersion by \cite[Theorem 1.9(2)]{SatrianoUsatine4}, and throughout this paper we identify $\cX$ with the corresponding open substack of $\sW(\cX)$.

The remainder of this section is dedicated to miscellaneous results that will be used later in the paper.

\begin{lemma}\label{lemmaPullBackCotangentComplexBecomesTwoVectorBundles}
Let $\cX$ be a smooth equidimensional finite type Artin stack over $k$ with affine diagonal, and assume that $\cX$ has a good moduli space. Let $k'$ be a field extension of $k$, and let $\cT$ be an Artin stack over $k$ with good moduli space isomorphic to either $\Spec(k'\llbracket t \rrbracket)$ or $\Spec(k')$. Let $\varphi: \cT \to \cX$ be a morphism over $k$. Then there exist $r, s \in \Z_{\geq 0}$ with $r-s = \dim\cX$ and an exact triangle
\[
	L\varphi^*L_{\cX} \to \cF \to \cG
\]
such that $\cF$ is a rank $r$ locally free sheaf on $\cT$ and $\cG$ is a rank $s$ locally free sheaf on $\cT$.
\end{lemma}

\begin{proof}
Let $X$ be the good moduli space of $\cX$. By \cite[Theorem 6.1(1)]{AHRetalelocal} there exists an algebraic space $X'$ and a Nisnevich covering $X' \to X$ such that $\cX' = \cX \times_X X' \cong [U/\GL_n]$ for some affine scheme $U$ over $k$. Let $\pi: \cX' \to \cX$ denote the projection. Let $T$ be the good moduli space of $\cT$. The map $\varphi: \cT \to \cX$ induces a map $T \to X$. Since $X' \to X$ is a Nisnevich cover, $T \to X$ factors as some $T \to X'$ followed by $X' \to X$ (in the case where $T \cong \Spec(k')$ this is immediate and when $T \cong \Spec(k'\llbracket t \rrbracket)$ we lift the special point then use infinitesimal lifting for the \'{e}tale map $X' \to X$ to get the desired map $T \to X'$). Therefore $\varphi: \cT \to \cX$ factors as some $\psi: \cT \to \cX'$ followed by the map $\pi: \cX' \to \cX$. Since $\cX' \to \cX$ is \'{e}tale, $\cX'$ is equidimensional and smooth over $k$, $\dim\cX' = \dim\cX$, and $L_{\cX'} \cong L\pi^*L_{\cX}$. Thus
\[
	L\varphi^*L_\cX \cong L\psi^* L_{\cX'}.
\]
Set $r = \dim U$ and $s = \dim\GL_n$. We have that $U$ is equidimensional and smooth and $r - s = \dim\cX' = \dim\cX$. Letting $\rho\colon U\to\cX'$ be the smooth cover, we have an exact triangle
\[
\rho^*L_{\cX'} \to\Omega^1_U\xrightarrow{\alpha}\Omega^1_\rho.
\]
Since $\alpha$ is an equivariant map of vector bundles with $\GL_n$-action, it descends to an exact triangle
\[
	L_{\cX'} \to \cF' \to \cG'
\]
where $\cF'$ is a rank $r$ locally free sheaf and $\cG'$ is a rank $s$ locally free sheaf. The lemma follows by applying $L\psi^*$ to this exact triangle.
\end{proof}

\begin{lemma}\label{lemmaSpreadingOutGenericFiber}
Let $S$ be a noetherian scheme, let $Z$ be a finite type scheme over $S$, let $\cY$ and $\cF$ be finite type Artin stacks over $S$, let $z \in Z$ be the generic point of an irreducible component of $Z$, and let $k(z)$ be the residue field of $z$. If $\cY \to Z$ is a morphism such that
\[
	(\cY \times_Z \Spec(k(z))_\red \cong (\cF \times_S \Spec(k(z)))_\red
\]
as stacks over $k(z)$, then there exists an open neighborhood $U \subset Z$ of $z$ such that
\[
	(\cY \times_Z U)_\red \cong (\cF \times_S U)_\red
\]
as stacks over $S$.
\end{lemma}

\begin{proof}
This exact statement is verified in the proof of \cite[Proposition 2.8]{SatrianoUsatine}.
\end{proof}

\begin{proposition}\label{propFibrationClass}
Let $\cY$ and $\cZ$ be finite type stacks over $k$ with affine geometric stabilizers, let $\cY \to \cZ$ be a morphism over $k$, and let $\Theta \in \widehat{\sM}_k$. Suppose that for every field extension $k'$ of $k$ and $k$-morphism $\Spec(k') \to \cZ$ there exists a finite type stack $\cF$ over $k$ with affine geometric stabilizers such that $\e(\cF) = \Theta$ and
\[
	(\cY \times_\cZ \Spec(k'))_\red \cong (\cF \times_{\Spec(k)} \Spec(k'))_\red
\]
as stacks over $k'$. Then
\[
	\e(\cY) = \Theta \e(\cZ).
\]
\end{proposition}

\begin{proof}
By \cite[Proposition 3.5.9]{Kresch}, $\cZ$ can be stratified by stacks of the form $[W / G]$ where $W$ is a finite type scheme over $k$ and $G$ is a general linear group. Thus we may assume that $\cZ$ is of the form $[W/G]$. Set $\cV = \cY \times_\cZ W$. Then \cite[Proposition 1.1(ii)]{Ekedahl} implies
\[
	\e(W) = \e(G)\e(\cZ) \qquad \text{and} \qquad \e(\cV) = \e(G)\e(\cY),
\]
and \cite[Proposition 1.1(i)]{Ekedahl} implies $\e(G)$ is a unit in $\widehat{\sM}_k$. Therefore we are reduced to the case where $\cZ$ is a scheme, but that case follows immediately from noetherian induction and \autoref{lemmaSpreadingOutGenericFiber}.
\end{proof}

Lastly, we give the following result akin to the infinitesimal lifting criterion.

\begin{proposition}[{Infinitesimal lifting criterion with cohomologically affine stacks}]
\label{prop:coh-affine-smooth}
Consider a commutative diagram of solid arrows and morphisms of Artin stacks 
\[
\xymatrix{
\cZ\ar[r]^-{f}\ar@{^{(}->}[d]_-{\iota} & \cY\ar[d]\\
\cZ'\ar[r]\ar@{-->}[ur]^-{f'} & \Spec k
}
\]
with $f$ representable. If $\cZ$ cohomologically affine, $\iota$ is a square-zero thickening, and $\cY$ is smooth over $k$, then there exists a representable map $f'$ making the diagram commute.
\end{proposition}
\begin{proof}
Upon showing $f'$ exists, it is necessarily representable by \cite[Tag 0CJ9]{stacks-project}. Let $\cI$ be the ideal sheaf defining $\iota$. By \cite[Theorem 1.5]{OlssonDefRep} (taking $X=\cZ$, $X'=\cZ'$, $Y=Y'=\cY$, and $Z=Z'=\Spec k$) we see the obstruction to the existence of the dotted arrow lives in $\Ext^1(Lf^*L_\cY,\cI)$. Since $\cZ$ is cohomologically affine, we see $\Ext^1(Lf^*L_\cY,\cI)=H^0\mathcal{E}xt^1(Lf^*L_\cY,\cI)$. Let $\rho\colon Y\to\cY$ be a smooth cover and $p\colon Z\to\cZ$ (resp.~$g\colon Z\to Y$) be the pullback of $\rho$ (resp.~$f$). It suffices to show
\[
p^*\mathcal{E}xt^1(Lf^*L_\cY,\cI)=\mathcal{E}xt^1(p^*Lf^*L_\cY,p^*\cI)=\mathcal{E}xt^1(Lg^*\rho^*L_\cY,p^*\cI)
\]
is the zero sheaf. From the exact triangle
\[
Lg^*\rho^*L_Y\to g^*\Omega^1_Y\to g^*\Omega^1_\rho
\]
we see $\mathcal{E}xt^1(Lg^*\rho^*L_\cY,p^*\cI)$ vanishes.
\end{proof}

\section{Coherent sheaves on twisted discs}

Throughout this paper for any $\ell \in \Z_{>0}$ we set
\[
	\cD^\ell = [\Spec(k\llbracket t^{1/\ell} \rrbracket) / \mu_\ell],
\] 
where $\xi \in \mu_\ell$ acts on $\Spec(k\llbracket t^{1/\ell} \rrbracket)$ by $f(t^{1/\ell}) \mapsto f(\xi t^{1/\ell})$. Then $\cD^\ell$ admits a tame proper quasi-finite flat finitely presented good moduli space map $\cD^\ell \to \Spec(k\llbracket t \rrbracket)$ where the composition $\Spec(k\llbracket t^{1/\ell}\rrbracket) \to \cD^\ell \to \Spec(k\llbracket t \rrbracket)$ is given by the usual inclusion $k\llbracket t \rrbracket \to k \llbracket t^{1/\ell} \rrbracket$. Note that the map
\[
	\cD^\ell \otimes_{k\llbracket t \rrbracket} k\llparenthesis t \rrparenthesis \to \Spec(k\llparenthesis t \rrparenthesis)
\]
is an isomorphism, so we may refer to the \emph{generic point} of $\cD^\ell$. For any $n \in \Z_{\geq 0}$ we set
\[
	\cD^\ell_n = \cD^\ell \otimes_{k\llbracket t \rrbracket} k[t]/(t^{n+1}) = [\Spec(k[t^{1/\ell}]/(t^{n+1})) / \mu_\ell].
\]
For any algebra $A$ over $k$, we set
\[
	\cD^\ell_{A} = \cD^\ell \otimes_{k\llbracket t \rrbracket} A\llbracket t \rrbracket, \quad \text{and} \quad \cD^\ell_{n, A} = \cD^\ell_n \otimes_{k[t]/(t^{n+1})} A[t]/(t^{n+1}).
\]

Note that coherent sheaves on $\cD^\ell$ are the same as $\Z/\ell$-graded modules over $k\llbracket t^{1/\ell} \rrbracket$, where $t^{\ell}$ is in degree 1. Therefore the following graded version of the structure theorem for finitely generated modules and Smith normal form will be useful in working with coherent sheaves on $\cD^\ell$.

\begin{proposition}[{Graded Structure Theorem and Smith Normal Form}]\label{prop:graded-Smith}
Let $\kappa$ be a field and $R=\kappa\llbracket t\rrbracket$ be $\bZ/\ell$-graded with $t$ of degree $1$.
\begin{enumerate}[label=(\roman*)]
\item\label{prop:graded-Smith::str-thm} If $M$ is a finitely generated graded $R$-module, then we have a graded isomorphism
\[
M\simeq\bigoplus_i R(a_i)\oplus\bigoplus_j R/t^{n_j}(b_j).
\]
\item\label{prop:graded-Smith::smith} If $\alpha\colon F'\to F$ is a graded morphism of finitely generated graded free $R$-module, then we have a commutative diagram
\[
\xymatrix{
F'\ar[r]^-{\alpha}\ar[d]_-{\simeq} & F\ar[d]^-{\simeq}\\
\bigoplus_i R(a_i)\ar[r]^-{A} & \bigoplus_j R(b_j)
}
\]
where the vertical maps are graded isomorphisms, and $A$ is a matrix all of whose non-zero entries are powers of $t$ and live on the diagonal.
\end{enumerate}
\end{proposition}
\begin{proof}
The proof proceeds in several steps.\\
\\
\noindent\emph{Step 1:~presentations of graded modules.} Let $M$ be a finitely generated graded $R$-module. Then $M/tM$ is a graded $\kappa$-vector space with $(M/tM)_i=M_i/tM_{i-1}$. Choosing bases for each subspace $(M/tM)_i$ and lifting to elements of $M_i$, we have a graded morphism
\begin{equation}\label{eqn:smith-surj}
F:=\bigoplus_j R(a_j)\xrightarrow{\pi} M,
\end{equation}
which is surjective by Nakayama's Lemma. Note that by the usual structure theorem for modules over a PID, the kernel of $\pi$ is free, so we have a presentation
\begin{equation}\label{eqn:smith-presentation}
0\to F'\to F\xrightarrow{\pi} M\to 0
\end{equation}
of graded modules. Furthermore, by construction,
\begin{equation}\label{eqn:smith-compare-ranks}
\rank_R F=\dim_\kappa M/tM.
\end{equation}

\noindent\emph{Step 2:~\ref{prop:graded-Smith::str-thm} holds for free modules.} Suppose $M$ is a free finite rank graded $R$-module. Consider the presentation \eqref{eqn:smith-presentation}.
By equation \eqref{eqn:smith-compare-ranks} and the fact that $\dim_\kappa M/tM=\dim_\kappa(M\otimes_R \kappa)=\rank_R M$, we see $F'=0$, so $\pi$ is an isomorphism. Thus, by \eqref{eqn:smith-surj}, $M$ has the desired form.
\\
\\
\noindent\emph{Step 3:~proof of \ref{prop:graded-Smith::smith}.} Since \ref{prop:graded-Smith::str-thm} holds for free modules, we may choose homogeneous bases $e_1,\dots,e_m$ for $F$ and $f_1,\dots,f_n$ for $F'$; let $e_j$ have degree $a_j\in\bZ/\ell$ and $f_i$ have degree $b_i\in\bZ/\ell$. Let $A$ be the matrix of $\alpha$ with respect to these bases. Thus, either $A_{ij}$ vanishes or $A_{ij}=u_{ij}t^{v_{ij}}$ with $u_{ij}\in\kappa$, $v_{ij}\geq0$, and $v_{ij}\equiv b_i-a_j\pmod{\ell}$. If $A\neq0$, then after rearranging the order of our bases, we may assume $v_{11}\leq v_{ij}$ for all $i,j$ for which $A_{ij}\neq0$; after further multiplying by a unit, we may assume $u_{11}=1$. Note that if $A_{i1}\neq0$, then $t^{v_{i1}}f_1$ has degree $b_i$, so replacing $f_i$ by $f_i-u_{i1}t^{v_{i1}}f_1$, we may assume $A_{i1}=0$ for all $i\neq1$. Similarly, we may assume $A_{1j}=0$ for all $j\neq1$. We conclude \ref{prop:graded-Smith::smith} holds by induction applied to the submatrix obtained by deleting the first row and column of $A$.\\
\\
\noindent\emph{Step 4:~proof of \ref{prop:graded-Smith::str-thm}.} This is now immediate by considering the presentation \eqref{eqn:smith-presentation} and applying \ref{prop:graded-Smith::smith} to the map $F'\to F$.
\end{proof}

We end this section with a lemma that will be useful in defining motivic integration over twisted arcs.

\begin{lemma}\label{lemmaPushforwardRankrVB}
Let $\ell \in \Z_{>0}$, let $r \in \Z_{\geq 0}$, let $k'$ be a field extension of $k$, and let $\cF_0$ be a rank $r$ locally free sheaf on $\cD^\ell_{0,k'}$. Then $\dim_{k'}H^0(\cF_0) = r$.
\end{lemma}

\begin{proof}
By \autoref{prop:graded-Smith}, $\cF_0$ is a direct sum of coherent sheaves of the form $\left(k'\llbracket t^{1/\ell} \rrbracket/(t)\right)(a)$ for some $a \in \{0, \dots, \ell - 1\}$. Therefore the proposition follows directly from the fact that the invariants of $\left(k'\llbracket t^{1/\ell} \rrbracket/(t)\right)(a)$ is the 1-dimensional $k'$-vector space spanned by $t^{a/\ell}$.
\end{proof}

\begin{remark}
The analog of \autoref{lemmaPushforwardRankrVB} where $\cD^\ell_{0,k'}$ is replaced with $B\mu_\ell$ is false.
\end{remark}

\section{The weight function and motivic integration over twisted arcs}\label{sectionMotivicIntegrationTwistedArcs}

The primary purpose of this section is to develop the formalism of motivic integration over twisted arcs of Artin stacks as well as define a \emph{weight function} (\autoref{def:wt-fnc}) on cyclotomic inertia; the latter function (up to a dimension factor) generalizes the well-known age function, and it will appear as a correction term in our general change of variables formula (\autoref{corollaryTwistedChangeOfVariables}) in the twisted setting.

Throughout this section, let $\cX$ be a finite type Artin stack over $k$ with affine diagonal.

\begin{notation}
If $\cY$ and $\cZ$ are Artin stacks over an algebraic space $S$, we let $\uHom_S(\cY, \cZ)$ denote the Hom stack parametrizing $S$-morphisms from $\cY$ to $\cZ$ and we let $\uHom^{\rep}_S(\cY, \cZ)$ denote the Hom stack parametrizing representable $S$-morphisms from $\cY$ to $\cZ$.
\end{notation}

\begin{proposition}\label{propositionJetStackOpenInHomStack}
Let $\ell \in \Z_{>0}$ and $n \in \Z_{\geq 0}$. Then the natural map
\[
	\uHom^{\rep}_k(\cD^\ell_n, \cX) \to \uHom_k(\cD^\ell_n, \cX)
\]
is a open immersion.
\end{proposition}

\begin{proof}
The natural map
\[
	\uHom^{\rep}_{k[t]/(t^{n+1})}(\cD^\ell_n, \cX \otimes_k k[t]/(t^{n+1})) \to \uHom_{k[t]/(t^{n+1})}(\cD^\ell_n, \cX \otimes_k k[t]/(t^{n+1}))
\]
is an open immersion by \cite[Proposition 4.4]{SatrianoUsatine4}. Using \cite[Proposition 3.5(vii)]{Rydh}, the result follows by applying the Weil restriction $\Res_{(k[t]/(t^{n+1}) )/ k}$ to this open immersion.
\end{proof}

A version of the following definition was introduced in \cite{Yasuda2006}, where it was used to develop motivic integration in the special case where $\cX$ is a Deligne--Mumford stack. The following version is consistent with \cite{Yasuda2024} (see Remark 15.6 (loc.~cit.) for a comparison).

\begin{definition}
Let $\ell \in \Z_{>0}$ and $n \in \Z_{\geq 0}$. The \emph{stack of twisted $n$-jets of $\cX$ of order $\ell$} is
\[
	\sJ^\ell_n(\cX) = \uHom^{\rep}_k(\cD^\ell_k, \cX).
\]
\end{definition}

\begin{remark}\label{remarkJetStackLocallyFiniteTypeAffineDiagonalFiniteType}
By \autoref{propositionJetStackOpenInHomStack} and \cite[Theorem 3.12(xv, xxi) and Remark 3.13]{Rydh}, $\sJ^\ell_n(\cX)$ is a finite type Artin stack over $k$ with affine diagonal.
\end{remark}

For $m \geq n$ the closed immersion $\Spec(k[t]/(t^{n+1})) \to \Spec(k[t]/(t^{m+1}))$ induces a closed immersion $\cD^\ell_n \to \cD^\ell_m$ that induces the \emph{truncation map}
\[
	\theta^m_n: \sJ^\ell_m(\cX) \to \sJ^\ell_n(\cX).
\]

\begin{definition}
If $\ell \in \Z_{>0}$, the \emph{stack of twisted arcs of $\cX$ of order $\ell$} is
\[
	\sJ^\ell(\cX) = \varprojlim_n \sJ^\ell_n(\cX),
\]
where the inverse system is given by the truncation morphisms $\theta^m_n$. We also call the canonical maps $\theta_n: \sJ^\ell(\cX) \to \sJ^\ell_n(\cX)$ \emph{truncation morphisms}.
\end{definition}

\begin{remark}
Although we have no reason to believe $\sJ^\ell(\cX)$ is an Artin stack, it is a stack \cite[Proposition 2.1.9]{Talpo}. We use the symbol $\sJ^\ell(\cX)(k')$ to denote the category of its $k'$-valued points, $\overline{\sJ^\ell(\cX)}(k')$ to denote the set of isomorphism classes of its $k'$-valued points, and the symbol $|\sJ^\ell(\cX)|$ to denote the set of equivalence classes of its points.
\end{remark}

\begin{remark}\label{remarkFormalGAGAandTannakaDuality}
For any field extension $k'$ of $k$, we have that $\sJ^\ell(\cX)(k')$ is the category of representable $k$-morphisms $\cD^\ell_{k'} \to \cX$. This follows from formal GAGA \cite[Theorem 1.6]{AHRetalelocal} and Tannaka duality \cite[1.7.6]{AHRetalelocal}.
\end{remark}

\begin{remark}\label{remarkSurjectiveTruncationMorphisms}
If $\cX$ is smooth over $k$, then $\overline{\sJ^\ell(\cX)}(k') \to \overline{\sJ^\ell_n(\cX)}(k')$ is surjective for all field extensions $k'$ of $k$ by \autoref{prop:coh-affine-smooth}. In particular $|\sJ^\ell(\cX)| \to |\sJ^\ell_n(\cX)|$ is surjective.
\end{remark}

\begin{definition}
Let $\ell \in \Z_{>0}$ and let $\cC \subset |\sJ^\ell(\cX)|$. We call $\cC$ a \emph{cylinder} if we can write $\cC = \theta_n^{-1}(\cC_n)$ for some $n \in \Z_{\geq 0}$ and $\cC_n$ a constructible subset of $|\sJ^\ell_n(\cX)|$.
\end{definition}

\begin{proposition}
Assume that $\cX$ is smooth over $k$, let $\ell \in \Z_{>0}$, and let $\cC \subset |\sJ^\ell(\cX)|$ be a cylinder. Then for all $n \in \Z_{\geq 0}$, the subset $\theta_n(\cC) \subset |\sJ^\ell_n(\cX)|$ is constructible.
\end{proposition}

\begin{proof}
There exists some $m \in \Z_{\geq 0}$ and a constructible subset $\cC_m \subset |\sJ^\ell_m(\cX)|$ such that $\cC = \theta_m^{-1}(\cC_n)$. By \autoref{remarkSurjectiveTruncationMorphisms}
\[
	\theta_n(\cC) = \begin{cases}(\theta^n_m)^{-1}(\cC_m), & m \leq n \\ \theta^m_n(\cC_m), & m \geq n \end{cases}.
\]
The proposition then follows from Chevalley's theorem for Artin stacks \cite[Theorem 5.1]{HallRydh}.
\end{proof}

\begin{proposition}\label{propositionFibersOfTruncation}
Assume that $\cX$ is equidimensional and smooth over $k$ and has a good moduli space. Let $n \in \Z_{\geq 0}$ and $\ell \in \Z_{>0}$, let $k'$ be a field extension of $k$, and let $\Spec(k') \to \sJ^\ell_n(\cX)$ be a $k$-morphism. Then there exist $a,b \in \Z_{\geq 0}$ with $a-b = \dim\cX$ such that
\[
	\sJ^\ell_{n+1}(\cX) \times_{\sJ^\ell_n(\cX)} \Spec(k') \cong \bA^{a}_{k'} \times_{k'} B\bG_{a,k'}^{b}
\]
as stacks over $k'$.
\end{proposition}

\begin{proof}
Let $\Spec(k') \to \sJ^\ell_n(\cX)$ correspond to the twisted jet $\varphi_n: \cD^\ell_{n, k'} \to \cX$, and let $\varphi_0$ be the composition $\cD^\ell_0 \hookrightarrow \cD^\ell_n \xrightarrow{\varphi_n} \cX$. Consider the $k'$ vector spaces
\begin{align*}
	E^0 &= \Ext_{\cD^\ell_{n,k'}}^0(L\varphi_n^* L_\cX, ((t^{n+1})/(t^{n+2}))\cO_{\cD^\ell_{n,k'}})\\
	&= \Ext_{\cD^\ell_{0,k'}}^0(L\varphi_0^* L_\cX, \cO_{\cD^\ell_{0,k'}})\otimes((t^{n+1})/(t^{n+2}))\cO_{\cD^\ell_{n,k'}}
\end{align*}
and
\begin{align*}
	E^{-1} &= \Ext_{\cD^\ell_{n,k'}}^{-1}(L\varphi_n^* L_\cX, ((t^{n+1})/(t^{n+2}))\cO_{\cD^\ell_{n,k'}})\\
	&= \Ext_{\cD^\ell_{0,k'}}^{-1}(L\varphi_{0}^* L_\cX, \cO_{\cD^\ell_{0,k'}})\otimes ((t^{n+1})/(t^{n+2}))\cO_{\cD^\ell_{n,k'}}.
\end{align*}
For any $k'$-algebra $A$, let $\varphi_{n,A}$ denote the composition $\cD^\ell_{n,A} \to \cD^\ell_{n,k'} \xrightarrow{\varphi_n} \cX$ and $\varphi_{0,A}$ denote the composition $\cD^\ell_{0,A} \to \cD^\ell_{0,k'} \xrightarrow{\varphi_0} \cX$. Since the $A$-valued points of the fiber $\sJ^\ell_{n+1}(\cX) \times_{\sJ^\ell_n(\cX)} \Spec(k')$ are nonempty by \autoref{prop:coh-affine-smooth}, deformation theory tell us
\begin{align*}
	\overline{\sJ^\ell_{n+1}(\cX) \times_{\sJ^\ell_n(\cX)} \Spec(k')}(A) &= \Ext_{\cD^\ell_{n,A}}^0(L\varphi_{n,A}^* L_\cX, ((t^{n+1})/(t^{n+2}))\cO_{\cD^\ell_{n,A}})\\
	&= \Ext_{\cD^\ell_{0,A}}^0(L\varphi_{0, A}^* L_\cX, \cO_{\cD^\ell_{0,A}})\otimes((t^{n+1})/(t^{n+2}))\cO_{\cD^\ell_{n,k'}}\\
	&= E^0 \otimes_{k'} A
\end{align*}
and each object of $(\sJ^\ell_{n+1}(\cX) \times_{\sJ^\ell_n(\cX)} \Spec(k'))(A)$ has automorphism group
\begin{align*}
	\Ext_{\cD^\ell_{n,A}}^{-1}(&L\varphi_{n,A}^* L_\cX, ((t^{n+1})/(t^{n+2}))\cO_{\cD^\ell_{n,A}})\\
	&= \Ext_{\cD^\ell_{0,A}}^{-1}(L\varphi_{0, A}^* L_\cX, \cO_{\cD^\ell_{0,A}})\otimes((t^{n+1})/(t^{n+2}))\cO_{\cD^\ell_{n,k'}}\\
	&= E^{-1} \otimes_{k'} A.
\end{align*}
Since these equalities are natural in $A$, we get that
\[
	\sJ^\ell_{n+1}(\cX) \times_{\sJ^\ell_n(\cX)} \Spec(k') \cong \bA^{\dim_{k'}E^0}_{k'} \times_{k'} B\bG^{\dim_{k'}E^{-1}}_{a,k'}.
\]
Therefore we only need to show that
\[
	\dim_{k'}E^0 - \dim_{k'}E^{-1} = \dim\cX.
\]
By \autoref{lemmaPullBackCotangentComplexBecomesTwoVectorBundles}, there exist $r, s \in \Z_{\geq 0}$ with $r-s = \dim\cX$ and an exact triangle
\[
	L\varphi_0^* L_\cX \to \cF \to \cG
\]
with $\cF$ a rank $r$ locally free sheaf and $\cG$ a rank $s$ locally free sheaf. We then have the exact sequence
\[
	0 \to E^{-1} \to H^0(\sHom(\cG, \cO_{\cD^\ell_{0,k'}})) \to H^0(\sHom(\cF, \cO_{\cD^\ell_{0,k'}})) \to E^{0} \to 0.
\]
Therefore
\begin{align*}
	\dim_{k'}E^0 - \dim_{k'}&E^{-1} \\
	&= \dim_{k'} H^0(\sHom(\cF, \cO_{\cD^\ell_{0,k'}})) - \dim_{k'}H^0(\sHom(\cG, \cO_{\cD^\ell_{0,k'}}))\\
	&= r - s\\
	&= \dim\cX,
\end{align*}
where the second equality is by \autoref{lemmaPushforwardRankrVB}.
\end{proof}

\begin{theorem}
Assume that $\cX$ is equidimensional and smooth over $k$ and has a good moduli space. Let $\ell \in \Z_{>0}$, and let $\cC \subset |\sJ^\ell(\cX)|$ be a cylinder. Then the sequence
\[
	\{ \e(\theta_n(\cC)) \bL^{-(n+1)\dim\cX} \}_{n \in \Z_{\geq 0}}
\]
stabilizes for $n$ sufficiently large. Furthermore if $m \in \Z_{\geq 0}$ and $\cC_m \subset |\sJ^\ell_m(\cX)|$ is a constructible subset such that $\cC = \theta_m^{-1}(\cC_m)$, then
\[
	\lim_{n \to \infty} \e(\theta_n(\cC)) \bL^{-(n+1)\dim\cX} = \e(\cC_m)\bL^{-(m+1)\dim\cX}.
\]
\end{theorem}

\begin{proof}
The theorem follows immediately from \autoref{propFibrationClass}, \autoref{propositionFibersOfTruncation}, and the fact that $\bG_a$ is a special group so $\e(B\bG_{a,k}) = \bL^{-1}$.
\end{proof}

We may now define the motivic volume $\nu_\cX$ on twisted arcs.

\begin{definition}
Assume that $\cX$ is equidimensional and smooth over $k$ and has a good moduli space. Let $\ell \in \Z_{>0}$, and let $\cC \subset |\sJ^\ell(\cX)|$ be a cylinder. Then the \emph{motivic volume} of $\cC$ is
\[
	\nu_\cX(\cC) = \lim_{n \to \infty} \e(\theta_n(\cC)) \bL^{-(n+1)\dim\cX}.
\]
\end{definition}

\begin{definition}\label{definitionCylinderForUnion}
The \emph{stack of twisted arcs} of $\cX$ is
\[
	\sJ(\cX) = \bigsqcup_{\ell \in \Z_{>0}} \sJ^\ell(\cX).
\]
We call a subset $\cC \subset |\sJ(\cX)|$ \emph{bounded} if $\cC \cap |\sJ^\ell(\cX)| = \emptyset$ for all but finitely many $\ell$. For any $n \in \Z_{\geq 0}$ the \emph{stack of twisted $n$-jets} of $\cX$ is
\[
	\sJ_n(\cX) = \bigsqcup_{\ell \in \Z_{>0}} \sJ^\ell_n(\cX).
\]
We have a \emph{truncation morphism} $\theta_n: \sJ(\cX) \to \sJ_n(\cX)$ that restricts to each truncation morphism $\theta_n: \sJ^\ell(\cX) \to \sJ^\ell_n(\cX)$. We call a subset $\cC \subset |\sJ(\cX)|$ a \emph{cylinder} if there exists some $n \in \Z_{\geq 0}$ and locally constructible subset $\cC_n \subset |\sJ_n(\cX)|$ such that $\cC = \theta_n^{-1}(\cC_n)$. We note that a subset $\cC \subset |\sJ(\cX)|$ is a bounded cylinder if and only if each $\cC \cap |\sJ^\ell(\cX)|$ is a cylinder in $|\sJ^\ell(\cX)|$ and $\cC \cap |\sJ^\ell(\cX)| = \emptyset$ for all but finitely many $\ell$.

If $\cX$ is equidimensional and smooth over $k$ and has a good moduli space and $\cC \subset |\sJ(\cX)|$ is a bounded cylinder, we define the \emph{motivic volume} of $\cC$ as the finite sum
\[
	\nu_\cX(\cC) = \sum_{\ell \in \Z_{>0}} \nu_{\cX}(\cC \cap |\sJ^\ell(\cX)|).
\]
\end{definition}

\subsection{Weight functions}

In this subsection we define the function $\wt_\cX$ used in \autoref{maintheorem} and show that it arises from a locally constant function on the cyclotomic inertia of $\cX$. The following definition was introduced in \cite{AbramovichGraberVistoli} in order to study Gromov--Witten theory of Deligne--Mumford stacks.

\begin{definition}
Let $\ell \in \Z_{>0}$. The \emph{cyclotomic inertia stack of $\cX$ of order $\ell$} is
\[
	I_{\mu_\ell}(\cX) = \uHom_k^{\rep}(B\mu_\ell, \cX).
\]
\end{definition}

\begin{proposition}
Let $\ell \in \Z_{>0}$. Then $I_{\mu_\ell}(\cX)$ is a finite type Artin stack over $k$ with affine diagonal.
\end{proposition}

\begin{proof}
By \cite[Proposition 4.4]{SatrianoUsatine4}, $I_{\mu_\ell}(\cX)$ is an open substack of $\uHom_k(B\mu_\ell, \cX)$. The latter Hom stack is a finite type Artin stack over $k$ with affine diagonal by \cite[Theorem 3.12(xv, xxi) and Remark 3.13]{Rydh}.
\end{proof}

For any $\ell \in \Z_{>0}$, $w \in \Z$, and $k$-algebra $A$, we let $\cO_{B\mu_{\ell,A}}(w)$ denote the coherent sheaf on $B\mu_{\ell,A} = [\Spec(A) / \mu_{\ell, A}]$ corresponding to the rank one free $A$-module $A$ with $\mu_{\ell,A}$-action where $\xi \in \mu_{\ell}$ acts by multiplication by $\xi^{-w}$. If $E \in D(B\mu_{\ell,A})$, we set $E(w) = E \otimes^L_{\cO_{B\mu_{\ell,A}}} \cO_{B\mu_{\ell,A}}(w)$.

\begin{definition}\label{def:wt-fnc}
Let $\ell \in \Z_{>0}$. We define a \emph{weight function} $\overline{\wt}_\cX: |I_{\mu_\ell}(\cX)| \to \Q$ as follows. If $p \in |I_{\mu_\ell}(\cX)|$ is represented by $\varphi: B\mu_{\ell, k'} \to \cX$ for some field extension $k'$ of $k$, then we set
\begin{align*}
	\overline{\wt}_\cX(p) &= \\
	\dim\cX &+ (1/\ell)\sum_{w = 1}^\ell w\left[ \dim_{k'} H^1\left((L\varphi^*L_\cX)(-w) \right) - \dim_{k'} H^0\left((L\varphi^*L_\cX)(-w) \right) \right].
\end{align*}
\end{definition}

\begin{remark}
An equivalent description of $\overline{\wt}_{\cX}(p)$ is as follows. Let $\{c_j\}_j$ and $\{d_m\}_m$ be integers in $\{1, \dots, \ell\}$ such that
\[
	L^1\varphi^*L_\cX \cong \bigoplus_{j} \cO_{B\mu_{\ell,k'}}(c_j)
\]
and
\[
	L^0\varphi^*L_\cX \cong \bigoplus_{m} \cO_{B\mu_{\ell,k'}}(d_m).
\]
Then
\[
	\overline{\wt}_\cX (p) = \dim\cX + (1/\ell)\left(\sum_j c_j - \sum_m d_m\right).
\]
\end{remark}

\begin{proposition}\label{propositionWeightFunctionIsLocallyConstant}
Assume that $\cX$ is equidimensional and smooth over $k$ and has a good moduli space. Then for any $\ell \in \Z_{>0}$, the function $\overline{\wt}_\cX: |I_{\mu_\ell}(\cX)| \to \Q$ is locally constant.
\end{proposition}

\begin{proof}
Let $k'$ be a field extension of $k$, and let $\varphi: B\mu_{\ell, k'} \otimes_{k'} k'\llbracket t \rrbracket \to \cX$ be a representable $k$-morphism. Then $\varphi$ corresponds to a $k$-morphism $\psi: \Spec(k'\llbracket t \rrbracket) \to I_{\mu_\ell}(\cX)$. Let $\varphi_0: B\mu_{\ell, k'} \to \cX$ denote the composition 
\[
	B{\mu_{\ell, k'}} = B\mu_{\ell, k'} \otimes_{k'} k'\llbracket t \rrbracket/(t) \hookrightarrow B\mu_{\ell, k'} \otimes_{k'} k'\llbracket t \rrbracket \xrightarrow{\varphi} \cX,
\]
and let $\varphi^\circ: B\mu_{\ell, k'\llparenthesis t \rrparenthesis} \to \cX$ denote the composition
\[
	B{\mu_{\ell, k'\llparenthesis t \rrparenthesis}} = B\mu_{\ell, k'} \otimes_{k'} k'\llparenthesis t \rrparenthesis \to B\mu_{\ell, k'} \otimes_{k'} k'\llbracket t \rrbracket \xrightarrow{\varphi} \cX.
\]
Then $\varphi_0$ (resp. $\varphi^\circ$) represents the image in $|I_{\mu_\ell}(\cX)|$ of the special (resp. generic) point of $\Spec(k'\llbracket t \rrbracket)$ under $\psi$. It is therefore sufficient to show that
\begin{align*}
	\dim_{k'} &H^1\left((L\varphi_0^*L_\cX)(-w) \right) - \dim_{k'} H^0\left((L\varphi_0^*L_\cX)(-w) \right)\\
	&= \dim_{k'\llparenthesis t \rrparenthesis} H^1\left((L(\varphi^\circ)^*L_\cX)(-w) \right) - \dim_{k'\llparenthesis t \rrparenthesis} H^0\left((L(\varphi^\circ)^*L_\cX)(-w) \right)
\end{align*}
for any $w \in \{1, \dots, \ell\}$. By \autoref{lemmaPullBackCotangentComplexBecomesTwoVectorBundles} there exist finite rank locally free sheaves $\cF$ and $\cG$ on $B\mu_{\ell, k'} \otimes_{k'} k'\llbracket t \rrbracket$ and an exact triangle
\[
	L\varphi^*L_\cX \to \cF \to \cG.
\]
Let $\cF(-w)_0, \cG(-w)_0$ be the pullback to $\Spec(k')$ of $\cF(-w), \cG(-w)$, respectively, and let $\cF(-w)^\circ, \cG(-w)^\circ$ be the pullback to $\Spec(k'\llparenthesis t \rrparenthesis)$ of $\cF(-w), \cG(-w)$, respectively. We then have exact triangles
\[
	L\varphi_0^*L_\cX(-w) \to \cF(-w)_0 \to \cG(-w)_0
\]
and
\[
	L(\varphi^\circ)^*L_\cX(-w) \to \cF(-w)^\circ \to \cG(-w)^\circ,
\]
which induce exact sequences
\begin{align*}
	0 \to &H^0\left((L\varphi_0^*L_\cX)(-w) \right) \to H^0\left(  \cF(-w)_0 \right)\\
	&\to H^0\left( \cG(-w)_0 \right) \to H^1\left( (L\varphi_0^*L_\cX)(-w) \right) \to 0
\end{align*}
and
\begin{align*}
	0 \to &H^0\left((L(\varphi^\circ)^*L_\cX)(-w) \right) \to H^0\left(  \cF(-w)^\circ \right)\\
	&\to H^0\left( \cG(-w)^\circ \right) \to H^1\left( (L(\varphi^\circ)^*L_\cX)(-w) \right) \to 0.
\end{align*}
Therefore it is sufficient to show that
\[
	\dim_{k'} H^0\left(  \cF(-w)_0 \right) = \dim_{k'\llparenthesis t \rrparenthesis}H^0\left(  \cF(-w)^\circ \right)
\]
and
\[
	\dim_{k'} H^0\left(  \cG(-w)_0 \right) = \dim_{k'\llparenthesis t \rrparenthesis}H^0\left(  \cG(-w)^\circ \right).
\]
Since $\cF(-w)$ and $\cG(-w)$ are finite rank locally free sheaves on $B\mu_{\ell, k'} \otimes_{k'} k'\llbracket t \rrbracket$, we have that $H^0(\cF(-w))$ and $H^0(\cG(-w))$ are finitely generated and torsion free over $k'\llbracket t \rrbracket$ and are therefore finite rank free $k'\llbracket t \rrbracket$-modules. Therefore
\[
	\dim_{k'}\left( H^0(\cF(-w)) \otimes_{k'\llbracket t \rrbracket} k' \right) = \dim_{k'\llparenthesis t \rrparenthesis}\left( H^0(\cF(-w)) \otimes_{k'\llbracket t \rrbracket} k'\llparenthesis t \rrparenthesis \right)
\]
and
\[
	\dim_{k'}\left( H^0(\cG(-w)) \otimes_{k'\llbracket t \rrbracket} k' \right) = \dim_{k'\llparenthesis t \rrparenthesis}\left( H^0(\cG(-w)) \otimes_{k'\llbracket t \rrbracket} k'\llparenthesis t \rrparenthesis \right).
\]
The result then follows from adjunction \cite[Proposition 4.7(iii)]{Alper}.
\end{proof}

The closed immersion $B\mu_\ell = (\cD^\ell_0)_{\red} \hookrightarrow \cD^\ell_0$ induces a morphism $\theta^0: \sJ^\ell_0(\cX) \to I_{\mu_\ell}(\cX)$ and we let $\theta: \sJ^\ell(\cX) \to I_{\mu_\ell}(\cX)$ denote the composition
\[
	\sJ^\ell(\cX) \xrightarrow{\theta_0} \sJ^\ell_0(\cX) \xrightarrow{\theta^0} I_{\mu_\ell(\cX)}.
\]

\begin{definition}
For any $\ell \in \Z_{>0}$, we define the \emph{weight function} $\wt_\cX: |\sJ^\ell(\cX)| \to \Q$ as the composition $|\sJ^\ell(\cX)| \xrightarrow{\theta} |I_{\mu_\ell}(\cX)| \xrightarrow{\overline{\wt}_\cX} \Q$. We let $\wt_\cX: |\sJ(\cX)| \to \Q$ be the function that restricts to $\wt_\cX: |\sJ^\ell(\cX)| \to \Q$ for all $\ell$.
\end{definition}

\begin{remark}
By \autoref{propositionWeightFunctionIsLocallyConstant}, if $\cX$ is equidimensional and smooth over $k$ and has a good moduli space, then $\wt_\cX: |\sJ^\ell(\cX)| \to \Q$ has cylinder level sets.
\end{remark}

\subsection{Measurable sets and integrable functions}

Throughout this subsection, we will additionally assume that $\cX$ is equidimensional and smooth over $k$ and has a good moduli space. In particular, the motivic volume $\nu_\cX$ is well defined on bounded cylinders. There is a standard approach to using motivic volumes of (bounded) cylinders to define measurable sets and their motivic volumes. The following definitions make this explicit in our setting.

\begin{definition}
Let $\cC \subset |\sJ(\cX)|$, let $\cC^0 \subset |\sJ(\cX)|$ be a bounded cylinder, let $\{\cC^i\}_{i \in I}$ be a collection of bounded cylinders in $|\sJ(\cX)|$, and let $\varepsilon \in \R_{>0}$. We call $(\cC^0, \{\cC^i\}_{i \in I})$ a \emph{bounded cylindrical $\varepsilon$-approximation} of $\cC$ if $\Vert \nu_\cX(\cC^i)\Vert < \varepsilon$ for all $i$ and
\[
	(\cC \cup \cC^0) \setminus (\cC \cap \cC^0) \subset \bigcup_{i \in I} \cC^i.
\]
\end{definition}

\begin{definition}\label{definitionMeasurableSetOfTwistedArcs}
Let $\cC \subset |\sJ(\cX)|$. We call $\cC$ \emph{measurable} if for any $\varepsilon \in \R_{>0}$,  it has a bounded cylindrical $\varepsilon$-approximation. If $\cC$ is measurable, the \emph{motivic volume} $\nu_\cX(\cC)$ of $\cC$ is the unique element of $\widehat{\sM}_k$ satisfying
\[
	\Vert \nu_\cX(\cC) - \nu_\cX(\cC^0) \Vert < \varepsilon
\]
for any bounded cylindrical $\varepsilon$-approximation $(\cC^0, \{\cC^i\}_{i \in I})$.
\end{definition}

\begin{remark}
The fact that such a unique element exists follows from standard methods. See, e.g., \cite[Chapter 6 Theorem 3.3.2]{ChambertLoirNicaiseSebag}. Note that the proof there applies word for word except for the following two exceptions. \cite[Proposition 4.4]{SatrianoUsatine5} must be used in place of the special case of schemes, and we need that covers of bounded cylinders by bounded cylinders admit a finite subcover. The latter statement follows directly from \cite[Proposition 2.2]{SatrianoUsatine2} and part (a) and the beginning of part (b) of \autoref{theoremTwistedToWarped} below. We note that the proof of \autoref{theoremTwistedToWarped} parts (a) and (b) never use \autoref{definitionMeasurableSetOfTwistedArcs}.
\end{remark}

We now define measurable and integrable functions on $|\sJ(\cX)|$.

\begin{definition}
Let $\cC \subset |\sJ(\cX)|$, and let $f: \cC \to \Q \cup \{\infty\}$. We call $f$ \emph{measurable} if the level sets of $f$ are measurable. We say $\bL^f$ is \emph{integrable} if $f$ is measurable and there exists some $m \in \Z_{>0}$ and a measurable subset $\cE \subset |\sJ(\cX)|$ such that
\begin{itemize}

\item $\nu_{\cX}(\cE)$ = 0,

\item the restriction of $f$ to $\cC \setminus \cE$ takes values in $\frac{1}{m}\Z$, and

\item the series
\[
	\sum_{q \in \frac{1}{m}\Z} \bL^q \nu_\cX(f^{-1}(q))
\]
converges in $\widehat{\sM}_k[\bL^{1/m}]$.

\end{itemize}
If $\bL^f$ is integrable we set
\[
	\int_\cC \bL^f \diff\nu_\cX = \sum_{q \in \frac{1}{m}\Z} \bL^q \nu_\cX(f^{-1}(q)) \in \widehat{\sM}_k[\bL^{1/m}].
\]
\end{definition}

\subsection{Height and order functions}

We end this section by introducing some useful functions on $|\sJ(\cX)|$. For any field extension $k'$ of $k$ and twisted arc $\varphi: \cD^\ell_{k'} \to \cX$, we may consider the untwisted arc of $\cX$ given by the composition
\[
	\Spec(k'\llbracket t \rrbracket) \xrightarrow{\sim} \Spec(k'\llbracket t^{1/\ell}\rrbracket) \to \cD^\ell_{k'} \xrightarrow{\varphi} \cX
\]
where the first map is given by $f(t^{1/\ell}) \mapsto f(t)$ and the second map is the usual covering map. This assignment induces a map of sets
\[
	\omega: |\sJ^\ell(\cX)| \to |\sL(\cX)|.
\]
\begin{definition}
Let $\cI$ be a quasi-coherent ideal sheaf on $\cX$, let $\cD$ be a $\Q$-Cartier $\Q$-divisor on $\cX$, and let $\cX \to \cY$ be a morphism to a locally finite type Artin stack $\cY$ over $k$. For any $\ell \in \Z_{>0}$, the \emph{order function} $\ord_{\cI}: |\sJ^\ell(\cX)| \to \Q \cup \{\infty\}$ of $\cI$ is the composition
\[
	|\sJ^\ell(\cX)| \xrightarrow{\omega} |\sL(\cX)| \xrightarrow{\ord_{\cI}} \Z \cup \{\infty\} \xrightarrow{\cdot (1/\ell)} \Q \cup \{\infty\},
\]
the \emph{order function} $\ord_{\cD}: |\sJ^\ell(\cX)| \to \Q \cup \{\infty\}$ of $\cD$ is the composition
\[
	|\sJ^\ell(\cX)| \xrightarrow{\omega} |\sL(\cX)| \xrightarrow{\ord_{\cD}} \Q \cup \{\infty\} \xrightarrow{\cdot (1/\ell)} \Q \cup \{\infty\},
\]
and the \emph{height function} $\het_{\cX/\cY}: |\sJ^\ell(\cX)| \to \Q \cup \{\infty\}$ is the composition
\[
	|\sJ^\ell(\cX)| \xrightarrow{\omega} |\sL(\cX)| \xrightarrow{\het_{\cX/\cY}} \Z \cup \{\infty\} \xrightarrow{\cdot (1/\ell)} \Q \cup \{\infty\}.
\]
We let $\ord_{\cI}, \ord_{\cD}, \het_{\cX/\cY}: |\sJ(\cX)| \to \Q \cup \{\infty\}$ denote the functions that restrict to $\ord_{\cI}, \ord_{\cD}, \het_{\cX/\cY}: |\sJ^\ell(\cX)| \to \Q \cup \{\infty\}$, respectively, for all $\ell$.
\end{definition}

\section{The locus of syntomic warps}
\label{sec:syntomic-locus}


The warping stack $\sW(\cX)$ need not be smooth even when $\cX$ is. Thus, it is often useful to restrict attention to smooth substacks. We introduce two such substacks, $\sW^{\synt}(\cX)$ and $\widetilde{\sW}(\cX)$, which will play an important role throughout the paper.

\begin{definition}
Let $\sW^{\synt}(\cX)$ be the fibered category which is the full subcategory of $\sW(\cX)$ whose objects are warped maps $(\cT\xrightarrow{\pi} T,\cT\to\cX)$ with $\pi$ is syntomic. 
\end{definition}

\begin{proposition}\label{prop:syntomic-locus}
Let $\cX$ be a smooth Artin stack over $k$ with affine diagonal. Then $\sW^{\synt}(\cX)$ is an open substack of $\sW(\cX)$ that is smooth over $k$.
\end{proposition}
\begin{proof}
We first show $\sW^{\synt}(\cX)\subset\sW(\cX)$ is an open substack. Let $T\to\sW(\cX)$ be a map corresponding to a warped map $(\cT\xrightarrow{\pi}T,\cT\xrightarrow{f}\cX)$. Since the syntomic condition can be checked on a smooth cover of $\cT$, \cite[Tag 02V3]{stacks-project} shows the locus $\cU\subset\cT$ where $\pi|_\cU$ is syntomic is an open substack of $\cT$. Thus, $U:=T\setminus\pi(\cT\setminus\cU)$ is the locus in $T$ where $\pi$ is syntomic. Since $\pi$ is a good moduli space map, hence closed, we see $U\subset T$ is open.

We now know $\sW^{\synt}(\cX)$ is locally of finite type over $k$, so to prove smoothness, we must show formal smoothness. Let $\Spec A\to\Spec A'$ be a nilpotent thickening with ideal sheaf $\cI$, and let $(\cT\xrightarrow{\pi}\Spec A,\cT\xrightarrow{f}\cX)\in\sW^{\synt}(\cX)(A)$. The obstruction to deforming $\pi$ over $\Spec A'$ lives in $\Ext^2(L_{\cT/A},\pi^*\cI)=H^0\mathcal{E}xt^2(L_{\cT/A},\pi^*\cI)$, where the equality uses that $\pi$ is cohomologically affine. Thus, to show $\mathcal{E}xt^2(L_{\cT/A},\pi^*\cI)$ vanishes, we may check after pullback by a smooth cover $\rho\colon\widetilde{T}\to \cT$. Then $\rho^*\mathcal{E}xt^2(L_{\cT/A},\pi^*\cI)=\mathcal{E}xt^2(\rho^*L_{\cT/A},\rho^*\pi^*\cI)$. From the exact triangle
\[
\rho^*L_{\cT/A}\to L_{\widetilde{T}/A}\to \Omega^1_\rho
\]
and the fact that $\widetilde{T}$ is syntomic over $A$, we see
\[
\mathcal{E}xt^2(\rho^*L_{\cT/A},\rho^*\pi^*\cI)=\mathcal{E}xt^2(L_{\widetilde{T}/A},\rho^*\pi^*\cI)=0.
\]
We therefore have a flat deformation $\pi'\colon\cT'\to\Spec A'$ of $\pi$. We see it is syntomic by applying \cite[Tag 06AG]{stacks-project} to a smooth cover of $\cT'$. Furthermore, \cite[Proposition 3.1]{SatrianoUsatine4} shows that $\pi'$ is a warp. Lastly, the representable map $f$ deforms to a representable map $f'\colon\cT'\to\cX$ by \autoref{prop:coh-affine-smooth}.
\end{proof}

\begin{lemma}\label{lemmaQuasiCompactBasis}
Let $\cZ$ be an Artin stack. Then there is a basis for the topology of $|\cZ|$ consisting of quasi-compact sets.
\end{lemma}

\begin{proof}
Since open subsets of $|\cZ|$ are given by open substacks of $\cZ$, it is sufficient to show that $|\cZ|$ has an open cover by quasi-compact sets. Let $Z \to \cZ$ be a smooth cover by a scheme. The result follows from the fact that $Z \to \cZ$ is an open map and the fact that affine schemes are quasi-compact.
\end{proof}

\begin{proposition}\label{propositionClosureInSyntomicLocus}
Let $\cX$ be a smooth Artin stack over $k$ with affine diagonal. Then the closure of $|\cX|$ in $|\sW^{\synt}(\cX)|$ is an open subset of $|\sW^{\synt}(\cX)|$.
\end{proposition}

\begin{proof}
Let $w \in |\sW^{\synt}(\cX)|$ be in the closure of $|\cX|$. We will prove the proposition by finding an open neighborhood of $w$ that is contained in the closure of $|\cX|$. By \autoref{lemmaQuasiCompactBasis} there exists a quasi-compact open substack $\cV$ of $\sW^{\synt}(\cX)$ such that $w \in |\cV|$. Because $\cV$ is finite type, we may replace it with its connected component containing $w$ and therefore assume $\cV$ is connected. Since $\cV$ is smooth by \autoref{prop:syntomic-locus}, $\cV$ is irreducible. As $w$ is in the closure of $|\cX|$, we have that $|\cX| \cap |\cV|$ is nonempty and therefore dense in $|\cV|$. Thus $|\cV|$ is contained in the closure of $|\cX|$, and we are done.
\end{proof}

\autoref{propositionClosureInSyntomicLocus} allows us to 
introduce the following definition.

\begin{definition}
If $\cX$ is a smooth Artin stack over $k$ with affine diagonal, then 
 let $\widetilde{\sW}(\cX)$ denote the open substack of $\sW^{\synt}(\cX)$ such that $|\widetilde{\sW}(\cX)|$ is the closure of $|\cX|$ in $|\sW^{\synt}(\cX)|$.
\end{definition}

\begin{remark}\label{remarkSmoothnessAndDimensionOfClosureInLCILocus}
By construction $\widetilde{\sW}(\cX)$ is smooth over $k$, and if $\cX$ is equidimensional then $\widetilde{\sW}(\cX)$ is equidimensional of dimension $\dim \cX$.
\end{remark}

\section{Triviality of forms of (truncated) twisted discs}

The goal of this section is to prove the following proposition which shows there are no twisted forms of twisted (truncation) discs. This will play an important role in computing the fibers of the twisted-to-warp jet space map, see \autoref{thm:fibers-twisted->warped-aut-explicit}.

\begin{proposition}\label{prop:no-non-trivial-forms-truncated-twisted}
Let $L/K$ be a finite separable extension of fields and assume $\ell\geq2$ is prime to the characteristic of $K$. Let $n\in\bZ\cup\{\infty\}$ and assume $n\geq2$.

If $\cT_n\to D_{n,K}$ is a map from an Artin stack and we have a $D_{n,L}$-isomorphism $\cT_n\times_K L\xrightarrow{\simeq}\cD_{n,L}^{(\ell)}$, then there exists a $D_{n,K}$-isomorphism $\cT_n\xrightarrow{\simeq}\cD_{n,K}^{(\ell)}$; here $\cD_\infty^{(\ell)}:=\cD^{(\ell)}$.
\end{proposition}

For notational convenience throughout this section, when a ground field $M$ is understood from context, we let $\cD_n^{(\ell)}:=\cD_{n,M}^{(\ell)}$.

We begin by computing the deformation space of a truncated twisted disc.

\begin{lemma}\label{l:twisted-twisted-discs-def-sp}
Fix $\ell\in\bZ_{\geq1}$ and let $A$ be a ring such that $\mu_{\ell,A}$ is \'etale. 
Let $n\in\bZ_{\geq0}$, $\pi_n\colon\cD_{n,A}^{(\ell)}\to D_{n,A}$ be the coarse space map, and $J=(t^{n+1})/(t^{n+2})$. Then
\[
\Ext^i(L_{\cD^{(\ell)}_{n,A} / D_{n,A}},\pi_n^*J)=0
\]
for all $i\neq0,1$ and 
\[
\Ext^i(L_{\cD^{(\ell)}_{n,A} / D_{n,A}},\pi_n^*J)\simeq\pi_n^*J\ \textrm{for\ } i\in [0,1].
\]

More specifically, let $R=A[t]/(t^{n+1})$, $S=R[u]/I$ with $I=(u^\ell-t)$, and $\rho\colon D_{m,A}\to\cD^{(\ell)}_{n,A}$ the standard \'etale cover with $m=\ell(n+1)-1$ and $D_{m,A}=\Spec S$. Then we have natural isomorphisms
\[
\Ext^0(L_{\cD^{(\ell)}_{n,A} / D_{n,A}},\pi_n^*J)\simeq\Hom(S\otimes_{A[u]}\Omega^1_{A[u]/A},\pi_n^*J)^{\mu_\ell}
\]
and
\[
\Ext^1(L_{\cD^{(\ell)}_{n,A} / D_{n,A}},\pi_n^*J)\simeq\Hom(I/I^2,\pi_n^*J)^{\mu_\ell}.
\]
\end{lemma}
\begin{proof}
We suppress the subscript $A$ throughout the proof. Since $J$ is annihilated by $t$, we have 
\[
R\Hom(L_{\cD^{(\ell)}_n / D_n},\pi_n^*J)\simeq R\Hom(L_{\cD^{(\ell)}_0 / D_0},\cO_{\cD^{(\ell)}_0})\otimes_A \pi_n^*J.
\]
We are therefore reduced to the case where $n=0$. Then $\cD^{(\ell)}_0=[D_{\ell-1}/\mu_\ell]$ where $D_{\ell-1}=\Spec S$ and $S=A[u]/(u^\ell)$. Let $\rho\colon D_{\ell-1}\to\cD^{(\ell)}_0$ be the quotient map. Since $\rho$ is a $\mu_\ell$-torsor, 
it is \'etale. Therefore, $\rho^*L_{\cD^{(\ell)}_0 / A}=L_{S/A}$. As $\cD^{(\ell)}_0$ is cohomologically affine, we see
\begin{align*}
\Ext^i(L_{\cD^{(\ell)}_0 / A},\cO_{\cD^{(\ell)}_0}) &= H^0\mathcal{E}xt^i(L_{\cD^{(\ell)}_0 / A},\cO_{\cD^{(\ell)}_0})\\
&=(H^0\rho^*\mathcal{E}xt^i(L_{\cD^{(\ell)}_0 / A},\cO_{\cD^{(\ell)}_0}))^{\mu_\ell} = \Ext^i(L_{S/A},S)^{\mu_\ell}.
\end{align*}

Next, we have a regular closed immersion $i\colon\Spec S\to\bA^1_A$ and hence an exact triangle
\[
I/I^2\to i^*\Omega^1_{A[u]/A}\to L_{S/A}
\]
where $I=(u^\ell)$. It follows that $\Ext^i(L_{S/A},A)=0$ for $i\neq0,1$ and we have an exact sequence
\[
0\to \Ext^0(L_{S/A},A)\to \Hom(i^*\Omega^1_{A[u]/A},S)\to \Hom(I/I^2,A)\to \Ext^1(L_{S/A},S)\to 0.
\]
Let $f=u^\ell$. Then $I/I^2=Sf$ and $i^*\Omega^1_{A[u]/A}=S du$ as $S$-modules. The map $I/I^2\to i^*\Omega^1_{A[u]/A}$ is given by $f\mapsto df=\ell u^{\ell-1} du$. So an element of $\Hom(i^*\Omega^1_{A[u]/A},S)^{\mu_\ell}$ is a choice of element in the weight $1$ piece $S_1=Au$. We see that under the map 
\[
Au=\Hom(i^*\Omega^1_{A[u]/A},S)^{\mu_\ell}\to \Hom(I/I^2,S)^{\mu_\ell}=S_0=A,
\]
the generator $u$ is sent to $\ell u^\ell=0$. Thus, we see $\Ext^0(L_{\cD^{(\ell)}_0 / A},\cO_{\cD^{(\ell)}_0})$ and $\Ext^1(L_{\cD^{(\ell)}_0 / A},\cO_{\cD^{(\ell)}_0})$ are both isomorphic to $A$.
%
\end{proof}

Using our calculation of the deformation space, we show that deformations of truncated twisted discs are again truncated twisted discs.

\begin{proposition}\label{prop:twisted-twisted-discs-defs-truncated-twisted}
Fix $\ell\in\bZ_{\geq1}$ and let $A$ be a ring such that $\mu_{\ell,A}$ is \'etale. 
Let $n\in\bZ_{\geq0}$, $\pi_n\colon\cD_{n,A}^{(\ell)}\to D_{n,A}$ be the coarse space map, and 
\[
\xymatrix{
\cD_{n,A}^{(\ell)}\ar[r]\ar[d]_-{\pi_n} & \cC\ar[d]^-{q}\\
D_{n,A}\ar[r] & D_{n+1,A}
}
\]
be a cartesian diagram with $q$ flat. If $n\geq2$, then we have a $D_{n+1,A}$-isomorphism $\sigma\colon\cC\xrightarrow{\simeq}\cD_{n+1,A}^{(\ell)}$ such that $\sigma\times_{D_{n+1,A}} D_{n-1,A}$ is the identity automorphism.
\end{proposition}

\begin{remark}\label{rmk:restriction-not-id}
At first glance one may think that \autoref{prop:twisted-twisted-discs-defs-truncated-twisted} contradicts \autoref{l:twisted-twisted-discs-def-sp} which tells us the deformation space is $1$-dimensional. However, (for $\ell>1$) the isomorphism in \autoref{prop:twisted-twisted-discs-defs-truncated-twisted} is of algebraic stacks, not an isomorphism of deformations, i.e., the isomorphism does not necessarily yield the identity automorphism on $\cD_{n,A}^{(\ell)}$.
\end{remark}

\begin{proof}[{Proof of \autoref{prop:twisted-twisted-discs-defs-truncated-twisted}}]
We suppress the subscript $A$ throughout the proof. Let $R:=A[t]/t^{n+2}$, i.e., the coordinate ring of $D_{n+1}$. By \autoref{l:twisted-twisted-discs-def-sp}, we necessarily have $\cC\simeq [C/\mu_\ell]$, where
\[
C=\Spec R[u]/(u^\ell - t - ct^{n+1})
\]
and $c\in A$. Indeed, letting $\rho\colon D_m\to\cD^{(\ell)}_n$ be the standard \'etale cover with $m=\ell(n+1)-1$, \autoref{l:twisted-twisted-discs-def-sp} and its proof show (following notation from the lemma) $\Ext^1(L_{\cD^{(\ell)}_n / D_n},\pi_n^*J)=\Ext^1(L_{D_m / D_n},\rho^*\pi_n^*J)^{\mu_\ell}=\Hom(I/I^2,\pi_n^*J)^{\mu_\ell}$. The first equality tells us that $\cC$ is given by an equivariant deformation of $D_m$, i.e., $\cC=[C/\mu_\ell]$ with $C$ an equivariant deformation of $D_m$; the second equality tells us $C$ is given by deforming the equation $u^\ell-t$ to one of the form $u^\ell - t - ct^{n+1}$ with $c\in A$.

Note that $t + ct^{n+1}=tv$ with $v:=1 + ct^n\in R^*$ as $n\geq1$. Since $n\geq2$,
\[
(1+c\ell^{-1}t^n)^\ell=v
\]
in $R$. As a result, the map $\sigma\colon u\mapsto u(1+c\ell^{-1}t^n)$ defines a $\mu_\ell$-equivariant isomorphism $C\to\Spec R[u]/(u^\ell-t)$. Lastly, note that $1+c\ell^{-1}t^n=1$ mod $t^n$, showing that $\sigma\times_{D_{n+1}} D_{n-1}$ is the identity automorphism.
\end{proof}

Ultimately, we prove \autoref{prop:no-non-trivial-forms-truncated-twisted} by lifting $\cT_n$ to $D$. The following technical lemma allows us to conclude that our resulting lift is a twisted form of a twisted disc.

\begin{lemma}\label{l:coh-complete-lift-->twisted-disc}
We work over a field $K$ whose characteristic is prime to $\ell$. Let $\cT\to D$ be a flat map and let $\cT_m\to D_m$ denote its base change for each $m$. Suppose $\cT$ is the completion of the adic family $\{\cT_m\}_m$ in the sense of \cite[Definition 1.9]{AHRetalelocal}. Let $n\geq2$ and $\sigma_n\colon\cT_n\xrightarrow{\simeq} \cD_n^{(\ell)}$ be a $D_n$-isomorphism. Then $\sigma_n \times_{D_n} D_{n-1}$ lifts to a $D$-isomorphism $\sigma\colon\cT\xrightarrow{\simeq} \cD^{(\ell)}$.
\end{lemma}

\begin{proof}
First note that by \cite[Example 3.10]{AHRetalelocal}, $\cD^{(\ell)}$ is coherently complete along $\cD_0^{(\ell)}$ since $\cD^{(\ell)}\to D$ is flat and proper. Thus, by Tannaka duality (see \cite[\S1.7.6]{AHRetalelocal}), it suffices to construct a compatible family of isomorphisms $\sigma'_m:\cT_m\xrightarrow{\sim}\cD_m^{(\ell)}$ over $D_m$ with $\sigma'_{n-1} = \sigma_n \times_{D_n} D_{n-1}$. 

Suppose we have constructed a family of isomorphisms $\sigma'_m: \cT_m \xrightarrow{\sim} \cD^\ell_m$ for $m < r$ and $\tau_r: \cT_r \xrightarrow{\sim} \cD^\ell_r$ which are all compatible. Then by \autoref{prop:twisted-twisted-discs-defs-truncated-twisted}, there exists an isomorphism $\tau_{r+1}: \cT_{r+1} \xrightarrow{\sim} \cD^\ell_{r+1}$ over $D_{r+1}$ such that $\sigma'_{r-1} = \tau_{r+1} \times_{D_{r+1}} D_{r-1}$. Setting $\sigma'_r = \tau_{r+1} \times_{D_{r+1}} D_r$ and repeating this process completes our proof.
\end{proof}

The following is the last ingredient needed for the proof of \autoref{prop:no-non-trivial-forms-truncated-twisted}.

\begin{lemma}\label{l:twisted-twisted-disc-->twisted-disc}
\autoref{prop:no-non-trivial-forms-truncated-twisted} holds when $n=\infty$.
\end{lemma}
\begin{proof}
We drop the subscript $\infty$. Let $\pi\colon\cT\to D_K$ denote the map and let $\pi_L\colon\cT_L\to D_L$ denote the base change to $L$. Note that $\pi$ is flat since it is after the \'etale base change $\Spec L\to\Spec K$. 
Since we have \'etale surjections $D_L\to\cD_L^{(\ell)}\to\cT$ and $D_L$ is regular, we see $\cT$ is regular. Letting $D^\circ_K$ denote the generic point of $D_K$, since $K((t))\otimes_{K} L=L((t))$ and $\pi_L$ is an isomorphism over the generic point, we see $\pi$ is an isomorphism over $D^\circ_K$. Thus, $\cT$ is integral.

Consider the ideal sheaf $\cI_{D_{0,K}/D_K}$ of $D_0\subset D$ and let $\cI_{\cT^{\red}/\cT}$ be the ideal sheaf of the reduction $\cT^{\red}\subset\cT$. Since $\pi$ is flat, we see $\pi^*\cI_{D_{0,K}/D_K}$ is an ideal sheaf and that $\pi^*\cI_{D_{0,K}/D_K}\subset\cI_{\cT^{\red}/\cT}\subset\cO_\cT$. We claim that $\cI_{\cT^{\red}/\cT}=(\pi^*\cI_{D_{0,K}/D_K})^\ell$. Indeed, this can be checked after base change to $L$; since $\Spec L\to\Spec K$ is \'etale, $\cT^{\red}\times_K L$ is the reduction of $\cD_L^{(\ell)}$ and the ideal sheaves over $K$ pull back to the corresponding ideal sheaves over $L$, and hence the statement holds by the construction of $\cD_L^{(\ell)}$ as a root stack.

Since $\cT$ is integral and regular, dualizing the inclusion $\cI_{\cT^{\red}/\cT}\subset\cO_\cT$ yields a line bundle $\cL$ with section $s$ on $\cT$. By construction, we have $\cL^{\otimes\ell}\simeq\pi^*\cI_{D_{0,K}/D_K}^\vee\simeq\cO_\cT$ such that $s^{\otimes\ell}$ maps to $t\in\cO_\cT$. This therefore defines a map to the root stack $\cT\to\cD_K^{(\ell)}$ over $D_K$. This is an isomorphism since it becomes so after base change to $L$.
\end{proof}

\begin{proof}[{Proof of \autoref{prop:no-non-trivial-forms-truncated-twisted}}]
Let $\pi_n\colon\cT_n\to D_{n,K}$ be the structure map. For $m<n$, let $\pi_m\colon\cT_m\to D_{m,K}$ denote the base change. We first construct a compatible family of flat deformations $\pi_m\colon\cT_m\to D_{m,K}$ for all $m\geq 0$. Suppose we have constructed such a family for $m\leq N$. Note that $\pi_n$ has finite diagonal since its base change $\pi_{n,L}=\pi_n\times_K L$ does. Therefore, by deformation theory, every $\pi_m$ has finite diagonal. It then follows from \cite[Lemma 3.2(1)]{SatrianoUsatine4} that each $\pi_m$ is a cohomologically affine. Letting $J_N=(t^{N+1})/(t^{N+2})$, the obstruction to deforming $\pi_N\colon\cT_N\to D_{N,K}$ to $D_{N+1,K}$ lives in $\Ext^2(L_{\cT_N / D_{N,K}},\pi_N^*J_N)=H^0\mathcal{E}xt^2(L_{\cT_N / D_{N,K}},\pi_N^*J_N)$, where the equality follows because $\pi_N$ is cohomologically affine. Then this $\Ext$-group vanishes since we may check this after base change to $L$, where it follows from \autoref{l:twisted-twisted-discs-def-sp}. We have therefore constructed our desired family of deformations $\pi_m\colon\cT_m\to D_{m,K}$ for all $m\geq n$.

Next, since $\pi_0$ has finite diagonal, \cite[Corollary 4.14(1)]{AHRLuna} shows $\cT_0$ is linearly fundamental in the sense of \cite[Definition 2.7]{AHRetalelocal}. Then Theorem 1.10 of (loc.~cit) shows that the completion $\cT$ of the adic sequence $\{\cT_n\}_n$ exists; furthermore, the proof shows that $\cT_n$ is the $n$-th infinitesimal thickening of $\cT_0$ in $\cT$. The argument in the second and third paragraphs of the proof of \cite[Proposition 4.5]{SatrianoUsatine4} then show that $\pi$ is flat and that the natural map $\cT_n\to\cT\times_{D_K} D_{n,K}$ is an isomorphism.

Let $\pi_L\colon\cT_L\to D_L$ be the base change to $L$. By \cite[Lemma 3.5(1)]{AHRetalelocal}, $(\cT_{0,L},\cT_L)$ is coherently complete, and so $\cT_L$ is the completion of the adic sequence $\{\cT_{m,L}\}_m$ in the sense of \cite[Definition 1.9]{AHRetalelocal}. Therefore, \autoref{l:coh-complete-lift-->twisted-disc} proves $\cT_L\simeq\cD^{(\ell)}_L$ over $D_L$. \autoref{l:twisted-twisted-disc-->twisted-disc} then shows $\cT\simeq\cD^{(\ell)}_K$ over $D_K$, and hence $\cT_n\simeq\cD^{(\ell)}_{n,K}$ over $D_{n,K}$.
\end{proof}



\section{The fibers of the twisted-to-warped jet space map}

In this section, we prove \autoref{thm:fibers-twisted->warped-aut-explicit} below, which is the key ingredient in the proof of \autoref{theoremTwistedToWarped}, which relates the motivic measure of twisted arcs of $\cX$ and untwisted arcs of the warping stack $\sW(\cX)$.

\begin{theorem}\label{thm:fibers-twisted->warped-aut-explicit}
Let $\cX$ be a finite type Artin stack with affine diagonal over $k$. 
Let $n\geq2$, let $L/k$ be a field extension, and consider an $L$-point of $\sL_n(\sW(\cX))$ given by the warped map $(\cD_n\to D_{n,L},f\colon\cD_n\to\cX)$.

If the fiber of $\cJ^{\ell}_n(\cX)\to \sL_n(\sW(\cX))$ is non-empty, then it is given by
\[
\cJ^{\ell}_n(\cX)\times_{\sL_n(\sW(\cX))} \Spec L\simeq \bA^1_L.
\]
More specifically, the fiber is a $\bG_a$-torsor, where $\bG_a$ is canonically the vector group associated to the $1$-dimensional vector space
\[
\Ext^0(L_{\cD^{(\ell)}_{0,L} / L},\mcJ);
\]
here $\mcJ$ denotes the pull back to $\cD^{(\ell)}_{0,L}$ of the ideal sheaf defining the closed immersion $D_{n,L}\hookrightarrow D_{n+1,L}$.
\end{theorem}

We prove \autoref{thm:fibers-twisted->warped-aut-explicit} in two main steps. In \autoref{prop:fibers-twisted->warped-aut} we show the fibers are given by the Weil restriction of the stack $\uAut_{D_{n,L}}(\cD^{(\ell)}_{n,L})$. In \autoref{prop:resaut-is-A1} we then prove that this Weil restriction is given by $\bA^1_L$.

\subsection{Initial computation of the fibers}

Keep the hypotheses of \autoref{thm:fibers-twisted->warped-aut-explicit}. In this brief subsection, we prove:

\begin{proposition}\label{prop:fibers-twisted->warped-aut}
Let $n\geq2$. Consider a $k'$-point of $\sL_n(\sW(\cX))$ corresponding to a warped map $(\cD_n\to D_{n,k'},f\colon\cD_n\to\cX)$. If the fiber of $\cJ^\ell_n(\cX)\to \sL_n(\sW(\cX))$ is non-empty, then it is given by
\[
\Spec k'\times_{\sL_n(\sW(\cX))} \cJ^\ell_n(\cX)\simeq \Res_{D_{n,k'}/k'}\uAut_{D_n}(\cD^{(\ell)}_{n,k'})
\]
where $\Res$ denotes Weil restriction.
\end{proposition}

We prove this after a preliminary lemma.

\begin{lemma}\label{l:fibers-twisted->warped}
Consider a $k'$-point of $\sL_n(\sW(\cX))$ corresponding to a warped map $(\cD_n\to D_{n,k'},f\colon\cD_n\to\cX)$. Then 
\[
\Spec k'\times_{\sL_n(\sW(\cX))} \cJ^\ell_n(\cX)\simeq \Res_{D_{n,k'}/k'}\uIsom_{D_n}(\cD^{(\ell)}_{n,k'},\cD_n)
\]
where $\Res$ denotes Weil restriction.
\end{lemma}
\begin{proof}
The $T$-valued objects of the fiber product are given by an object of $\cJ^\ell_n(\cX)(T)$ and an isomorphism with our given warped map, namely $g\colon\cD_{n,T} ^{\ell}\to\cX$, $\phi\colon\cD_{n,T}^{\ell}\xrightarrow{\simeq}\cD_{n,T}$, and $\alpha\colon g\Rightarrow f_T\phi$. Importantly, recall that $(\phi,\alpha)$ is only an equivalence class of $2$-arrows, i.e., it is only well-defined up to replacing $\phi$ by a $2$-isomorphic map $\cD^{\ell}_{n,T}\to\cD_{n,T}$.

Then a morphism $(g,\phi,\alpha)\to(g',\phi',\alpha')$ is given by a morphism $g\to g'$ compatible with the identity morphism $f_T\to f_T$, i.e., we have $\xi\colon g\Rightarrow g'$ such that there exists a $2$-isomorphism $\eta\colon\phi\Rightarrow\phi'$ with the property that the diagram
\[
\xymatrix{
g\ar@{=>}[r]^-{\alpha}\ar@{=>}[d]_-{\xi} & f_T\phi\ar@{=>}[d]^-{\varphi_T(\eta)}\\
g'\ar@{=>}[r]^-{\alpha'} & f_T\phi'\\
}
\]
commutes. If such an $\eta$ exists, then it is unique since $f_T$ is representable. Thus, we see that $\eta$ uniquely characterizes $\xi$ as $(\alpha')^{-1}\circ f_T(\eta)\circ\alpha$ and if $f_T(\eta)=f_T(\eta')$, then $\eta=\eta'$ since $f_T$ is representable. Therefore, a morphism $(g,\phi,\alpha)\to(g',\phi',\alpha')$ is simply given by $\eta\colon\phi\Rightarrow\phi'$. In particular, $(g,\phi,\alpha)$ is isomorphic to $(f_T\phi,\phi,\id)$ and $(\phi,\id)$ is the unique element in its equivalence class whose second coordinate is $\id$. We have therefore normalized all objects to be of the form $(f_T\phi,\phi,\id)$ and morphisms $(f_T\phi,\phi,\id)\to(f_T\phi,\phi',\id)$ are then simply given by arrows $\phi\Rightarrow\phi'$. In other words, we have shown that $\Spec k'\times_{\sL_n(\sW(\cX))} \cJ^\ell_n(\cX)$ is canonically identified with $\Isom_{D_{n,T}}(\cD^{(\ell)}_{n,T},\cD_{n,T})$
\end{proof}

\begin{proof}[{Proof of \autoref{prop:fibers-twisted->warped-aut}}]
The fiber is given by $\cF:=\Res_{D_{n,k'}/k'}\uIsom_{D_n}(\cD^{(\ell)}_{n,k'},\cD_n)$ by \autoref{l:fibers-twisted->warped}. We may assume $\cF$ is non-empty. It suffices to show $\cF$ has a $k'$-point since, given such a point, we have an isomorphism $\cD^{(\ell)}_{n,k'}\to\cD_n$ which we can then use to identify $\cF$ with $\Res_{D_{n,k'}/k'}\uAut_{D_n}(\cD^{(\ell)}_{n,k'})$.

Since $\cF$ is non-empty and locally of finite type, it has a $T$-valued point for some finite type $k'$-scheme $T$. Since $T$ has a $k''$-point for some finite extension $k''/k'$, we see there is a $k''$-point of $\cF$. In other words, we have an isomorphism 
\[
\cD^{(\ell)}_{n,k''}\xrightarrow{\simeq}\cD_n\times_{k'}k'',
\]
over $D_{n,k''}$, which shows that $\cD_n$ is a twisted form of $\cD^{(\ell)}_{n,k'}$. By \autoref{prop:no-non-trivial-forms-truncated-twisted}, we see $\cD_n$ is in fact isomorphic to $\cD^{(\ell)}_{n,k'}$ over $D_{n,k'}$. Thus, $\cF$ has a $k'$-point.
\end{proof}

\subsection{Automorphisms of truncated twisted discs}

Throughout this section, $n$ is taken to be any non-negative integer unless otherwise specified. The following are the two main results of this subsection.

\begin{proposition}\label{prop:resaut-is-A1}
Fix $\ell\in\bZ_{>1}$. Let $K$ be a field with characteristic prime to $\ell$ that contains all $\ell$-th roots of unity. If $n>0$, then
\[
\Res_{D_{n,K}/K}\uAut_{D_{n,K}}(\cD^{(\ell)}_{n,K})\simeq\bA^1_K.
\]
More specifically, the left hand side is canonically identified with the affine space associated to the $1$-dimensional vector space 
\[
\Ext^0(L_{\cD^{(\ell)}_{0,K} / K},\mcJ),
\]
where $\mcJ$ denotes the pull back to $\cD^{(\ell)}_{0,K}$ of the ideal sheaf defining the closed immersion $D_{n,K}\hookrightarrow D_{n+1,K}$
\end{proposition}

\begin{proposition}\label{prop:AutD0-is-Gm}
If $\ell>1$ and $k$ is our algebraically closed field of characteristic $0$, then $\uAut_{D_{0,k}}(\cD^{(\ell)}_{n,k})\simeq\bG_{m,k}$.
\end{proposition}

We begin with the following technical lemma.

\begin{lemma}\label{l:strange-actions-muell-torsor}
Fix $\ell\in\bZ_{\geq1}$ and let $S$ be a scheme such that $\mu_{\ell,S}\simeq\bZ/\ell$, e.g., a scheme over a field $L$ with characteristic prime to $\ell$ containing all $\ell$-th roots of unity. Let $n\in\bZ_{\geq0}$ and $m=\ell(n+1)-1$. Let $P=\mu_{\ell,S}\times_S D_{m,S}$, $q\colon P\to D_{m,S}$ be the projection map, and $p\colon P\to D_{m,S}$ a morphism.

Equip $P$ with two $\mu_\ell$-actions, the first being the standard one acting on the $\mu_\ell$-component, and the second unspecified. Assume
\begin{enumerate}
\item $q$ is equivariant for second action when the target of $q$ is given the standard $\mu_\ell$-action,
\item $p$ is equivariant for the first action when the target of $p$ is given the standard $\mu_\ell$-action, and
\item $p$ is a $\mu_\ell$-torsor for the second action (where the target of $p$ is given the trivial action).
\end{enumerate}
Then the second $\mu_\ell$-action on $P=\Spec_S \cO_S[u,v]/(u^{m+1},v^\ell-1)$ is uniquely determined, given by assigning $u$ weight $1$ and $v$ weight $-1$.
\end{lemma}
\begin{proof}
We supress the subscript $S$ throughout this proof. To begin, fix an isomorphism $\mu_{\ell,S}\simeq\bZ/\ell$ and let $\zeta\in \mu_{\ell,S}$ correspond to the generator $1\in\bZ/\ell$. Then $P$ consists of $\ell$ disjoint copies of $D_m$ labeled by the elements of $\bZ/\ell$. Now via the second $\mu_\ell$-action, $\zeta$ determines an automorphism of $P$; every automorphism of $P$ is given by a permutation of the connected components composed with automorphisms of the individual components. Since $q$ is equivariant for the second $\mu_\ell$-action, we see $\zeta$ must act in a simple manner:~it permutes the connected components via a permutation $\tau$ and on each connected component $D_m$, it acts via the standard $\mu_\ell$-action. 
Furthermore, $p$ is a torsor for this $\mu_\ell$-action over the connected base $D_m$, so $\tau$ must be transitive.

We next compute the map $p$. Fix a permutation $\epsilon$ such that $\epsilon\tau\epsilon^{-1}$ is the $\ell$-cycle $(\ell,\ell-1,\dots,1)$. Using that $p$ is the quotient by the second $\mu_\ell$-action, we will show that it is given as follows:~the map $D_m\to D_m$ from the $i$-th component of $P$ is given by the standard action of $\zeta^{\epsilon(i)}$. To see this, note that $\epsilon$ gives a way to reorder to components of $P$ so that second action is simpler, i.e., $\epsilon$ defines an automorphism $\psi$ of $P$ which maps the $i$-th component to the $\epsilon(i)$-th component via the identity map. Via $\psi$, our second $\mu_\ell$-action corresponds to the coaction on $\cO_S[u,v]/(u^{m+1},v^\ell-1)$ where $u$ has weight $1$ (since we act via $\zeta$ on each component) and $v$ has weight $-1$ (since the $i$-th component maps to $(i-1)$-th component after conjugation by $\psi$). Thus, the invariant ring is $\cO_S[uv]/(uv)^{m+1}$ where the map $D_m\to D_m$ from the $i$-th component is given by the action of $\zeta^i$. Using the isomorphism $\psi$, we then see the quotient map $p$ is given as stated above:~the map $D_m\to D_m$ from $i$-th component is given by the action of $\zeta^{\epsilon(i)}$.

We now determine the second $\mu_\ell$-action. Recall that $p$ is also equivariant where the source is given the first $\mu_\ell$-action and the target is given the standard $\mu_\ell$-action. This first $\mu_\ell$-action maps the $i$-th component to the $(i+1)$-th component via the identity map. Thus, we obtain a commutative square
\[
\xymatrix{
D_m\ar[r]^-{\id}\ar[d]_-{\zeta^{\epsilon(i)}} & D_m\ar[d]^-{\zeta^{\epsilon(i+1)}}\\
D_m\ar[r]^-{\zeta} & D_m
}
\]
and so $\epsilon(i+1)=\epsilon(i)+1$. It follows that $\tau=\epsilon^{-1}(\ell,\ell-1,\dots,1)\epsilon=(\ell,\ell-1,\dots,1)$. Therefore, the second $\mu_\ell$-action is given by the coaction on $\cO_S[u,v]/(u^{m+1},v^\ell-1)$ where $u$ has weight $1$ and $v$ has weight $-1$.
%
\end{proof}

We next show that \'etale locally, automorphisms of $\cD_n^{(\ell)}$ lift to equivariant automorphisms of its standard cover.

\begin{proposition}\label{prop:auts-Dnell}
Fix $\ell\in\bZ_{\geq1}$ and let $S$ be a scheme such that $\mu_{\ell,S}\simeq\bZ/\ell$. If $\sigma$ is an $S$-automorphism of $\cD_{n,S}^{(\ell)}$, then there is a finite \'etale cover $h\colon S'\to S$ such that $\sigma\times_S S'$ is induced by an equivariant automorphism on the cover $D_{\ell(n+1)-1,S'}$ of $\cD_{n,S'}^{(\ell)}$.

Furthermore, if every $\mu_\ell$-torsor on $S$ is trivial (e.g., $S=\Spec L$ with $L$ algebraically closed), then we may take $h=\id$.
\end{proposition}

\begin{remark}\label{rmk:auts-Dnell}
Let $m=\ell(n+1)-1$, $\rho\colon D_{m,S}\to\cD_{n,S}^{(\ell)}$ be the standard \'etale cover, and $\iota\colon S=D_{0,S}\hookrightarrow D_{m,S}$ the closed immersion. Let $q_0\colon P_0\to S$ be the $\mu_\ell$-torsor given by the cartesian diagram
\[
\xymatrix{
P_0\ar[r]\ar[d]_-{q_0} & D_{m,S}\ar[d]^-{\rho}\\
S\ar[r]^-{\sigma\rho\iota} & \cD_{n,S}^{(\ell)}
}
\]
The proof of \autoref{prop:auts-Dnell} shows that we may take $h\colon S'\to S$ to be any finite \'etale cover trivializing the torsor $q_0$, e.g., we may choose $S'=P_0$ and $h=q_0$.
\end{remark}

\begin{proof}[{Proof of \autoref{prop:auts-Dnell}}]
We supress the subscript $S$ throughout this proof. Let $m=\ell(n+1)-1$ and $\rho\colon D_m\to\cD_n^{(\ell)}$ be standard \'etale cover. Then $\sigma$ is induced by a $\mu_\ell$-torsor $\rho'\colon\cP\to\cD_n^{(\ell)}$ and a $\mu_\ell$-equivariant isomorphism $f\colon\cP\to D_m$. It suffices to prove that $\rho$ and $\rho'$ are isomorphic as torsors over $\cD_n^{(\ell)}$, for then $\sigma$ is induced by the equivariant automorphism $f$.

Consider the cartesian diagram
\[
\xymatrix{
P
\ar[d]_-{q}\ar[r]^-{p} & \cP\ar[r]^-{f}_-{\simeq}\ar[d]^-{\rho'} & D_m\ar[d]^-\rho\\
D_m\ar[r]^-{\rho} & \cD^{(\ell)}_n\ar[r]^-{\sigma}_{\simeq} & \cD^{(\ell)}_n
}
\]
Then $P$ comes equipped with two $\mu_\ell$-actions coming from the pullbacks of the $\mu_\ell$-actions on vertical (respectively horozontal) maps $\rho$. We refer to these as the first and second actions. The map $q$ is a torsor for the first action and equivariant for the second action; similarly, $p$ is a torsor for the second action and equivariant for the first action. 

Note that $\rho'$ is obtained by quotienting the torsor $q$ by the second $\mu_\ell$-action. Similarly the torsor $\rho$ is obtained by quotienting the trivial torsor $\mu_\ell\times D_m\to D_m$ by the action $\xi\cdot(\zeta,x)=(\zeta\xi^{-1},\xi x)$; let us say the first action on $\mu_\ell\times D_m$ is the one where it acts through the $\mu_\ell$-factor, and let us refer to $\xi\cdot(\zeta,x)=(\zeta\xi^{-1},\xi x)$ as the second action on $\mu_\ell\times D_m$. We must therefore construct an isomorphism $\mu_\ell\times D_m\simeq P$ of first-action torsors which is equivariant with respect to the second actions.

Pulling back $q$ via the closed immersion $D_0\hookrightarrow D_m$, we obtain a $\mu_\ell$-torsor $q_0\colon P_0\to D_0=S$. Thus, $q_0$ is a finite \'etale cover. So, after replacing $S$ by any finite \'etale cover which trivializes $q_0$ (e.g., base changing by $q_0$ itself), we may assume $P_0$ is the trivial $\mu_\ell$-torsor. It then follows from invariance of the \'etale site that $q$ is also the trivial torsor. Thus, there exists an isomorphism $\psi\colon\mu_\ell\times D_m\xrightarrow{\simeq} P$ of torsors for the first action. We claim that $\psi$ is automatically equivariant for the second actions. Indeed, the second action on $P$ induces a new action on $\mu_\ell\times D_m$ which we will call the third action; we must show the second and third actions are the same. Note that $fp\psi$ is a torsor for the third action and so the hypotheses of \autoref{l:strange-actions-muell-torsor} are met. As a result, the second and third actions on $\mu_\ell\times D_m$ are the same.
\end{proof}

We next classify equivariant automorphisms of the standard cover of $\cD_{n,S}^{(\ell)}$ as well as their $2$-automorphisms.

\begin{lemma}\label{l:when-are-equiv-auts-Dnell-isomorphic}
Fix $\ell\in\bZ_{>1}$ and let $A$ be a ring such that $\mu_{\ell,A}\simeq\bZ/\ell$. Let $m=\ell(n+1)-1$ and $\sigma$ be a $D_{n,A}$-automorphism of $\cD_{n,A}^{(\ell)}$ induced by an equivariant automorphism $f$ of $D_{m,A}=\Spec A[u]/(u^{m+1})$. Then
\begin{enumerate}
\item\label{l:when-are-equiv-auts-Dnell-isomorphic::fzetaa} we have
\[
f(u) = 
\begin{cases}
\zeta u + au^{\ell n+1},& n>0\\
bu,& n=0
\end{cases}
\]
with $a\in A$, $b\in A^*$, and $\zeta\in\mu_\ell(A)$. Furthermore, if $n>0$, then after replacing $\sigma$ by a $2$-isomorphic $D_{n,A}$-automorphism of $\cD_{n,A}^{(\ell)}$, we may assume $\zeta=1$. In particular, the $D_{n-1,A}$-automorphism of $\cD_{n-1,A}^{(\ell)}$ induced by $\sigma$ is $2$-isomorphic to $\id$.

\item\label{l:when-are-equiv-auts-Dnell-isomorphic::n-id-2auts} 
if $\alpha\colon\sigma\Rightarrow\sigma$ is a $2$-automorphism and $\Spec A$ has no non-trivial $\mu_\ell$-torsors (e.g., $A$ is an algebraically closed field), then $\alpha=\id$. 
\end{enumerate}
\end{lemma}
\begin{proof}
We suppress the subscript $A$ throughout the proof. Let $f$ be a $D_n$-equivariant automorphism of $D_m$. We may write $D_n=\Spec R$ and $D_m=\Spec S$ where $R=A[t]/(t^{n+1})$ and $S=R[u]/(u^\ell-t)$. Since $f$ is equivariant over $R$, we must have $f(u)=ug(t)$ with $(ug(t))^\ell=t=u^\ell$. If $n=0$, then $g(t)=b\in A\subset R$; requiring $f$ to be an automorphism is equivalent to requiring $b\in A^*$. Next assume $n>0$. Since the annihilator of $u^\ell$ in $S$ is $(u^{\ell n})$, and $u^\ell(g(t)^\ell-1)=0$, we see $g(t)^\ell=1$ in $S/(u^{\ell n})$. Using the fact that $\Spec A=D_0\to D_{\ell n-1}$ is a nilpotent thickening (this uses $n>0$) and $\mu_{\ell,A}$ is \'etale, the infinitesimal lifting criterion tells us $\mu_\ell(S/(u^{\ell n}))=\mu_\ell(A)$. Recalling that $u^{\ell n}=t^n$, we therefore find
\[
f(u)=u(\zeta+t^n h(t))
\]
with $\zeta\in\mu_\ell(A)$. Since $t^{n+1}=0$ in $S$, we see $t^nh(t)=t^n a$ for some $a\in A$. This proves $f(u)$ has the form as specified in (\ref{l:when-are-equiv-auts-Dnell-isomorphic::fzetaa}).

Next, recall the moduli interpretation for $\cD_n^{(\ell)}$ as $\mu_\ell$-torsors with equivariant maps to $D_m$. Thus, if $\sigma$ (resp.~$\sigma'$) is a $D_n$-automorphism of $\cD_n^{(\ell)}$ given by an equivariant automorphism $f$ (resp.~$f'$) of $D_m$, then $2$-isomorphisms $\sigma'\Rightarrow\sigma$ are given by isomorphisms $g\colon D_m\to D_m$ 
of torsors over $\cD_n^{(\ell)}$ such that $f'=fg$. Isomorphisms of $\mu_\ell$-torsors over $\cD_n^{(\ell)}$ are given by the action of an element $\xi\in\mu_\ell(\cD_n^{(\ell)})=\mu_\ell(D_n)=\mu_\ell(A)$.
%
Note that the action of $\xi\in\mu_\ell(A)$ is given by $u\mapsto\xi u$, so we see $f'(u)=f(\xi u)$.


From this description, (\ref{l:when-are-equiv-auts-Dnell-isomorphic::n-id-2auts}) as well as the ``furthermore'' sentence in (\ref{l:when-are-equiv-auts-Dnell-isomorphic::fzetaa}) follow easily. For the latter, note that when $n>0$, we know $f(u)=\zeta u+au^{\ell n+1}$, so choosing $\xi=\zeta^{-1}$, we may assume $\zeta=1$. For (\ref{l:when-are-equiv-auts-Dnell-isomorphic::n-id-2auts}), (the proof of) \autoref{prop:auts-Dnell} shows that every $\mu_\ell$-torsor on $\cD_n^{(\ell)}$ is isomorphic (as torsors) to $\rho$. Thus, $\alpha$ must be induced by an automorphism of $\rho$, i.e., an equivariant automorphism of $f$. Now note that (even when $n=0$) if $f(\xi u)=f(u)$, then $\xi=1$.

Lastly, to finish the proof of (\ref{l:when-are-equiv-auts-Dnell-isomorphic::fzetaa}), notice that the induced $D_{n-1}$-automorphism of $\cD_{n-1}^{(\ell)}$ is given by the reduction of $f(u)$ mod $t^n=u^{\ell n}$. Then simply observe that $u+au^{\ell n+1}=u$ modulo $u^{\ell n}$. 
\end{proof}

\begin{corollary}\label{cor:ResAut-alg-sp}
Let $\ell>1$, $n\geq0$, and $K$ be a field with characteristic prime to $\ell$. Then $\Res_{D_{n,K}/K}\uAut_{D_{n,K}}(\cD^{(\ell)}_{n,K})$ is an algebraic space.
\end{corollary}
\begin{proof}
We first show $\cR:=\Res_{D_{n,K}/K}\uAut_{D_{n,K}}(\cD^{(\ell)}_{n,K})$ is of finite type with separated diagonal. By \cite[Theorem 3.12 and Remark 3.13]{Rydh}, $\cH:=\uHom_{D_{n,K}}(\cD^{(\ell)}_{n,K})$ is of finite type and has separated diagonal. Since $\uAut_{D_{n,K}}(\cD^{(\ell)}_{n,K})$ is open in $\cH$ (see e.g.,~\cite[Proposition 4.4]{SatrianoUsatine4}) it is also finite type with separated diagonal; to see the finite type statement, note that finite type stacks over fields are Noetherian, hence every open substack is quasi-compact. Lastly, we apply \cite[Proposition 3.8]{Rydh}.

Then by \cite[Theorem 2.2.5(1)]{ConradKM} (which assumes all stacks have separated finite type diagonal), it suffices to show that the geometric points of $\cR$ have trivial automorphism groups. Fix an algebraically closed field $L$, let $\sigma$ be an $L$-valued point of $\cR$, and let $\alpha\colon\sigma\Rightarrow\sigma$ be a $2$-automorphism. \autoref{prop:auts-Dnell} tells us $\sigma$ comes from an equivariant automorphism of the standard \'etale cover of $\cD_{n,L}^{(\ell)}$. Hence, 
\autoref{l:when-are-equiv-auts-Dnell-isomorphic}(\ref{l:when-are-equiv-auts-Dnell-isomorphic::n-id-2auts}) shows $\alpha=\id$.
\end{proof}

\begin{corollary}\label{cor:all-auts-Dnell-reduce-to-id}
Fix $\ell>1$ and $n>0$. Let $K$ be a field with characteristic prime to $\ell$ that contains all $\ell$-th roots of unity, and let $S$ be a $K$-scheme. Let $\sigma$ be a $D_{n,S}$-automorphism of $\cD_{n,S}^{(\ell)}$ and let $\sigma':=\sigma\times_{D_{n,S}} D_{n-1,S}$ be the induced $D_{n-1,S}$-automorphism of $\cD_{n-1,S}^{(\ell)}$. Then there exists a unique $2$-isomorphism $\alpha\colon\sigma'\Rightarrow\id$.
\end{corollary}
\begin{proof}
Note that $\sigma$ is an $S$-point of $\Res_{D_{n,K}/K}\uAut_{D_{n,K}}(\cD^{(\ell)}_{n,K})$ and $\sigma'$ its image under the map $\Res_{D_{n,K}/K}\uAut_{D_{n,K}}(\cD^{(\ell)}_{n,K})\to\Res_{D_{n-1,K}/K}\uAut_{D_{n-1,K}}(\cD^{(\ell)}_{n-1,K})$. By \autoref{cor:ResAut-alg-sp}, $2$-isomorphisms in these spaces are unique when they exist. Thus, to prove $\sigma'$ is (uniquely) isomorphic to $\id$, it suffices to look \'etale locally on $S$, where by \autoref{prop:auts-Dnell}, we may assume that $S=\Spec A$ and that $\sigma$ is induced by an equivariant automorphism of $D_{m,A}$ (where $m=\ell(n+1)-1$). \autoref{l:when-are-equiv-auts-Dnell-isomorphic}(\ref{l:when-are-equiv-auts-Dnell-isomorphic::fzetaa}) then tells us that 
$\sigma'$ is isomorphic to $\id$.
\end{proof}

We turn now to the main results of this subsection.

\begin{proof}[{Proof of \autoref{prop:resaut-is-A1}}]
Let $J$ denote the ideal sheaf defining the closed immersion $D_{n,K}\hookrightarrow D_{n+1,K}$; then $J=\iota_*\iota^*J$ where $\iota\colon\Spec K=D_{0,K}\to D_{n,K}$ is the closed immersion. 

To prove the result, it suffices to consider the $A$-valued points for every $K$-scheme $A$. Let $S=\Spec A$ be a scheme, $p\colon S\to\Spec K$ the structure map,
$q\colon\cD_{0,A}^{(\ell)}\to\cD_{0,K}^{(\ell)}$ be the induced map, and let $\jmath\colon\cD_{0,A}^{(\ell)}\to\cD_{n,A}^{(\ell)}$ denote the base change of $\iota$. If $\cK$ is the ideal sheaf of $\jmath$, then by flat base change, it is the pullback of $J$ and $\cK=\jmath_*\jmath^*\cK$.

By \autoref{cor:all-auts-Dnell-reduce-to-id} and \cite[Theorem 1.5]{OlssonDefRep} (combined with \autoref{cor:ResAut-alg-sp} which shows $2$-isomorphisms are unique when they exist), the $A$-valued points of $\cR:=\Res_{D_{n,K}/K}\uAut_{D_{n,K}}(\cD^{(\ell)}_{n,K})$ are in canonical bijection with
\[
\Ext^0(L_{\cD^{(\ell)}_{n,A} / D_{n,A}},\cK)=\Ext^0(L_{\cD^{(\ell)}_{0,A} / A},\jmath^*\cK)
\]
where the equality uses that $L\jmath^*L_{\cD^{(\ell)}_{n,A} / D_{n,A}}=L_{\cD^{(\ell)}_{0,A} / A}$ since $\cD_{n,K}^{(\ell)}\to D_{n,K}$ is flat. 
Letting $\pi\colon\cD_{0,K}^{(\ell)}\to D_{0,K}$ be the coarse space map which is flat, we see $L_{\cD^{(\ell)}_{0,A} / A}=q^*L_{\cD^{(\ell)}_{0,K} / K}$, and so
\[
\Ext^0(L_{\cD^{(\ell)}_{0,A} / A},\jmath^*\cK)=\Ext^0(q^*L_{\cD^{(\ell)}_{0,K} / K},q^*\pi^*\iota^*J).
\]
Since $\cD^{(\ell)}_{0,A}$ is cohomologically affine, the global $\Ext^0$-group is given by global sections of the local $\Ext^0$-group. Then by flatness of $q$, we have
\[
\Ext^0(q^*L_{\cD^{(\ell)}_{0,K} / K},q^*\pi^*\iota^*J)=\Ext^0(L_{\cD^{(\ell)}_{0,K} / K},\pi^*\iota^*J)\otimes_K A;
\]
this can be seen, e.g., by applying \cite[Tag 0ATN]{stacks-project} to $D_{\ell-1,A}\to D_{\ell-1,K}$ and using that $\Ext^0(q^*L_{\cD^{(\ell)}_{0,K} / K},q^*\pi^*\iota^*J)$ is the $\mu_\ell$-invariants of $\Ext^0(\rho^*q^*L_{\cD^{(\ell)}_{0,K} / K},\rho^*q^*\pi^*\iota^*J)$, where $\rho\colon D_{\ell-1,K}\to\cD^{(\ell)}_{0,K}$ is the \'etale cover. 
This completes the proof.
\end{proof}

\begin{proof}[{Proof of \autoref{prop:AutD0-is-Gm}}]
We suppress the subscript $k$. Note that $\uAut_{D_0}(\cD^{(\ell)}_0)=\Res_{D_0/k}\uAut_{D_0}(\cD^{(\ell)}_0)$ is a group algebraic space by \autoref{cor:ResAut-alg-sp}. Consider the map
\[
F\colon\bG_m\to\uAut_{D_0}(\cD^{(\ell)}_0)
\]
sending $a\in A^*=\bG_m(A)$ to the automorphism of $\cD^{(\ell)}_{0,A}$ coming from the equivariant automorphism $f_a$ of $D_{\ell-1,A}=\Spec A[u]/u^\ell$ with $f_a(u)=au$. Note that $F$ is a homomorphism of group algebraic spaces.

\autoref{prop:auts-Dnell} shows that \'etale locally, every automorphism of $\cD^{(\ell)}_0$ is induced by an equivariant one. \autoref{l:when-are-equiv-auts-Dnell-isomorphic}(\ref{l:when-are-equiv-auts-Dnell-isomorphic::fzetaa}) then shows all equivariant automorphisms are of the form $f_a$ for $a\in A^*$. Hence, $F$ is a surjective map of sheaves on the \'etale topology. Furthermore, from our description of when equivariant automorphisms of $D_m$ yield isomorphic maps of $\cD_n^{(\ell)}$ given in the proof of \autoref{l:when-are-equiv-auts-Dnell-isomorphic}, we see
$f_a$ is $2$-isomorphic to $\id$ if and only if $a\in\mu_\ell(A)$. As a result, $F$ induces an injective map of sheaves
\[
\overline{F}\colon\bG_m\simeq\bG_m/\mu_\ell\to\uAut_{D_0}(\cD^{(\ell)}_0).
\]
Surjectivity of $\overline{F}$ follows from that of $F$, hence $\overline{F}$ is an isomorphism.
%
\end{proof}


\section{Relating measures for twisted and warped arcs}

The goal of this section is to prove \autoref{theoremTwistedToWarped} below, which relates the motivic measure $\nu_\cX$ on twisted arcs to the motivic measure $\mu_{\widetilde{\sW}(\cX)}$ on untwisted arcs of $\widetilde{\sW}(\cX)$, i.e., the motivic measure on (certain) warped arcs. We begin with the following minor generalization of \cite[Theorems 1.9(3) and 6.1]{SatrianoUsatine4}, which will allow us to consider the motivic measure $\mu_{\widetilde{\sW}(\cX)}$ in the desired level of generality.

\begin{proposition}\label{propositionQuasiAffineInertia}
Let $\cX$ be a locally finite type Artin stack over $k$ with affine diagonal, and assume that $\cX$ has a good moduli space. Then the inertia map $I_{\sW(\cX)} \to \sW(\cX)$ is quasi-affine. In particular, $\sW(\cX)$ has separated diagonal and affine geometric stabilizers.
\end{proposition}

\begin{proof}
Since quasi-affine algebraic groups over a field are affine by \cite[VIB.11.1]{SGA3}, $I_{\sW(\cX)} \to \sW(\cX)$ being quasi-affine would imply that $\sW(\cX)$ has affine geometric stabilizers. Since an Artin stack having separated inertia is equivalent to it having separated diagonal, $I_{\sW(\cX)} \to \sW(\cX)$ being quasi-affine would imply that $\sW(\cX)$ has separated diagonal. Thus we only need to show that $I_{\sW(\cX)} \to \sW(\cX)$ is quasi-affine.

Let $X$ be the good moduli space of $\cX$. Then $X$ is locally noetherian by \cite[Theorem 4.16(x)]{Alper}. Since it suffices to check Zariski-locally on $X$, we may assume that $X$ is quasi-compact by \autoref{lemmaQuasiCompactBasis}. Therefore there exists an \'{e}tale cover $X' \to X$ where $X'$ is a noetherian affine scheme. Since every open subset of $X'$ is quasi-compact, the map $X' \to X$ is quasi-separated. Therefore $X'$ being separated over $k$ and the valuative criterion for algebraic spaces together imply that $X' \to X$ is separated. Now let $\cX' = \cX \times_X X'$. Since $\cX' \to \cX$ is locally finite type, separated, and representable and $\cX \to \Spec(k)$ is locally finite type and has affine diagonal, we have that $\cX' \to \Spec(k)$ is locally finite type and has affine diagonal. Therefore $\sW(\cX')$ is a locally finite type Artin stack over $k$ \cite[Theorem 1.9(1)]{SatrianoUsatine4} with quasi-affine inertia \cite[Theorem 6.1]{SatrianoUsatine4}. Since $X'$ is an algebraic space, the natural map $I_{\sW(\cX')/X'} \to I_{\sW(\cX')}$ is an isomorphism, so $I_{\sW(\cX')/X'} \to \sW(\cX')$ is quasi-affine.

Now \cite[Proposition 5.2]{SatrianoUsatine4} implies that the diagram
\[
\xymatrix{
\sW(\cX')\ar[r]^{}\ar[d]_-{} & \sW(\cX)\ar[d]^-{}\\
X'\ar[r]^-{} & X
}
\]
is cartesian. Thus $I_{\sW(\cX)/X} \to \sW(\cX)$ is quasi-affine by \cite[Lemma 06PQ and 02L7]{stacks-project}. Because $X$ is an algebraic space, this implies that $\sW(\cX)$ has quasi-affine inertia, and we are done.
\end{proof}

Since $\widetilde{\sW}(\cX)$ may only be \emph{locally} finite type, we need to recall the following definition for the measure $\mu_{\widetilde{\sW}(\cX)}$.

\begin{definition}[{\cite[Definition 5.1]{SatrianoUsatine5}}]
If $\cY$ is a locally finite type Artin stack over $k$, we call a subset $\cE \subset |\sL(\cY)|$ a \emph{bounded\footnote{We used the terminology ``small'' in that paper, but we have since decided that ``bounded'' is more appropriate.} cylinder} if there exists some $n \in \Z_{\geq 0}$ and quasi-compact locally constructible subset $\cE_n \subset |\sL_n(\cY)|$ such that $\cE = \theta_n^{-1}(\cE_n)$. In the special case where $\cY$ is equidimensional and smooth over $k$ and has affine geometric stabilizers, the bounded cylinder $\cE$ has a \emph{motivic volume} defined by
\[
	\mu_\cY(\cE) = \lim_{m \to \infty} \e(\theta_m(\cE))\bL^{-(m+1)\dim\cY} = \e(\cE_n)\bL^{-(n+1)\dim\cY}.
\]
Measurable sets and their motivic volumes are defined in terms of bounded cylinders and their motivic volumes exactly as in \autoref{definitionMeasurableSetOfTwistedArcs}. Alternatively, see \cite[Definition 5.5]{SatrianoUsatine5} in the special case where $d = \dim\cY$.
\end{definition}

\begin{remark}\label{remarkTildeWarpedHasWellDefinedMeasure}
By \autoref{remarkSmoothnessAndDimensionOfClosureInLCILocus} and \autoref{propositionQuasiAffineInertia}, $\mu_{\widetilde{\sW}(\cX)}(\cE)$ is well defined for any bounded cylinder $\cE \subset |\sL(\widetilde{\sW}(\cX))|$.
\end{remark}

\begin{remark}
Let $\ell \in \Z_{\geq 0}$. Since $\cD^\ell \to D$ is syntomic and an isomorphism over the generic point of $D$, we have that for all $n \in \Z_{> 0}$ the map $\sJ^\ell_n(\cX) \to \sL_n(\sW(\cX))$ factors through $\sL_n(\widetilde{\sW}(\cX))$. In particular, $\sJ^\ell(\cX) \to \sL(\sW(\cX))$ factors through $\sL(\widetilde{\sW}(\cX))$.
\end{remark}

\begin{theorem}\label{theoremTwistedToWarped}
Let $\cX$ be an equidimensional smooth finite type Artin stack over $k$ with affine diagonal, and assume that $\cX$ has a good moduli space. Let $\ell \in \Z_{>0}$, let $\cC \subset |\sJ^\ell(\cX)|$, and let $\cE \subset |\sL(\widetilde{\sW}(\cX))|$ denote its image along $\sJ^\ell(\cX) \to \sL(\widetilde{\sW}(\cX))$
\begin{enumerate}[label=(\alph*)]

\item For any field extension $k'$ of $k$, the induced map $\overline{\cC}(k') \to \overline{\cE}(k')$ is a bijection.

\item The set $\cC$ is a cylinder if and only if $\cE$ is a cylinder if and only if $\cE$ is a bounded cylinder, and in that case
\[
	\nu_\cX(\cC) = \begin{cases}  \mu_{\widetilde{\sW}(\cX)}(\cE), & \ell =1\\ \bL \mu_{\widetilde{\sW}(\cX)}(\cE), & \ell > 1 \end{cases}.
\]

\item The set $\cC$ is measurable if and only if $\cE$ is measurable, and in that case
\[
	\nu_\cX(\cC) = \begin{cases}  \mu_{\widetilde{\sW}(\cX)}(\cE), & \ell =1\\ \bL \mu_{\widetilde{\sW}(\cX)}(\cE), & \ell > 1 \end{cases}.
\]

\end{enumerate}
\end{theorem}

\begin{proof}
\begin{enumerate}[label=(\alph*)]

\item We will first show that $\overline{\sJ^\ell(\cX)}(k') \to \overline{\sL(\sW(\cX))}(k')$ is injective. Let $\varphi, \varphi': \cD^\ell_{k'} \to \cX$ be twisted arcs whose images in $\sL(\sW(\cX))(k')$ are isomorphic. Then by the definition of the category $\sW(\cX)(k'\llbracket t \rrbracket)$, see e.g., \cite[Remark 1.4]{SatrianoUsatine4}, there exists an isomorphism $\phi: \cD^\ell_{k'} \xrightarrow{\sim} \cD^\ell_{k'}$ over $k'\llbracket t \rrbracket$ and a 2-isomorphism $\alpha: \varphi \circ \phi \Rightarrow \varphi'$. Because $\phi$ is over $k'\llbracket t \rrbracket$ and $\cD^\ell \to \Spec(k'\llbracket t \rrbracket)$ is an isomorphism over $\Spec(k'\llparenthesis t \rrparenthesis)$, we have that $\phi$ is generically 2-isomorphic to the identity. Therefore $\phi$ is 2-isomorphic to the identity by \cite[Proposition A.1]{FantechiMannNironi}, so $\varphi$ and $\varphi'$ are 2-isomorphic. Therefore $\overline{\sJ^\ell(\cX)}(k') \to \overline{\sL(\sW(\cX))}(k')$ is injective, so its restriction $\overline{\cC}(k') \to \overline{\cE}(k')$ is injective.

Now let $(\pi: \cD \to \Spec(k'\llbracket t \rrbracket), \psi: \cD \to \cX)$ be a warped arc in $\cE(k')$. We will first show that there exists an isomorphism $\phi: \cD^\ell_{k'} \xrightarrow{\sim} \cD$ over $k'\llbracket t \rrbracket$. Set $\cD_2 = \cD \otimes_{k'\llbracket t \rrbracket} k'[t]/(t^{3})$, and let $\Spec(k') \to \sL_2(\sW(\cX))$ be the $k'$-point corresponding to the warped jet $(\cD_2 \to \Spec(k'[t]/(t^3)), \cD_2 \hookrightarrow \cD \xrightarrow{\psi} \cX)$. Since $\cE$ is the image of $\cC$, the fiber of $\sJ^\ell_{2}(\cX) \to \sL_2(\sW(\cX))$ over $\Spec(k') \to \sL_2(\sW(\cX))$ is nonempty. Since this fiber is finite type over $k'$, it contains a $k''$ point for some finite field extension $k''$ of $k'$. Thus by \autoref{l:fibers-twisted->warped}, there exists an isomorphism $\cD^\ell_{2, k''} \xrightarrow{\sim} \cD_2 \otimes_{k'[t]/(t^3)} k''[t]/(t^3)$ over $k''[t]/(t^3)$. The existence of the desired isomorphism $\phi: \cD^\ell_{k'} \xrightarrow{\sim} \cD$ then follows from \autoref{prop:no-non-trivial-forms-truncated-twisted} and \autoref{l:coh-complete-lift-->twisted-disc}. Now consider the twisted arc
\[
	\varphi = \psi \circ \phi: \cD^\ell_{k'} \to \cX.
\]
Then the warped arc $(\cD^\ell_{k'} \to \Spec(k'\llbracket t \rrbracket), \varphi)$ is isomorphic to the warped arc $(\pi, \psi)$, so we only need to show that $\varphi$ is in $\cC(k')$. Since $\cE$ is the image of $\cC$, there exists some field extension $k'''$ of $k'$ and a twisted arc $\varphi': \cD^\ell_{k'''} \to \cX$ in $\cC(k''')$ whose image in $\sL(\sW(\cX))(k''')$ is isomorphic to $(\pi \otimes_{k'\llbracket t \rrbracket} k'''\llbracket t \rrbracket, \cD \otimes_{k'\llbracket t \rrbracket} k'''\llbracket t \rrbracket \to \cD \xrightarrow{\psi} \cX)$. Thus $\varphi'$ is 2-isomorphic to $\varphi \otimes_{k'\llbracket t \rrbracket} k'''\llbracket t \rrbracket$ by the injectivity of $\overline{\sJ^\ell(\cX)}(k''') \to \overline{\sL(\sW(\cX))}(k''')$, and therefore $\varphi \in \cC(k')$ as desired.

\item The case where $\ell = 1$ follows directly from \cite[Theorem 1.9(2)]{SatrianoUsatine4}, so we assume $\ell > 1$.

We will first assume that $\cE$ is a cylinder and show that $\cC$ is a cylinder. Let $\cE_n$ be a locally constructible subset of $|\sL_n(\widetilde{\sW}(\cX))|$ such that $\cE = \theta_n^{-1}(\cE_n)$, and set $\cC_n \subset |\sJ^\ell_n(\cX)|$ to be the peimage of $\cE_n$. We have already shown that $\sJ^\ell(\cX)(k') \to \sL(\sW(\cX))(k')$ is injective for all fields $k'$. Thus $\cC$ is the preimage of $\cE$, so $\cC = \theta_n^{-1}(\cC_n)$ is a cylinder.

We will now assume that $\cC$ is a cylinder and show that $\cE$ is a bounded cylinder. Let $\cC_n$ be a constructible subset of $|\sJ^\ell_n(\cX)|$ such that $\cC = \theta_n^{-1}(\cC_n)$. By possibly replacing with a preimage along $\theta^1_0$, we may assume $n \geq 1$. Set $\cC_{n+1} = (\theta^{n+1}_n)^{-1}(\cC_n) \subset |\sJ^\ell_{n+1}(\cX)|$, and let $\cE_{n+1}$ be the image of $\cC_{n+1}$ in $|\sL_{n+1}(\widetilde{\sW}(\cX))|$. Then by \autoref{remarkJetStackLocallyFiniteTypeAffineDiagonalFiniteType} and Chevalley's theorem for Artin stacks \cite[Theorem 5.1]{HallRydh}, $\cE_{n+1}$ is a quasi-compact locally constructible subset of $|\sL_{n+1}(\widetilde{\sW}(\cX))|$. It is therefore sufficient to show that $\cE = \theta_{n+1}^{-1}(\cE_{n+1})$. Clearly $\cE \subset \theta_{n+1}^{-1}(\cE_{n+1})$. For the reverse inclusion, suppose that $(\pi: \cD \to \Spec(k'\llbracket t \rrbracket), \psi: \cD \to \cX)$ is a warped arc in $\theta_{n+1}^{-1}(\cE_{n+1})$, and set $\cD_{n+1} = \cD \otimes_{k'\llbracket t \rrbracket} k'[t]/(t^{n+2})$, $\pi_{n+1} = \pi \otimes_{k'\llbracket t \rrbracket} k'[t]/(t^{n+2})$, and $\psi_{n+1} = \psi \otimes_{k'\llbracket t \rrbracket} k'[t]/(t^{n+2})$. After possibly replacing $k'$ with a field extension, there exists a twisted jet $\varphi_{n+1}: \cD^\ell_{n+1,k'} \to \cX$ in $\cC_{n+1}$ such that the warped jet $(\pi_{n+1}, \psi_{n+1})$ is isomorphic to the warped jet $(\cD^\ell_{n+1,k'} \to \Spec(k'[t]/(t^{n+2})), \varphi_{n+1})$. Therefore there exists an isomorphism $\phi_{n+1}: \cD^\ell_{n+1,k'} \xrightarrow{\sim} \cD_{n+1}$ over $k'[t]/(t^{n+2})$ and a 2-isomorphism $\alpha_{n+1}: \psi_{n+1} \circ \phi_{n+1} \Rightarrow \varphi_{n+1}$. Noting that $n+1 \geq 2$, \autoref{l:coh-complete-lift-->twisted-disc} implies that there exists an isomorphism $\phi': \cD^\ell_{k'} \xrightarrow{\sim} \cD$ over $k'\llbracket t \rrbracket$ such that $\phi' \otimes_{k'\llbracket t \rrbracket} k'[t]/(t^{n+1}) \cong \phi_{n+1} \otimes_{k'[t]/(t^{n+2})} k'[t]/(t^{n+1})$. Now consider the twisted arc
\[
	\varphi' = \psi \circ \phi': \cD^\ell_{k'} \to \cX.
\]
Then the warped arc $(\cD^\ell_{k'} \to \Spec(k'\llbracket t \rrbracket), \varphi')$ is isomorphic to $(\pi, \psi)$, so it is sufficient to prove that $\varphi'$ is in $\cC(k')$. We have
\begin{align*}
	\varphi' &\otimes_{k'\llbracket t \rrbracket} k'[t]/(t^{n+1}) \\
	&= \left(\psi \otimes_{k'\llbracket t \rrbracket} k'[t]/(t^{n+1})\right) \circ \left(\phi' \otimes_{k'\llbracket t \rrbracket} k'[t]/(t^{n+1})\right) \\
	 &\cong\left(\psi \otimes_{k'\llbracket t \rrbracket} k'[t]/(t^{n+1})\right) \circ \left(\phi_{n+1} \otimes_{k'[t]/(t^{n+2})} k'[t]/(t^{n+1})\right) \\
	 &= \left(\psi_{n+1} \circ \phi_{n+1}\right) \otimes_{k'[t]/(t^{n+2})} k'[t]/(t^{n+1})\\
	 &\cong \varphi_{n+1} \otimes_{k'[t]/(t^{n+2})} k'[t]/(t^{n+1}),
\end{align*}
so $\varphi' \otimes_{k'\llbracket t \rrbracket} k'[t]/(t^{n+1})$ is in $\cC_n(k')$. Therefore $\varphi'$ is in $\cC(k')$, and we have finished showing that $\cE$ is a bounded cylinder.

All that remains is to assume that $\cC$ is a cylinder and $\cE$ is a bounded cylinder and prove that
\[
	\nu_\cX(\cC) = \bL\mu_{\widetilde{\sW}(\cX)}(\cE).
\]
Let $\cE_n$ be a quasi-compact locally constructible subset of $|\sL_n(\widetilde{\sW}(\cX))|$ such that $\cE = \theta_n^{-1}(\cE_n)$, and set $\cC_n \subset |\sJ^\ell_n(\cX)|$ to be the peimage of $\cE_n$. We have already seen that $\cC = \theta_n^{-1}(\cC_n)$. By \autoref{remarkSurjectiveTruncationMorphisms}, $\cC_n$ is the image of $\cC$. Since $\widetilde{\sW}(\cX)$ is smooth by \autoref{remarkSmoothnessAndDimensionOfClosureInLCILocus}, we also have $\cE_n$ is the image of $\cE$, so $\cE_n$ is the image of $\cC_n$. Therefore by \autoref{propFibrationClass} and \autoref{thm:fibers-twisted->warped-aut-explicit},
\[
	\e(\cC_n) = \bL\e(\cE_n).
\]
Finally,
\[
	\nu_\cX(\cC) = \e(\cC_n)\bL^{-(n+1)\dim\cX} = \bL\e(\cE_n)\bL^{-(n+1)\dim\cX} =  \bL\mu_{\widetilde{\sW}(\cX)}(\cE),
\]
where the last equality follows from \autoref{remarkTildeWarpedHasWellDefinedMeasure}.

\item This follows immediately from the definitions and the previous parts.\qedhere
\end{enumerate}
\end{proof}

\section{Thin subsets of twisted arcs}

As an application of \autoref{theoremTwistedToWarped}, we show the measure of twisted arcs factoring through a non-trivial closed substack vanishes. This is used in the proof of \autoref{maintheorem} to discard certain subsets of arcs in the computation of our motivic integrals.

\begin{theorem}\label{theoremThinSubsetOfTwistedJets}
Let $\ell \in \Z_{\geq 0}$, let $\cX$ be an equidimensional smooth finite type Artin stack over $k$ with affine diagonal, and assume that $\cX$ has a good moduli space. If $\cZ\subset\cX$ is a closed substack with $\dim\cZ<\dim\cX$, then $|\cJ^\ell(\cZ)|$ is a measurable subset of $|\cJ^\ell(\cX)|$, and
\[
	\nu_{\cX}(|\cJ^\ell(\cZ)|) = 0.
\]
\end{theorem}

We prove \autoref{theoremThinSubsetOfTwistedJets} by applying \autoref{theoremTwistedToWarped} to relate our motivic measure of twisted arcs to the motivic measure of certain untwisted arcs of the warping stack $\sW(\cX)$. Thus, our first goal is to prove an untwisted version of \autoref{theoremThinSubsetOfTwistedJets}:

\begin{theorem}\label{theoremThinSubsetsAreNegligible}
Let $\cX$ be an equidimensional smooth finite type Artin stack over $k$ with affine geometric stabilizers. If $\cZ\subset\cX$ is a closed substack with $\dim \cZ < \dim \cX$, then $|\sL(\cZ)|$ is a measurable subset of $|\sL(\cX)|$, and
\[
	\mu_\cX(|\sL(\cZ)|) = 0.
\]
\end{theorem}

We turn to the proof of \autoref{theoremThinSubsetsAreNegligible} after a couple preliminary results.

\begin{lemma}\label{lemmaMeasure0SameAsLimitOfDifferences}
Let $\cX$ be an equidimensional smooth finite type Artin stack over $k$ with affine geometric stabilizers, and let $\cZ$ be a closed substack of $\cX$. Then $|\sL(\cZ)|$ is a measurable subset of $|\sL(\cX)|$ with measure 0 if and only if
\[
	\lim_{n\to\infty}(\dim\sL_n(\cZ) - \dim\sL_n(\cX))=-\infty.
\]
\end{lemma}

\begin{proof}
Consider the cylinders $\cC^{(n)} = \theta_n^{-1}(\sL_n(\cZ))$ for all $n \in \Z_{\geq 0}$. Then $\cC^{(n)} \supset \cC^{(n+1)}$ and $|\sL(\cZ)|=\bigcap_{n\geq0}\cC^{(n)}$, so $|\sL(\cZ)|$ is measureable with measure $0$ if and only if 
\[
	\lim_{n\to\infty}\mu_\cX(\cC^{(n)})=0
\]
if and only if 
\[
	\lim_{n\to\infty}\Vert\mu_\cX(\cC^{(n)})\Vert=0.
\] 
By \cite[Proposition 4.4]{SatrianoUsatine5} and the fact that $\mu_\cX(\cC^{(n)})=\e(\sL_n(\cZ))\bL^{-(n+1)\dim\cX}$, this is equivalent to

\[
\lim_{n\to\infty}(\dim\sL_n(\cZ) - (n+1)\dim\cX)=-\infty.
\]
Furthermore \cite[Lemmas 3.34 and 3.35]{SatrianoUsatine} imply $\e(\sL_n(\cX)) = \e(\cX)\bL^{(n+1)\dim\cX}$, so \cite[Proposition 4.4]{SatrianoUsatine5} implies
\[
	\dim\sL_n(\cX) = (n+1)\dim\cX,
\]
and we are done.
\end{proof}

The next proposition, along with the criterion proved in \autoref{lemmaMeasure0SameAsLimitOfDifferences}, will  allow us to recover \autoref{theoremThinSubsetsAreNegligible} from the well-known scheme case.

\begin{proposition}\label{prop:sm-cover-const-rel-dim-preserved}
Let $\cZ$ be an Artin stack, $V$ be a scheme, and $\rho\colon V\to\cZ$ be a smooth cover of constant relative dimension $r$. Then $\sL_n(\rho)\colon\sL_n(V)\to\sL_n(\cZ)$ is a smooth cover of constant relative dimension $(n+1)r$.
\end{proposition}
\begin{proof}
By \cite[Proposition 3.5(v)]{Rydh}, $\sL_n(\rho)$ is a smooth cover, so we need only show it is of constant relative dimension $(n+1)r$. We prove this by induction on $n$. When $n=0$, this is simply the hypothesis that $\rho$ is a smooth cover of constant relative dimension $r$.

For $n>0$, let $\Spec K\to\sL_n(\cZ)$ be a field-valued point; note we have an induced map $\Spec K\to\sL_n(\cZ)\to\sL_{n-1}(\cZ)$. Let $\cF_n=\sL_n(V)\times_{\sL_n(\cZ)} K$ and $\cF_{n-1}=\sL_{n-1}(V)\times_{\sL_{n-1}(\cZ)} K$. By induction, $\dim\cF_{n-1}=rn$, so it suffices to show that $\cF_n\to\cF_{n-1}$ has fibers of constant relative dimension $r$.

Let $\Spec L\to\cF_{n-1}$ be a field-valued point. This corresponds to a diagram
\[
\xymatrix{
D_{n-1,L}\ar@{^{(}->}[r]\ar[d]\ar@/^1.1pc/[rr]^-{\beta} & D_{n,L}\ar[d]^-{p} & V\ar[d]^-{\rho}\\
D_{n-1,K}\ar@{^{(}->}[r] & D_{n,K}\ar[r]^-{\alpha} & \cZ
}
\]
where the square is cartesian. Letting $\beta_0\colon \Spec L=D_{0,L}\to D_{n-1,L}\xrightarrow{\beta} V$ be the induced map, we show that $\cF_n\times_{\cF_{n-1}}L$ is the affine space given by
\[
H^0(\beta_0^*T_\rho)\otimes_L (t^n)/(t^{n+1});
\]
in particular, it has relative dimension $r$ since the tangent bundle $T_\rho$ is locally free of rank $r$.

To show this, let $A$ be an $L$-algebra. Giving a map $\Spec A\to\cF_n\times_{\cF_{n-1}}L$ is equivalent to choosing a dotted arrow
\[
\xymatrix{
D_{n-1,A}\ar@{^{(}->}[r]^-{\iota}\ar[d]_-{q'} & D_{n,A}\ar[d]^-{q}\ar@{-->}[dr]^-{\gamma} & \\
D_{n-1,L}\ar@{^{(}->}[r]\ar@/_1.1pc/[rr]_-{\beta} & D_{n,L} & V
}
\]
where $\rho\gamma=\alpha pq$ and $\gamma\iota=\beta q'$. Rewriting this condition, we see it is equivalent to giving a dotted arrow making the diagram
\[
\xymatrix{
D_{n-1,A}\ar@{^{(}->}[d]_-{\iota}\ar[r]^-{\beta q'} & V\ar[d]^-{\rho}\\
D_{n,A}\ar[r]^-{\alpha pq}\ar@{-->}[ur]^-{\gamma} & \cZ
}
\]
commute. Since $\rho$ is smooth, such a $\gamma$ exists; in particular, we may fix a choice $\gamma_L\colon D_{n,L}\to V$ (by considering the case $A=L$) which then induces a map $\gamma_A\colon D_{n,A}\to V$ for all $L$-algebras $A$. With this choice, the set of all possible $\gamma\colon D_{n,A}\to V$ as above is identified with
\[
\Hom(q^*\beta^*\Omega^1_\rho, t^nA/t^{n+1}A)=(H^0(\beta_0^*T_\rho)\otimes_L t^nL/t^{n+1}L) \otimes_L A.
\]
This proves the result.
\end{proof}

\begin{proof}[{Proof of \autoref{theoremThinSubsetsAreNegligible}}]
Let $\rho\colon V\to\cX$ be a smooth cover by a scheme. Since $V$ is finite type and smooth, there are finitely many irreducible components $V_i$ of $V$, and the $V_i$ are also the connected components of $V$. Let $d_i=\dim V_i$, and let $d=\max d_i$. Then replacing $V$ with $\coprod_i (V_i\times\bA^{d-d_i})$, we may assume that $\rho$ is smooth, finite type, and has constant relative dimension $r$. Let $W = V \times_\cX \cZ$, so that $W\subset V$ is closed and $p\colon W\to\cZ$ is a finite type smooth cover of constant relative dimension $r$. By \autoref{prop:sm-cover-const-rel-dim-preserved}, $\sL_n(\rho)$ and $\sL_n(p)$ are smooth covers of constant relative dimension $(n+1)r$, so
\[
\dim\sL_n(\cZ)-\dim\sL_n(\cX)=\dim\sL_n(W)-\dim\sL_n(V).
\]
The theorem now follows from \autoref{lemmaMeasure0SameAsLimitOfDifferences} and the well-known case where $\cX$ is a scheme applied to $V$.
\end{proof}

As a consequence, we may now prove \autoref{theoremThinSubsetOfTwistedJets}.

\begin{proof}[{Proof of \autoref{theoremThinSubsetOfTwistedJets}}]
Let $\cE$ be the image of $|\cJ^\ell(\cZ)|$ in $|\sL(\widetilde{\sW}(\cX))|$. Then by \autoref{theoremTwistedToWarped}, it suffices to show that $\cE$ is measurable with measure $0$. Let $\cU$ be an open substack of $\widetilde{\sW}(\cX)$ that is finite type over $k$ and such that $|\cU|$ contains $|\cX|$ and the image of $|\sJ^\ell_0(\cX)|$ in $|\widetilde{\sW}(\cX)|$. Then the image of $|\sJ^\ell(\cX)|$ in $|\sL(\widetilde{\sW}(\cX))|$ is contained in $|\sL(\cU)|$. In particular, $\cE \subset |\sL(\cU)|$ and it suffices to show that $\cE$ is measurable in $|\sL(\cU)|$ and $\mu_\cU(\cE) = 0$.

Let $\cZ'$ be a closed substack of $\cU$ such that $|\cZ'|$ is the closure of $|\cZ|$ in $|\cU|$. Since 
\[
\dim\cZ'=\dim\cZ<\dim\cX= \dim\widetilde{\sW}(\cX) = \dim\cU,
\]
\autoref{theoremThinSubsetsAreNegligible} implies that $|\sL(\cZ')|$ is measurable in $|\sL(\cU)|$ and $\mu_\cU(|\sL(\cZ')|) = 0$. Therefore it is sufficient to show that
\[
	\cE \subset |\sL(\cZ')|.
\]
Let $\Spec(k'\llbracket t \rrbracket) \to \cU$ represent a point of $\cE$. Then possibly after extending the field $k'$, this untwisted arc of $\sW(\cX)$ is 2-isomorphic to a warped arc of the form $(\cD^\ell_{k'} \to \Spec(k'\llbracket t \rrbracket), \varphi: \cD^\ell_{k'} \to \cX)$ with $\varphi$ factoring through $\cZ$. Since $\cD^\ell_{k'} \to \Spec(k'\llbracket t \rrbracket)$ is an isomorphism over the generic point, the composition $\Spec(k'\llparenthesis t \rrparenthesis)\to \Spec(k'\llbracket t \rrbracket) \to \cU$ factors through $\cZ$. Therefore $\Spec(k'\llbracket t \rrbracket) \to \cU$ factors through $\cZ'$, and we are done.
\end{proof}

\section{Moduli interpretation for the warping stack height}
\label{appendix:htWY-using-X-data}

The results of this section concern warped arcs to singular stacks, as opposed to twisted arcs to smooth stacks (which have been the primary focus of this paper). Throughout, $k$ denotes an algebraically closed field of characteristic 0. We use the notation
\[
\cZ_T:=\cZ\times_k T
\]
to denote the base change of a $k$-stack $\cZ$ by a morphism $T\to\Spec k$.

Our motivic change of variables formula for warped arcs \cite[Theorem 1.18]{SatrianoUsatine4} involves the relative height function $\het_{\sW(\cX)/Y}$ whose definition we recall below. Since $\sW(\cX)$ is a moduli stack, it is desirable to have a moduli-theoretic interpretation for $\het_{\sW(\cX)/Y}$.

We give such an interpretation in \autoref{thm:moduli-interp-hetWY}, which we state after recalling the definitions of the height functions and relative height functions.

\begin{definition}[{\cite[Definition 6.2]{SatrianoUsatine5}}]
\label{def:relhet}
Let $q\colon\cZ \to \cY$ be a morphism of locally finite type Artin stacks over $k$. If $\psi\colon\Spec k'[[t]]\to\cZ$ is an arc with $k'$ a field, and if $\het^{(1)}_{L_{\cZ/\cY}}(\psi) \neq \infty$, we let
\begin{align*}
	\het_{\cZ/\cY}(\psi) = &\het^{(0)}_{L_{\cZ/\cY}}(\psi) - \het^{(1)}_{L_{\cZ/\cY}}(\psi)\\
	&- \dim_{k'}\coker\left(H^0(L\psi^*Lq^*L_\cY)_{\tors} \to H^0(L\psi^* L_\cZ)_{\tors}\right),
\end{align*}
\end{definition}

\begin{theorem}\label{thm:moduli-interp-hetWY}
Let $p\colon\cX\to Y$ be a morphism over $k$ with $Y$ an algebraic space and $\cX$ a locally finite type Artin stack with affine diagonal. Assume $p$ is weakly birational, i.e., there exists a dense open substack $\cU\subset\cX$ such that $\cU\to Y$ is an open immersion. 
Let $k'$ be field, $D=\Spec k'[[t]]$, 
\[
\psi\colon D\to\sW(\cX)
\]
be an arc corresponding to a warped arc $\gamma=(\cD\xrightarrow{\pi} D,\cD\xrightarrow{\varphi}\cX)$. 
If the generic fiber of $\pi$ factors through $\cU$, then 
\[
\het_{\sW(\cX)/Y}(\psi)=\dim_{k'}\coker\pi_*\Lambda_\gamma-\dim_{k'}
\ker\pi_*\Lambda_\gamma
\]
where $\Lambda_\gamma$ is the canonical map
\[
\mathcal{E}xt^1(L_{\cD/\cX_D},\cO)\xrightarrow{\Lambda_\gamma} \mathcal{E}xt^1(L\pi^*L_{D/Y_D},\cO)\simeq\mathcal{E}xt^0(L(\overline{\varphi}\pi)^*L_Y,\cO),
\]
and $\overline{\varphi}\colon D\to Y$ is the arc induced by $\gamma$.
\end{theorem}

\begin{remark}\label{rmk:Lambda-simplify-target}
The canonical isomorphism $\mathcal{E}xt^1(L\pi^*L_{D/Y_D},\cO)\simeq\mathcal{E}xt^0(L(\overline{\varphi}\pi)^*L_Y,\cO)$ arises as follows. The map $\overline{\varphi}\colon D\to Y$ induces a map $\alpha\colon D\to Y_D$ and hence a transitivity triangle
\[
L\alpha^*L_{Y_D/D}\to L_{D/D}\to L_{D/Y_D}.
\]
Since $L_{D/D}=0$, we see $L_{D/Y_D}[-1]\xrightarrow{\simeq}L\alpha^*L_{Y_D/D}$. Since $Y$ is flat over $k$, $L_{Y_D/D}$ is the pullback of $L_Y$, and hence $L\alpha^*L_{Y_D/D}\simeq L\overline{\varphi}^*L_Y$.
\end{remark}


The next result relates height functions to $\Ext$-groups.

\begin{lemma}\label{l:ext-vs-height}
Let $\cY$ be a finite type Artin stack over $k$ and let $E\in D^-_{\coh}(\cY)$. Let $\beta\colon D\to\cY$ be an arc with $D=\Spec k'[[t]]$ and $k'$ a field. Then there is a canonical isomorphism
\begin{equation}\label{eqn:Extn-tors}
\Ext^1(L^{1-n}\beta^*E,\cO)\xrightarrow{\simeq} \Ext^n(L\beta^*E,\cO)_{\tors}.
\end{equation}
In particular, if $L^{1-n}\beta^*E$ is torsion, then 
\begin{equation}\label{eqn:Extn-tors-height}
\het^{(1-n)}_E(\beta) = \dim_{k'} \Ext^n(L\beta^*E,\cO)_{\tors}
\end{equation}
and if $L^{-n}\beta^*E$ is also torsion, then $\Ext^n(L\beta^*E,\cO)=\Ext^n(L\beta^*E,\cO)_{\tors}$.

Furthermore, if $\beta$ factors through the standard cover $D\to\cD^\ell_{k'}$ of an $\ell$-th twisted disc, then \eqref{eqn:Extn-tors} holds as graded $\bZ/\ell$-modules and we have a short exact sequence of graded modules
\begin{equation}\label{eqn:Extn-tors-ss-degen}
0\to \Ext^1(L^{1-n}\beta^*E,\cO)\to \Ext^n(L\beta^*E,\cO)\to \Hom(L^{-n}\beta^*E,\cO)\to 0.
\end{equation}
\end{lemma}
\begin{proof}
Since $k'[[t]]$ is a PID, the spectral sequence (which consists of graded maps if $k'[[t]]$ has the aforementioned $\bZ/\ell$-grading)
\[
E_2^{p,q}=\Ext^p(L^{-q}\beta^*E,\cO)\Longrightarrow \Ext^{p+q}(L\beta^*E,\cO)
\]
degenerates, yielding the short exact sequence \eqref{eqn:Extn-tors-ss-degen}. 
This yields an injection 
\[
	\Ext^1(L^{1-n}\beta^*E,\cO)_{\tors}\hookrightarrow \Ext^n(L\beta^*E,\cO)_{\tors}.
\] 
This map is an isomorphism since $\cO$ is torsion-free, hence $\Hom(L^{-n}\beta^*E,\cO)$ is as well. Since $\Ext^1(\cF,\cO)_{\tors}=\Ext^1(\cF,\cO)$ for any coherent sheaf $\cF$, we have proved that the map in \eqref{eqn:Extn-tors} is an isomorphism. 

For any torsion sheaf $k'[[t]]$-module, $\dim M=\dim\Ext^1(M,k'[[t]])$. So, \eqref{eqn:Extn-tors-height} follows by taking $k'$-dimensions in \eqref{eqn:Extn-tors}. Lastly, if $L^{-n}\beta^*E$ is torsion, then $\Hom(L^{-n}\beta^*E,\cO)$ vanishes, showing that $\Ext^n(L\beta^*E,\cO)= \Ext^1(L^{1-n}\beta^*E,\cO)$ is torsion.
\end{proof}

\begin{lemma}\label{l:ext1-ker-vs-coker}
Let $R$ be a PID and $f\colon M\to N$ be a morphism of torsion $R$-modules. Then
\[
\len\coker(f)=\len\ker\left(\Ext^1(N,R)\to\Ext^1(M,R)\right).
\]
\end{lemma}
\begin{proof}
Let $K=\ker(f)$, $Q=\coker(f)$, and $P=\textrm{im}(f)$. Then we have short exact sequences
\[
0\to K\to M\to P\to 0\quad\textrm{and}\quad 0\to P\to N\to Q\to 0.
\]
Since all modules are torsion, we obtain 
short exact sequences
\[
0\to \Ext^1(P,R)\to \Ext^1(M,R)\to \Ext^1(K,R)\to 0,
\]
\[
0\to \Ext^1(Q,R)\to \Ext^1(N,R)\xrightarrow{g} \Ext^1(P,R)\to 0,
\]
from which we see that the kernel of $\Ext^1(N,R)\to\Ext^1(M,R)$ equals $\ker(g)$, i.e., $\Ext^1(Q,R)$. As $Q$ is torsion, $\len\Ext^1(Q,R)=\len Q$.
\end{proof}

\begin{proposition}\label{prop:ExtWY-equals-ExtDX}
Use the notation and hypotheses of \autoref{thm:moduli-interp-hetWY} however one may drop the assumption that $p$ is weakly birational. We have a commutative diagram
\[
\xymatrix{
\Ext^0(L\psi^*L_{\sW(\cX)},\cO)\ar[r]\ar[d]_-{\simeq} & \Ext^0(L\overline{\varphi}^*L_Y,\cO)\ar[d]^-{\simeq}\\
\Ext^1(L_{\cD/\cX_D},\cO)\ar[r]^-{\pi_*\Lambda_\gamma} & \Ext^1(L_{D/Y_D},\cO)
}
\]
\end{proposition}
\begin{proof}
We first show the target space of $\pi_*\Lambda_\gamma$ is naturally isomorphic to $\Ext^1(L_{D/Y_D},\cO)$. Since $\pi$ is a good moduli space, $\pi_*$ is exact, so $\pi_*\mathcal{E}xt^1(E,\cO)=\Ext^1(E,\cO)$ for any $E\in D^-_{\coh}(\cX)$. Next, using adjunction combined with the fact that $\pi_*\cO=\cO$, we see
\[
\pi_*\mathcal{E}xt^1(L\pi^*L_{D/Y_D},\cO)\simeq
\Ext^1(L_{D/Y_D},\pi_*\cO)\simeq 
\Ext^1(L_{D/Y_D},\cO).
\]

Fix $D'=D\times_{k'} D_1$. Since the closed immersion $D\to D'$ has a section, we have a canonical extension of $\psi\colon D\to\sW(\cX)$ to a map $D'\to\sW(\cX)$ over $k$. Thus, by \cite[Theorem 1.5]{OlssonDefRep}, we may canonically identify $\Ext^0(L\psi^*L_{\sW(\cX)},\cO)$ with the set of isomorphism classes of maps $D'\to\sW(\cX)$ over $k$ which extend $\psi$. By the moduli interpretation of $\sW(\cX)$, 
this is identified with isomorphism classes of cartesian diagrams
\[
\xymatrix{
\cD\ar@{-->}[r]\ar[d]_-{\alpha} & \cD'\ar@{-->}[d]^-{\alpha'}\\
\cX_D\ar[r]\ar[d] & \cX_{D'}\ar[d]\\
D\ar[r] & D'
}
\]
where $\cD'$ and the dotted arrows can vary and $\pi'\colon\cD'\to D'$ is flat; \emph{a priori} one also needs to impose the condition that $\pi'$ is a warp, but this is, in fact, not necessary in light of \cite[Proposition 3.1]{SatrianoUsatine4}. Isomorphism classes of such diagrams are classified by the deformation space $\Ext^1(L_{\cD/\cX_D},\cO)$, so we obtain the left isomorphism in the statement of the proposition. The other isomorphism follows by the same argument, and the diagram commutes by the definition of the map $\sW(\cX) \to Y$.
\end{proof}

\begin{proof}[{Proof of \autoref{thm:moduli-interp-hetWY}}]
Throughout this proof, all dimensions are taken over $k'$. The map $p\colon\cX\to Y$ induces a map $\underline{p}\colon\sW(\cX)\to Y$ such that $\overline{\varphi}=\underline{p}\psi$, see \cite[Section 5]{SatrianoUsatine4}. Applying $L\psi^*$ to the transitivity triangle
\[
L\underline{p}^*L_Y\to L_{\sW(\cX)}\to L_{\sW(\cX)/Y}
\]
yields a long exact sequence
\begin{align*}
0\to &\Ext^0(L\psi^*L_{\sW(\cX)/Y},\cO)\to \Ext^0(L\psi^*L_{\sW(\cX)},\cO)\xrightarrow{\Lambda} \Ext^0(L\overline{\varphi}^*L_Y,\cO)\to \\
&\Ext^1(L\psi^*L_{\sW(\cX)/Y},\cO)\to \Ext^1(L\psi^*L_{\sW(\cX)},\cO)\xrightarrow{\Lambda'} \Ext^1(L\overline{\varphi}^*L_Y,\cO).
\end{align*}
Let $\pi^\circ\colon\cD^\circ\to D^\circ$ be the generic fiber of $\pi$. Since $\cD^\circ\to\cU$ is representable and $\cU\to Y$ is an open immersion, we see $\cD^\circ$ is an algebraic space. Hence, $\pi^\circ$ is an isomorphism, i.e., $D^\circ\to D\xrightarrow{\psi}\sW(\cX)$ factors through $\cU\to\cX\xrightarrow{\tau}\sW(\cX)$ where $\tau$ is the locus of trivial warps, see \cite[Theorem 1.9(2)]{SatrianoUsatine4}. Thus, the pullback of $L\psi^*L_{\sW(\cX)/Y}$ to $D^\circ$ agrees with the pullback of $L_{\cU/Y}=0$ to $D^\circ$. 
As a result, $L^i\psi^*L_{\sW(\cX)/Y}$ is torsion for all $i$. So, \autoref{l:ext-vs-height} implies
\[
\het^{(1-n)}_{L_{\sW(\cX)/Y}}(\psi) = \dim \Ext^n(L\psi^*L_{\sW(\cX)/Y},\cO)
\]
for all $n$. Note then that
\[
\het^{(1)}_{L_{\sW(\cX)/Y}}(\psi) = \dim \ker\Lambda
\]
and
\[
\het^{(0)}_{L_{\sW(\cX)/Y}}(\psi) = \dim \coker\Lambda + \dim \ker\Lambda'.
\]
Since $\Ext^1(L\psi^*L_{\sW(\cX)/Y},\cO)$ is torsion, $\ker\Lambda'$ is as well, and so $\ker\Lambda'=\ker(\Lambda'_{\tors})$. We then see
\begin{align*}
\dim\ker(\Lambda'_{\tors}) &=\dim\ker\!\big(\Ext^1(L^0\psi^*L_{\sW(\cX)},\cO)\to \Ext^1(L^0\overline{\varphi}^*L_Y,\cO)\big)\\
&=\dim\coker\!\big((L^0\overline{\varphi}^*L_Y)_{\tors}\to (L^0\psi^*L_{\sW(\cX)})_{\tors}\big)
\end{align*}
where the first equality comes from \autoref{l:ext-vs-height} equation \eqref{eqn:Extn-tors} and the second equality comes from \autoref{l:ext1-ker-vs-coker}. Thus,
\begin{align*}
\het_{\sW(\cX)/Y}(\psi) &=\het^{(0)}_{L_{\sW(\cX)/Y}}(\psi)-\het^{(1)}_{L_{\sW(\cX)/Y}}(\psi)-\dim\ker\Lambda'\\
&=\dim \coker\Lambda-\dim \ker\Lambda.
\end{align*}
The result then follows from \autoref{prop:ExtWY-equals-ExtDX}.
\end{proof}

\section{Warped heights and the weight function}


\autoref{theoremTwistedToWarped} and the untwisted change of variables formula \cite{SatrianoUsatine2} immediately imply a change of variables formula for twisted arcs involving the correction term $\het_{\sW(\cX)/Y}$. The goal of this section is to compute $\het_{\sW(\cX)/Y}$ on warped arcs arising from twisted arcs.

The first main theorem of this section is:

\begin{theorem}\label{thm:LWY-Ext1}
Let $p\colon\cX\to Y$ be a morphism over $k$ with $Y$ an algebraic space and $\cX$ a smooth finite type Artin stack with affine diagonal. Assume $p$ is weakly birational, i.e., there exists a dense open substack $\cU\subset\cX$ such that $\cU\to Y$ is an immersion. 
Let $k'$ be field, $D=\Spec k'[[t]]$, $\ell \in \Z_{>1}$ and $\varphi\colon\cD\to\cX$ be an $\ell$th-twisted arc such that the generic point of $\cD$ maps into $U$. Letting $\psi\colon D\to\sW(\cX)$ be the corresponding arc, we have
\[
\het_{\sW(\cX)/Y}(\psi)=\dim\Ext^1(L\varphi^*L_{\cX/Y},\cO)-\dim\Ext^0(L\varphi^*L_\cX,\cO)_{\tors}-1.
\]
\end{theorem}

As a consequence of \autoref{thm:LWY-Ext1}, we will prove: 


\begin{theorem}\label{theoremHeightWarpingAndHeightWeight}
Let $\cX$ be an equidimensional smooth finite type Artin stack over $k$ with affine diagonal, and assume that $\cX$ has a good moduli space. Let $\ell \in \Z_{>1}$, and let $\eta: |\sJ^\ell(\cX)| \to |\sL(\sW(\cX))|$ denote the canonical map. Let $Y$ be a locally finite type scheme over $k$, let $\cX \to Y$ be a morphism over $k$, and let $\cU$ be an open substack of $\cX$ such that the composition $\cU \hookrightarrow \cX \to Y$ is an open immersion. Then
\[
	\het_{\sW(\cX)/Y} \circ \eta + 1 = \het_{\cX/Y} + \wt_\cX
\]
on $|\sJ^\ell(\cX)| \setminus |\sJ^\ell(\cX \setminus \cU)|$.
\end{theorem}

Throughout this section, we fix the notation as in \autoref{thm:LWY-Ext1} where $\cD$ is an $\ell$-th twisted disc and $\pi\colon\cD\to D$ is the coarse space map. Let $\overline{\varphi}\colon D\to Y$ be the arc induced by $\varphi$. Unless otherwise specified $\dim$ will mean dimension as $k'$-vector space. We use the notation
\[
\cZ_T:=\cZ\times_k T
\]
to denote the base change of a $k$-stack $\cZ$ by a morphism $T\to\Spec k$ and consider the maps
\[
\Phi\colon\mathcal{E}xt^0(L\varphi^*L_\cX,\cO)\to\mathcal{E}xt^1(L_{\cD/\cX_D},\cO)
\]
and
\[
\Lambda\colon\mathcal{E}xt^1(L_{\cD/\cX_D},\cO)\to\mathcal{E}xt^1(L\pi^*L_{D/Y_D},\cO).
\]
Note that we have canonical isomorphisms
\begin{equation}\label{eqn:target-of-Lambda}
\mathcal{E}xt^1(L\pi^*L_{D/Y_D},\cO)\simeq\mathcal{E}xt^0(L(\overline{\varphi}\pi)^*L_Y,\cO)\simeq\mathcal{H}om((\overline{\varphi}\pi)^*\Omega^1_Y,\cO),
\end{equation}
see e.g., \autoref{rmk:Lambda-simplify-target}. Since $\mathcal{H}om((\overline{\varphi}\pi)^*\Omega^1_Y,\cO)$ is torsion free, we have the following commutative diagram
\[
\xymatrix{
\mathcal{E}xt^0(L\varphi^*L_\cX,\cO)\ar[r]^-{\Phi}\ar@{->>}[d] & \mathcal{E}xt^1(L_{\cD/\cX_D},\cO)\ar[dr]^-{\Lambda}\ar@{->>}[d] & \\
\mathcal{E}xt^0(L\varphi^*L_\cX,\cO)/(\tors)\ar[r]^-{\overline{\Phi}} & \mathcal{E}xt^1(L_{\cD/\cX_D},\cO)\ar[r]^-{\overline{\Lambda}}/(\tors) & \mathcal{E}xt^1(L\pi^*L_{D/Y_D},\cO).
}
\]

\autoref{thm:moduli-interp-hetWY} expresses $\het_{\sW(\cX)/Y}(\psi)$ in terms of the dimension of the kernel and cokernel of $\pi_*\Lambda$. Thus, our goal is to compute these quantities in the case of twisted discs. We require several preliminary results.

\begin{lemma}\label{l:ExtLsDoverD}
Letting $\pi\colon\cD\to D=\Spec k'[[t]]$ be our $\ell$-th twisted disc, and $\Spec S=D\xrightarrow{\rho}\cD$ the standard \'etale cover where $S=k'[[t]][u]/(u^\ell-t)$, we have
\[
\rho^*\mathcal{E}xt^1(L_{\cD/D},\cO)\simeq S/u^{\ell-1}\quad\textrm{and\ otherwise}\quad\mathcal{E}xt^i(L_{\cD/D},\cO)=0.
\]
In particular, $\dim\Ext^1(L_{\cD/D},\cO)=1$.
\end{lemma}
\begin{proof}
Note that the result concerning $\dim\Ext^1(L_{\cD/D},\cO)$ follows from the assertion about $\mathcal{E}xt^1(L_{\cD/D},\cO)$ since then $\Ext^1(L_{\cD/D},\cO)$ is the degree $0$ piece of $S/u^{\ell-1}$.

The rest of the proof is nearly identical to that of \autoref{l:twisted-twisted-discs-def-sp}. Let $R=k'[[t]]$. Since $\rho$ is \'etale, $\rho^*\mathcal{E}xt^i(L_{\cD/D},\cO)=\Ext^1(L_{\pi\rho},S)$, where we have identified local and global $\Ext$ since $D$ is affine. Note that $\pi\rho$ factors as a regular closed immersion $i\colon D\to\bA^1_D$ followed by the projection to $D$, where $i$ is cut out by the ideal $I=(u^\ell-t)$. Then we have an exact triangle
\[
I/I^2\to \Omega^1_{R[u]/R}\otimes_{R[u]}S\to L_{S/R}.
\]
This shows vanishing of $\rho^*\mathcal{E}xt^i(L_{\cD/D},\cO)$ for all $i\neq0,1$ and yields an exact sequence
\[
0\to\Ext^0(L_{S/R},S)\to S\xrightarrow{\cdot\ell u^{\ell-1}} S\to \Ext^1(L_{S/R},S)\to 0,
\]
from which the result follows.
\end{proof}

\begin{lemma}\label{l:4termsequence}
We have an exact sequence
\[
0\to \mathcal{E}xt^0(L\varphi^*L_\cX,\cO)_{\tors}\xrightarrow{\Phi_{\tors}} \mathcal{E}xt^1(L_{\cD/\cX_D},\cO)_{\tors}\to \mathcal{E}xt^1(L_{\cD/D},\cO)\to \coker\overline{\Phi}\to 0.
\]
\end{lemma}
\begin{proof}
Let $\varphi'\colon\cD\to\cX_D$ be the map induced by $\varphi$ and $\pi$. Since $\cX$ is flat over $k$, $L_{\cX_D/D}$ is the pullback of $L_\cX$; thus, considering the transitivity triangle of cotangent complexes for the maps $\cD\to\cX_D\to D$, we obtain an exact sequence
\[
0\to\mathcal{E}xt^0(L\varphi^*L_\cX,\cO)\xrightarrow{\Phi} \mathcal{E}xt^1(L_{\cD/\cX_D},\cO)\to \mathcal{E}xt^1(L_{\cD/D},\cO)\to 0;
\]
note we have used vanishing of $\mathcal{E}xt^0(L_{\cD/D},\cO)$ from \autoref{l:ExtLsDoverD} and vanishing of $\mathcal{E}xt^1(L\varphi^*L_\cX,\cO)$ from smoothness of $\cX$. 

Since $\pi\colon\cD\to D$ is generically an isomorphism, 
we see $\Phi$ is generically an isomorphism. Thus, \[
\mathcal{E}xt^0(L\varphi^*L_\cX,\cO)_{\tors}=\Phi^{-1}(\mathcal{E}xt^1(L_{\cD/\cX_D},\cO)_{\tors}).
\]
As a result, we obtain a diagram
\[
\xymatrix{
& 0\ar[d] & 0\ar[d] & 0\ar[d] & \\
0\ar[r]& \mathcal{E}xt^0(L\varphi^*L_\cX,\cO)_{\tors}\ar[r]^-{\Phi_{\tors}}\ar[d] & \mathcal{E}xt^1(L_{\cD/\cX_D},\cO)_{\tors}\ar[r]\ar[d] & \mathcal{Q}\ar[r]\ar[d]& 0\\
0\ar[r]& \mathcal{E}xt^0(L\varphi^*L_\cX,\cO)\ar[r]^-{\Phi}\ar[r]\ar[d] & \mathcal{E}xt^1(L_{\cD/\cX_D},\cO)\ar[r]\ar[d] & \mathcal{E}xt^1(L_{\cD/D},\cO)\ar[r]\ar[d]& 0\\
0\ar[r]& \mathcal{E}xt^0(L\varphi^*L_\cX,\cO)/(\tors)\ar[r]^-{\overline{\Phi}}\ar[d] & \mathcal{E}xt^1(L_{\cD/\cX_D},\cO)/(\tors)\ar[r]\ar[d] & \coker\overline{\Phi}\ar[r]\ar[d] & 0\\
& 0 & 0 & 0 &
}
\]
where all rows and columns are short exact. Splicing together the top row and last column, we obtain the desired result.
\end{proof}

The following two lemmas gives us our desired reformulations of the kernel and cokernel of $\pi_*\Lambda$.

\begin{lemma}\label{l:kerPsi}
We have
\[
\ker\Lambda=\mathcal{E}xt^1(L_{\cD/\cX_D},\cO)_{\tors}
\]
and
\[
\dim\ker\pi_*\Lambda=\dim\Ext^0(L\varphi^*L_\cX,\cO)_{\tors}+1-\coker\pi_*\overline{\Phi}.
\]
\end{lemma}
\begin{proof}
To prove the assertion about $\ker\Lambda$, note first that $\mathcal{H}om((\overline{\varphi}\pi)^*\Omega^1_Y,\cO)$ is torsion-free since $\cO$ is. Then using \eqref{eqn:target-of-Lambda}, we see $\mathcal{E}xt^1(L_{\cD/\cX_D},\cO)_{\tors}\subset\ker\Lambda$. Since $\pi\colon\cD\to D$ is generically an isomorphism, $\Lambda$ is as well. As a result, $\ker\Lambda$ is torsion, hence $\ker\Lambda=\mathcal{E}xt^1(L_{\cD/\cX_D},\cO)_{\tors}$.

For the computation of $\ker\pi_*\Lambda$, we apply the exact functor $\pi_*$ to the exact sequence in \autoref{l:4termsequence}. \autoref{l:ExtLsDoverD} tells us $\Ext^1(L_{\cD/D},\cO)=\pi_*\mathcal{E}xt^1(L_{\cD/D},\cO)$ has dimension $1$. The result then follows from the fact that $\Ext^0(L\varphi^*L_\cX,\cO)_{\tors}=\pi_*(\mathcal{E}xt^0(L\varphi^*L_\cX,\cO)_{\tors})$.
\end{proof}


\begin{lemma}\label{l:cokerseq-and-ordjac}
The maps $\overline{\Phi}$ and $\overline{\Lambda}$ are injective, and we have an exact sequence
\[
0\to\coker\overline{\Phi}\to \coker(\overline{\Lambda}\circ\overline{\Phi})\to \coker\overline{\Lambda}\to 0.
\]
Furthermore,
\[
\coker(\overline{\Lambda}\circ\overline{\Phi})=\coker(\Lambda\circ\Phi)\simeq\mathcal{E}xt^1(L\varphi^*L_{\cX/Y},\cO).
\]
\end{lemma}
\begin{proof}
From the commutative diagram
\[
\xymatrix{
\cD\ar[r]^-{\varphi'}\ar[d]_-{\pi} & \cX_D\ar[r]\ar[d] & D\ar@{=}[d]\\
D\ar[r]^-{\overline{\varphi}'} & Y_D\ar[r] & D
}
\]
and \cite[Lemma 2.2.12(2)]{WebbThesis}, we obtain a morphism
\[
\xymatrix{
L\pi^*L(\overline{\varphi}')^*L_{Y_D/D}\ar[r]\ar[d] & 0\ar[r]\ar[d] & L\pi^*L_{D/Y_D}\ar[d]\\
L(\varphi')^*L_{\cX_D/D}\ar[r] & L_{\cD/D}\ar[r] & L_{\cD/\cX_D}
}
\]
of exact triangles. As discussed earlier, $L_{\cX_D/D}$ (resp.~$L_{Y_D/D}$) is the pullback of $L_\cX$ (resp.~$L_Y$). We therefore have a commutative diagram
\[
\xymatrix{
\mathcal{E}xt^0(L\varphi^*L_\cX,\cO)\ar[r]^-{\Phi}\ar[d] & \mathcal{E}xt^1(L_{\cD/\cX},\cO)\ar[d]^-{\Lambda}\\
\mathcal{E}xt^0(L\varphi^*Lp^*L_Y,\cO)\ar[r]^-{\simeq} & \mathcal{E}xt^1(L\pi^*L_{D/Y_D},\cO)
}
\]
Thus, we obtain an exact sequence 
\[
\mathcal{E}xt^0(L\varphi^*L_{\cX},\cO)\xrightarrow{\Lambda\Phi} \mathcal{E}xt^0(L(\overline{\varphi}\pi)^*L_Y,\cO)\to \mathcal{E}xt^1(L\varphi^*L_{\cX/Y},\cO)\to \mathcal{E}xt^1(L\varphi^*L_{\cX},\cO).
\]
Since $\cX$ is smooth, $\mathcal{E}xt^1(L\varphi^*L_{\cX},\cO)=0$ and so
\[
\coker(\Lambda\Phi)\simeq\mathcal{E}xt^1(L\varphi^*L_{\cX/Y},\cO).
\]
We see $\coker(\overline{\Lambda}\circ\overline{\Phi})=\coker(\Lambda\Phi)$ directly from the definitions. 
Lastly, since $\Phi$ and $\Lambda$ are generically isomorphisms, $\overline{\Phi}$ and $\overline{\Lambda}$ are injective, from which the short exact sequence in the statement of the lemma follows.
\end{proof}

We may now prove the main theorems of this section.

\begin{proof}[{Proof of \autoref{thm:LWY-Ext1}}]
By \autoref{thm:moduli-interp-hetWY} and \autoref{l:kerPsi}, we have
\[
\het_{\sW(\cX)/Y}(\psi)=\dim\coker\pi_*\Lambda+\dim\coker\pi_*\overline{\Phi}-\dim\Ext^0(L\varphi^*L_\cX,\cO)_{\tors}-1.
\]
Since $\coker\pi_*\Lambda=\coker\pi_*\overline{\Lambda}$, \autoref{l:cokerseq-and-ordjac} tells us 
\[
\dim\coker\pi_*\Lambda+\dim\coker\pi_*\overline{\Phi}=\dim\coker\pi_*(\overline{\Lambda}\circ\overline{\Phi})=\dim\Ext^1(L_{\cX/Y},\cO)
\]
and the result follows.
\end{proof}


\begin{proof}[{Proof of \autoref{theoremHeightWarpingAndHeightWeight}}]
Let $\varphi\colon\cD\to\cX$ be an $\ell$-th twisted disc and let $\rho\colon \Spec S=D\to\cD$ be the standard cover. For ease of notation, given $E\in D^-_{\coh}(\cX)$, we write $E|_D$ to mean $L(\varphi\rho)^*E$ and $E|_{D_0}$ to mean $L(\varphi\rho\iota)^*E$ where $\iota\colon D_0\to D$ is the natural closed immersion. By \autoref{thm:LWY-Ext1}, we must prove
\[
\dim\Ext^1(L_{\cX/Y}|_D,\cO)_0-\dim(\Ext^0(L_\cX|_D,\cO)_{\tors})_0=\het_{\cX/Y}(\varphi)+\wt_\cX(\varphi)
\]
where the zero subscript denotes the degree $0$ part.

First note that from the exact triangle
\begin{equation}\label{eqn:XYktriangle}
L_Y|_D \to L_{\cX}|_D \to L_{\cX/Y}|_D,
\end{equation}
we see $\cH^1(L_{\cX}|_D)=\cH^1(L_{\cX/Y}|_D)$ which is torsion. Thus, by \autoref{lemmaPullBackCotangentComplexBecomesTwoVectorBundles} combined with \autoref{prop:graded-Smith}, we have an exact triangle
\[
L_\cX|_D\to \bigoplus_{i=1}^I S(a_i)\oplus\bigoplus_{j=1}^J S(b_j-n_j)\xrightarrow{f} \bigoplus_{j=1}^J S(b_j)
\]
where $f$ maps the generator of $S(a_i)$ to $0$ and maps the generator of $S(b_j-n_j)$ to $u^{n_j}$ times the generator for $S(b_j)$; here $u$ a uniformizer, $n_j>0$, and $a_i,b_j\in(0,\ell]$. We see then that
\[
\cH^0(L_\cX|_D)=\bigoplus_{i=1}^I S(a_i)\quad\textrm{and}\quad \cH^1(L_\cX|_D)=\cH^1(L_{\cX/Y}|_D)=\bigoplus_{j=1}^J S/u^{n_j}(b_j).
\]
Then we have an exact triangle
\[
L_\cX|_{D_0}\to \bigoplus_{i=1}^I S/u(a_i)\oplus\bigoplus_{j=1}^J S/u(b_j-n_j)\xrightarrow{0} \bigoplus_{j=1}^J S/u(b_j).
\]
So, letting $q_j$ and $r_j\in(0,\ell]$ be integers with
\[
b_j-n_j=q_j\ell+r_j,
\]
we see
\[
\wt_\cX=\dim\cX+\frac{1}{\ell}\sum b_j-\frac{1}{\ell}\sum a_i-\frac{1}{\ell}\sum r_j.
\]

Our next goal is to compute $(\Ext^0(L_\cX|_D, \cO)_{\tors})_0$. We have
\begin{equation}\label{eqn:4-equal-vs}
\Ext^0(L_\cX|_D, \cO)_{\tors} = \Ext^1(\cH^1(L_\cX|_D),\cO)_{\tors}= \Ext^1(\cH^1(L_{\cX/Y}|_D), \cO)=    \Ext^0(L_{\cX/Y}|_D, \cO),
\end{equation}
where the first and last equalities are by \autoref{l:ext-vs-height} and the middle equality uses that $\cH^1(L_\cX|_D) = \cH^1(L_{\cX/Y}|_D)$ is torsion. From the resolution 
\[
0\to \bigoplus_{j=1}^J S(b_j-n_j)\xrightarrow{u^{n_j}} \bigoplus_{j=1}^J S(b_j)\to \cH^1(L_{\cX/Y}|_D)\to 0
\] 
we see
\[
\Ext^0(L_\cX|_D, \cO)_{\tors} =\Ext^1(\cH^1(L_{\cX/Y}|_D), \cO) = \bigoplus_{j=1}^J S/u^{n_j} (n_j- b_j);
\]
so by \eqref{eqn:4-equal-vs} and \autoref{l:ext-vs-height}, we see
\[
\het^{(1)}_{L_{\cX/Y}}(\varphi)=\frac{1}{\ell}\sum_j n_j.
\]
Note $n_j - b_k = -q_j\ell-r_j$ and that $\ell- r_j \in [0,\ell)$. A straightforward computation shows 
\begin{align*}
    \dim(S/u^n(c))_0 = 1 + \lfloor\frac{n-c-1}{\ell} \rfloor \quad \text{if }c \in [0,\ell)
\end{align*}
and so
\begin{align*}
    \dim (\Ext^0(L_\cX|_D, \cO)_{\tors})_0 &= \sum_{j}\left( 1+ \lfloor \frac{n_j-\ell+r_j-1}{\ell}\rfloor\right)\\
&= \sum_{j}\left( \lfloor \frac{(n_j+r_j-b_j)+(b_j-1)}{\ell}\rfloor\right)\\
    &= \frac{1}{\ell}\sum_j n_j - \frac{1}{\ell}\sum_j b_j + \frac{1}{\ell} \sum_j r_j\\
    &= \het^{(1)}_{L_{\cX/Y}}(\varphi) - \frac{1}{\ell}\sum_j b_j + \frac{1}{\ell} \sum_j r_j.
\end{align*}

We next compute $\dim\Ext^1(L_{\cX/Y}|_D,\cO)_0$. Again using the exact triangle \eqref{eqn:XYktriangle}, we have an exact sequence 
\begin{align*}
    0 \to \Ext^0(L_{\cX/Y}|_D,\cO) \to \Ext^0(L_{\cX}|_D,\cO) \to \Hom(\Omega^1_Y,\cO) \to \Ext^1(L_{\cX/Y}|_D,\cO) \to 0.
\end{align*}
Since $\Hom(\Omega^1_Y,\cO)$ is torsion free, and $\Ext^0(L_{\cX/Y}|_D,\cO) = \Ext^0(L_{\cX}|_D,\cO)_{\tors}$ by \eqref{eqn:4-equal-vs}, we see 
\begin{align*}
    0 \to \Ext^0(L_{\cX}|_D,\cO)/(\tors) \to \Hom(\Omega^1_Y,\cO) \to \Ext^1(L_{\cX/Y}|_D,\cO) \to 0
\end{align*}
is exact. We may rewrite this sequence as 
\begin{align*}
    0 \to \bigoplus_k S(c_k) \xrightarrow{u^{m_k}} S^{\dim(\cX)} \to \bigoplus_k S/u^{m_k} \to 0
\end{align*}
with $c_k\in[0,\ell)$. Note $c_k\equiv m_k\pmod{\ell}$. 
We see 
\begin{align*}
    \dim \Ext^1(L_{\cX/Y}|_D,\cO)_0 &= \sum_k \left(1+\lfloor\frac{m_k-1}{\ell}\rfloor\right)
    = \sum_k \left(1+\lfloor\frac{m_k-c_k + (c_k+1)}{\ell}\rfloor\right)\\
    &=\sum_k \left(\frac{m_k + c_k}{\ell} \right)
   = \het^{(0)}_{L_{\cX/Y}}(\varphi) + \frac{1}{\ell}\sum_k c_k
\end{align*}
where the last equality comes from \autoref{l:ext-vs-height} which shows
\[
\frac{1}{\ell}\sum_k m_k=\dim\Ext^1(L_{\cX/Y}|_D,\cO) = \het^{(0)}_{L_{\cX/Y}}(\varphi).
\]
It remains to relate the $c_k$ to the $a_i$. From Equation \eqref{eqn:Extn-tors-ss-degen} we see
\begin{align*}
     0 \to \Ext^1(\cH^1(L_{\cX}|_D),\cO) \to  \Ext^0(L_{\cX}|_D,\cO) \to \Hom(\cH^0(L_{\cX}|_D),\cO) \to 0 
\end{align*}
is exact. So using that $\Ext^1(\cH^1(L_{\cX}|_D),\cO) = \Ext^0(L_{\cX}|_D,\cO)_{\tors}$ from \eqref{eqn:4-equal-vs}, we see
\[
\Ext^0(L_{\cX}|_D,\cO)/(\tors) = \Hom(\cH^0(L_{\cX}|_D),\cO) = \bigoplus_i S(-a_i).
\]
Thus the multiset $\{c_k\}$ agrees with the multiset $\{ \ell-a_i\}$. 
It follows that
\begin{align*}
    \dim \Ext^1(L_{\cX/Y}|_D,\cO)_0 = \het^{(0)}_{L_{\cX/Y}}(\varphi) - \frac{1}{\ell}\sum_{i=1}^I a_i + I.
\end{align*}
Taking the difference of $\dim (\Ext^0(L_\cX|_D, \cO)_{\tors})_0$ with $\dim \Ext^1(L_{\cX/Y}|_D,\cO)_0$ and using that $I = \rank \cH^0 (L_{\cX}|_{D_0}) - \rank\cH^1(L_{\cX}|_{D_0})=\dim\cX$ by \autoref{lemmaPullBackCotangentComplexBecomesTwoVectorBundles}, we conclude
\begin{align*}
    (\het_{\sW(\cX)/Y}\circ\eta)(\varphi) + 1 &= \het_{L_{\cX/Y}}(\varphi) + \wt_{\cX}(\varphi).\qedhere
\end{align*}
\end{proof}

\section{Stringy Hodge numbers via crepant resolutions by Artin stacks}

This section is dedicated to proving \autoref{mainTheoremCrepantResolutionImpliesLogTerminal} and \autoref{maintheorem}. We begin with the following general change of variables formula for motivic integration over twisted arcs.

\begin{theorem}\label{corollaryTwistedChangeOfVariables}
Let $\cX$ be an equidimensional smooth finite type Artin stack over $k$ with affine diagonal, and assume that $\cX$ has a good moduli space. Let $Y$ be an irreducible finite type scheme over $k$, let $\cX \to Y$ be a morphism over $k$, let $\cU$ be an open substack of $\cX$ such that $\cU \hookrightarrow \cX \to Y$ is an open immersion, let $\cC \subset |\sJ(\cX)|$ be a bounded cylinder that is disjoint from $|\sJ(\cX \setminus \cU)|$, and let $E \subset \sL(Y)$ be a cylinder such that $\cX \to Y$ induces a bijection $\overline{\cC}(k') \to E(k')$ for every field extension $k'$ of $k$.
\begin{enumerate}[label=(\alph*)]

\item The restriction of $\het_{\cX/Y} + \wt_\cX$ to $\cC$ is integer valued and takes only finitely many values.

\item For all $n \in \Z$, the set $(\het_{\cX/Y} + \wt_\cX)^{-1}(n) \cap \cC$ is a bounded cylinder in $|\sJ(\cX)|$.

\item We have
\[
	\mu_Y(E) = \int_{\cC} \bL^{-\het_{\cX/Y} - \wt_\cX} \diff\nu_\cX.
\]

\end{enumerate}
\end{theorem}

\begin{proof}
We have that $\sW(\cX)$ is a locally finite type Artin stack over $k$ \cite[Theorem 1.9(1)]{SatrianoUsatine4} with affine geometric stabilizers and separated diagonal by \autoref{propositionQuasiAffineInertia}. For each $\ell \in \Z_{>0}$, set $\cC^\ell = \cC \cap |\sJ^
\ell(\cX)|$, and set $\cE^\ell \subset |\sL(\widetilde{\sW}(\cX))|$ to be the image of $\cC^\ell$. Since $\cD^{\ell}_{k'} \cong \cD^{\ell'}_{k'}$ for some $k'$ implies $\ell = \ell'$, we have that the $\cE^\ell$ are pairwise disjoint. Since $\cC$ is a bounded cylinder, each $\cC^\ell$ is a bounded cylinder and there are only finitely many nonempty $\cC^\ell$. By \autoref{theoremTwistedToWarped}, each $\cE^\ell$ is a bounded cylinder and $\overline{\cC^\ell}(k') \to \overline{\cE^\ell}(k')$ is a bijection for all field extensions $k'$ of $k$. Set $\cE \subset |\sL(\widetilde{\sW}(\cX)|$ to be the image of $\cC$. Then $\overline{\cC}(k') \to \overline{\cE}(k')$ is a bijection for all field extensions $k'$ of $k$. Because $\overline{\cC}(k') \to E(k')$ is a bijection for all $k'$, this implies that $\overline{\cE}(k') \to E(k')$ is a bijection for all $k'$. Since there are only finitely many nonempty $\cE^\ell$ and their union is $\cE$, we have that $\cE$ is a bounded cylinder. Because $\cC$ is disjoint from $|\sJ(\cX \setminus \cU)|$, we have that $\cE$ is disjoint from $|\sL(\widetilde{\sW}(\cX) \setminus \cU)|$. Therefore \cite[Theorem 7.1]{SatrianoUsatine5} implies that the restriction of $\het_{\sW(\cX)/Y}$ to $\cE$ is integer valued and takes only finitely many values, each $\het_{\sW(\cX)/Y}^{-1}(n) \cap \cE$ is a bounded cylinder in $|\sL(\widetilde{\sW}(\cX)|$, and
\[
	\mu_Y(E) = \int_{\cE} \bL^{-\het_{\sW(\cX)/Y}} \diff\mu_{\widetilde{\sW}(\cX)}.
\]
For each $\ell, n$, set $\cE^{\ell, n} = \het_{\sW(\cX)/Y}^{-1}(n) \cap \cE^{\ell}$ and 
\[
	\cC^{\ell, n} = \eta^{-1}(\cE^{\ell, n}) \cap \cC = (\het_{\sW(\cX)/Y}  \circ \eta)^{-1}(n) \cap \cC^{\ell},
\] 
where $\eta: |\sJ(\cX)| \to |\sL(\sW(\cX))|$ is the canonical map. It is straightforward to verify that each $\cE^{\ell, n}$ and $\cC^{\ell, n}$ is a bounded cylinder and that $\cE^{\ell,n}$ is the image of $\cC^{\ell,n}$. Then for all $n$,
\[
	\nu_\cX(\cC^{1, n}) = \mu_\cX(\cC^{1,n}) = \mu_{\widetilde{\sW}(\cX)}(\cE^{1,n}),
\]
and for all $\ell >1$ and all $n$, \autoref{theoremTwistedToWarped} implies
\[
	\nu_\cX(\cC^{\ell, n}) = \bL \mu_{\widetilde{\sW}(\cX)}(\cE^{\ell,n}).
\]
Set $f: |\sJ(\cX)| \to \Z \cup \{\infty\}$ be such that $f|_{|\sJ^1(\cX)|} = (\het_{\sW(\cX)/Y}  \circ \eta) = \het_{\cX/Y}$ and $f|_{|\sJ^\ell(\cX)|} = (\het_{\sW(\cX)/Y}  \circ \eta) + 1$ for $\ell > 1$. Because $\cC^{\ell, n}$ are bounded cylinders and there are only finitely many nonzero ones, we have that $f|_\cC$ is integer valued and takes only finitely many values and each $f^{-1}(n) \cap \cC$ is a bounded cylinder. Then
\begin{align*}
	\mu_Y(E) &= \sum_{\ell, n} \bL^{-n} \mu_{\widetilde{\sW}(\cX)}(\cE^{\ell, n}) \\
	&= \sum_{n} \bL^{-n} \nu_\cX(\cC^{1,n}) + \sum_{\ell > 1} \sum_n \bL^{-n-1} \nu_\cX(\cC^{\ell, n})\\
	&= \int_\cC \bL^{-f} \diff\nu_\cX.
\end{align*}
Finally $f|_\cC = (\het_{\cX/Y} + \wt_\cX)|_\cC$ by \autoref{theoremHeightWarpingAndHeightWeight}, and we are done.
\end{proof}

We will eventually use \autoref{corollaryTwistedChangeOfVariables} in the special case where $Y$ has $\Q$-Gorenstein singularities to obtain a formula for $\mu^\Gor_Y$ in terms of $\nu_\cX$ and the set $\cA_\cX \subset |\sJ(\cX)|$ defined in the introduction. The next sequence of results culminates in \autoref{corollaryLogTerminalConvergenceExpressionForGorensteinWithCanonical}, which gives that desired formula.

\begin{proposition}\label{propComparingRelativeCanonicalAndHeight}
Let $\cX$ be a smooth finite type Artin stack over $k$ with affine diagonal, let $Y$ be a finite type scheme over $k$, let $m \in \Z_{>0}$ be such that $Y$ is $m$-Gorenstein, let $\cX \to Y$ be a morphism over $k$, and let $\cU$ be an open substack of $\cX$ such that the composition $\cU \hookrightarrow \cX \to Y$ is an open immersion. Then
\[
	\ord_{mK_{\cX/Y}} = m\het_{\cX/Y} - \ord_{\sI_{Y,m}} \circ \zeta
\]
on $|\sJ(\cX)| \setminus |\sJ(\cX \setminus \cU)|$, where $\zeta: |\sJ(\cX)| \to \sL(Y)$ is the canonical map induced by $\cX \to Y$, and $\sI_{Y,m}$ is the unique ideal sheaf such that the image of $(\Omega^{\dim Y}_Y)^{\otimes m} \to \omega_{Y,m}$ is $\sI_{Y,m}\omega_{Y,m}$.
\end{proposition}

\begin{proof}
Let $\ell \in \Z_{> 0}$, let $k'$ be a field extension of $k$, let $\varphi: \cD^\ell_{k'} \to \cX$ be a $k'$-valued point of $|\sJ(\cX)| \setminus |\sJ(\cX \setminus \cU)|$, and let $\psi: \Spec(k'\llbracket t \rrbracket) \to \cX$ be the composition of the usual covering map $\Spec(k'\llbracket t \rrbracket) \to \cD^\ell_{k'}$ with $\varphi$. Then $\psi$ is a $k'$-valued point of $|\sL(\cX)| \setminus |\sL(\cX \setminus \cU)|$, and by definition
\[
	\ord_{mK_{\cX/ Y}}(\varphi) = \frac{\ord_{mK_{\cX/Y}}(\psi)}{\ell},
\]
and
\[
	\het_{\cX/Y}(\varphi) = \frac{\het_{\cX/Y}(\psi)}{\ell}.
\]
Let $\Spec(k'\llbracket t \rrbracket) \to Y$ be the image of $\varphi$ under the canonical map $\sJ(\cX) \to \sL(Y)$. Then we have a commuting diagram
\[
\xymatrix{
\Spec(k'\llbracket t \rrbracket) \ar[r] \ar[dr] & \cD^\ell_{k'} \ar[r]^{}\ar[d]_-{} & \cX\ar[d]^-{}\\
&\Spec(k'\llbracket t \rrbracket) \ar[r]^-{} & Y
}
\]
where the diagonal arrow $\Spec(k'\llbracket t \rrbracket) \to \Spec(k'\llbracket t \rrbracket)$ is given by $t \mapsto t^\ell$. Thus
\[
	(\ord_{\sI_{Y,m}} \circ \zeta)(\varphi) = \frac{(\ord_{\sI_{Y,m}} \circ \zeta)(\psi)}{\ell},
\]
where we note that $\zeta$ restricted to $|\sJ^1(\cX)| = |\sL(\cX)|$ is the canonical map $|\sL(\cX)| \to \sL(Y)$. Now the proposition follows immediately from the untwisted case, which is \cite[Theorem 1.13]{SatrianoUsatine3}.
\end{proof}

\begin{proposition}\label{propositionCrepantResolutionImpliesStableGMS}
Let $\cX$ be a finite type Artin stack over $k$, let $Y$ be a finite type algebraic space over $k$, and let $\cX \to Y$ be a strongly birational map that factors as a good moduli space morphism followed by a tame proper morphism. Then $\cX$ admits a properly stable good moduli space $X$ and the induced map $X \to Y$ is proper.

Furthermore if $V \hookrightarrow Y$ is an open sub-algebraic-space over which $\cX \to Y$ is an isomorphism, then the good moduli space map $\cX \to X$ is an isomorphism over the preimage of $V$ in $X$, and in particular, the stable locus of $\cX$ contains the preimage of $V$ in $\cX$.
\end{proposition}

\begin{proof}
We assumed $\cX \to Y$ factors as a good moduli space morphism $\cX \to \cY$ followed by a tame proper map $\cY \to Y$. Because $\cY \to Y$ is tame, $\cY$ admits a good moduli space $X$ with the good moduli space map $\cY \to X$ proper and also a coarse space map. In particular, $\cX \to X$ is a good moduli space morphism. Since $\cY \to Y$ is proper, by Keel--Mori the induced map $X \to Y$ is proper. Note that because $\cX \to Y$ is strongly birational, the last sentence of the proposition would imply that $\cX \to X$ is properly stable. It thus remains to verify that last sentence.

Suppose $V \hookrightarrow Y$ is an open sub-algebraic-space over which $\cX \to Y$ is an isomorphism. By base changing $\cX \to X \to Y$ along $V \hookrightarrow Y$, noting that an algebraic space is its own good moduli space, we see that $\cX \to X$ is an isomorphism over the preimage of $V$ along $X \to Y$.
\end{proof}

\begin{lemma}\label{lemmaLiftTwistedArcsProperRepresentable}
Let $\cX$ and $\cY$ be Artin stacks, let $\cX \to \cY$ be a proper representable morphism, and let $\cV$ be an open substack of $\cY$ over which $\cX \to \cY$ is an isomorphism. Let $k'$ be a field, let $\ell \in \Z_{>0}$, and let $\cD^\ell_{k'} \to \cY$ be a morphism such that the composition $\Spec(k'\llparenthesis t \rrparenthesis) \to \cD^\ell_{k'} \to \cY$ factors through $\cV$. Then there exists a morphism $\cD^\ell_{k'} \to \cX$, unique up to 2-isomorphism, such that the diagram
\[
\xymatrix{
& \cX \ar[d] \\
\cD^\ell_{k'} \ar[r]\ar[ur] & \cY
}	
\]
is 2-commutative. Furthermore, if $\cD^\ell_{k'} \to \cY$ is representable, then so is $\cD^\ell_{k'} \to \cX$.
\end{lemma}

\begin{proof}
Let $\cD$ be the normalization of $\cX \times_\cY \cD^\ell_{k'}$ in $\cX \times_\cY \Spec(k'\llparenthesis t \rrparenthesis) \cong \Spec(k'\llparenthesis t \rrparenthesis)$. We then have the following 2-commutative diagram.
\[
\xymatrix{
\cD \ar[r]\ar[d]& \cX \ar[d] \\
\cD^\ell_{k'} \ar[r] & \cY
}	
\]
We will prove that $\cD \to \cD^\ell_{k'}$ is an isomorphism. Let $\cD^\circ = \cD \times_{\cD^\ell_{k'}} \Spec(k'\llparenthesis t \rrparenthesis)$. Note that $\cD^\circ \to \cD$ and $\Spec(k'\llparenthesis t \rrparenthesis) \to \cD^\ell_{k'}$ are open immersions and $\cD^\circ \to \Spec(k'\llparenthesis t \rrparenthesis)$ is an isomorphism. By \cite[035L]{stacks-project}, $\cD$ is normal and $\cD^\circ$ is dense in $\cD$. By \cite[Lemma 032W and Proposition 0334]{stacks-project}, $\cX \times_\cY \cD^\ell_{k'}$ is Nagata, so \cite[Lemma 03GR]{stacks-project} implies that $\cD \to \cX \times_\cY \cD^\ell_{k'}$ is finite. Therefore $\cD \to \cD^\ell_{k'}$ is proper and representable. Now let $T \to \Spec(k'\llbracket t^{1/\ell} \rrbracket)$ be the base change of $\cD \to \cD^\ell_{k'}$ along the \'{e}tale cover $\Spec(k'\llbracket t^{1/\ell} \rrbracket) \to \cD^\ell_{k'}$. By the valuative criterion for algebraic spaces, $T \to \Spec(k'\llbracket t^{1/\ell} \rrbracket)$ has a section $\Spec(k'\llbracket t^{1/\ell} \rrbracket) \to T$. Since $T \to \Spec(k'\llbracket t^{1/\ell} \rrbracket)$ is separated, the section $\Spec(k'\llbracket t^{1/\ell} \rrbracket) \to T$ is a closed immersion. Thus $\Spec(k'\llbracket t^{1/\ell} \rrbracket) \to T$ is a finite birational morphism of integral schemes with normal target, so is an isomorphism by \cite[Lemma 0AB1]{stacks-project}. Therefore $T \to \Spec(k'\llbracket t^{1/\ell} \rrbracket)$ is an isomorphism, so $\cD \to \cD^\ell_{k'}$ is an isomorphism.

The composition $\cD^\ell_{k'} \xrightarrow{\sim} \cD \to \cX$, where $\cD^\ell_{k'} \xrightarrow{\sim} \cD$ is the inverse of $\cD \xrightarrow{\sim} \cD^\ell_{k'}$, is the map as desired by the proposition. Furthermore since $\cD \to \cX \times_\cY \cD^\ell_{k'}$ is representable, the constructed map $\cD^\ell_{k'} \to \cX$ is representable if $\cD^\ell_{k'} \to \cY$ is.

It remains to show uniqueness of $\cD^\ell_{k'} \to \cX$. Suppose $\varphi: \cD^\ell_{k'} \to \cX$ is another such map, and let $\varphi': \cD^\ell_{k'} \to \cX \times_\cY \cD^\ell_{k'}$ be the map induced by $\varphi$ and the identity map on $\cD^\ell_{k'}$. Since $\cX \to \cY$ is separated and representable, $\varphi'$ is a section of a separated and representable map and is therefore a closed immersion. Then the universal property of relative normalization, see \cite[Lemma 035I]{stacks-project}, gives a map $\psi: \cD \to \cD^\ell_{k'}$ fitting into the following 2-commutative diagram.
\[
\xymatrix{
\Spec(k'\llparenthesis t \rrparenthesis) \ar[r]\ar[d]& \cD^{\ell}_{k'} \ar[d]^-{\varphi'} \\
\cD \ar[r] \ar[ur]^-{\psi} & \cX \times_\cY \cD^\ell_{k'}
}
\]
This diagram shows that $\psi$ is 2-isomorphic to the isomorphism $\cD \xrightarrow{\sim} \cD^\ell_{k'}$ from earlier. In particular $\psi$ is an isomorphism. Thus $\varphi'$ is 2-isomorphic to the composition $\cD^{\ell}_{k'} \xrightarrow{\psi^{-1}} \cD \to \cX \times_{\cY} \cD^{\ell}_{k'}$, so $\varphi$ is 2-isomorphic to the composition $\cD^{\ell}_{k'} \xrightarrow{\psi^{-1}} \cD \to \cX$, which is 2-isomorphic to the first map $\cD^\ell_{k'} \to \cX$ that we constructed. We are thus done.
\end{proof}

\begin{lemma}\label{lemmaLiftingToEdidinRydh}
Let $\cX$ be an equidimensional smooth finite type Artin stack over $k$ with affine diagonal, and assume that $\cX$ has a properly stable good moduli space. Let $\cX^s$ be the stable locus of $\cX$, let $\cX' \to \cX$ be the canonical reduction of stabilizers of Edidin--Rydh, let $\cU'$ be the preimage of $\cX^s$ in $\cX'$, let $\cC' \subset |\sJ(\cX')| \setminus |\sJ(\cX' \setminus \cU')|$, and let $\cC$ be the image of $\cC'$ in $|\sJ(\cX)|$.
\begin{enumerate}[label=(\alph*)]

\item For any field extension $k'$ of $k$, the map
\[
	\overline{\cC'}(k') \to \overline{\cC}(k')
\]
is surjective.

\item If $\cC'$ is a cylinder, then $\cC$ is a bounded cylinder.

\end{enumerate}
\end{lemma}

\begin{proof}
\begin{enumerate}[label=(\alph*)]

\item Let $\cX' = \cX_n \to \cX_{n-1} \to \dots \to \cX_0 = \cX$ be the sequence of maps in the Edidin-Rydh canonical reduction of stabilizers algorithm, and let $\cU_i$ be the preimage of $\cX^s$ in $\cX_i$. \cite[Remark 2.12]{EdidinRydh} implies that over $\cU_i$, the map $\cX_{i+1} \to \cX_i$ is an isomorphism $\cU_{i+1} \to \cU_i$. Also by \cite[Remark 2.12]{EdidinRydh}, each $\cX_{i+1} \to \cX_{i}$ factors as an open immersion $\cX_{i+1} \to \cZ_{i+1}$ followed by a proper representable morphism $\cZ_{i+1} \to \cX_i$ that is an isomorphism over $\cU_i$.

Let $\varphi: \cD^\ell_{k'} \to \cX$ be a twisted arc with equivalence class in $\cC$. Then there exists a field extension $k''$ of $k'$ such that $\varphi': \cD^\ell_{k''} \to \cD^\ell_{k'} \xrightarrow{\varphi} \cX$ lifts to a twisted arc $\varphi'_n: \cD^\ell_{k''} \to \cX_n$ with equivalence class in $\cC'$. Let each $\varphi'_i$ be the composition $\cD^\ell_{k''} \xrightarrow{\varphi'_n} \cX_n \to \cX_i$, so in particular, $\varphi'_0$ is 2-isomorphic to $\varphi'$. We will prove that we have a sequence of twisted arcs $\{\varphi_i: \cD^\ell_{k'} \to \cX_i\}_i$ such that $\cD^\ell_{k'} \xrightarrow{\varphi_{i+1}} \cX_{i+1} \to \cX_i$ is 2-isomorphic to $\varphi_i$ and $\cD^\ell_{k''} \to \cD^\ell_{k'} \xrightarrow{\varphi_i} \cX_n$ is 2-isomorphic to $\varphi'_i$. We will induct on $i$ and thus assume we already have $\varphi = \varphi_0, \dots, \varphi_i$. Note that by setup, the generic point of $\varphi_i$ lands in $\cU_i$. Thus by \autoref{lemmaLiftTwistedArcsProperRepresentable}, there exists a twisted arc $\psi_{i+1}: \cD^\ell_{k'} \to \cZ_{i+1}$ such that $\cD^\ell_{k'} \xrightarrow{\psi_{i+1}} \cZ_{i+1} \to \cX_i$ is 2-isomorphic to $\varphi_i$. By the uniqueness part of \autoref{lemmaLiftTwistedArcsProperRepresentable}, $\cD^\ell_{k''} \to \cD^\ell_{k'} \xrightarrow{\psi_{i+1}} \cZ_{i+1}$ is 2-isomorphic to $\cD^\ell_{k''} \xrightarrow{\varphi'_{i+1}} \cX_{i+1} \to \cZ_{i+1}$. Noting that $\cX_{i+1} \to \cZ_{i+1}$ is an open immersion, this implies that $\psi_{i+1}$ factors as $\cD^{\ell}_{k'} \xrightarrow{\varphi_{i+1}} \cX_{i+1} \to \cZ_{i+1}$ for some twisted arc $\varphi_{i+1}$. Furthermore the composition $\cD^\ell_{k''} \to \cD^\ell_{k'} \xrightarrow{\varphi_{i+1}} \cX_{i+1}$ is 2-isomorphic to $\varphi'_{i+1}$. Therefore we have our desired sequence $\varphi_0, \dots, \varphi_n$. Since $\cD^\ell_{k''} \to \cD^\ell_{k'} \xrightarrow{\varphi_n} \cX_n$ is 2-isomorphic to $\varphi'_n$, we have that $\varphi_n$ has equivalence class in $\cC'$ and we are done.

\item Suppose $\cC'$ is a cylinder. Then $\cC'$ is automatically a bounded cylinder since $\cX'$ is finite type and tame. In particular the bounded cylinder $\cC' \cap |\sJ^\ell(\cX')| = \emptyset$ for all but finitely many $\ell$. Thus we may assume that $\cC' \subset |\sJ^\ell(\cX')|$ for some $\ell \in \Z_{>0}$. Consider the following 2-commutative diagram
\[
\xymatrix{
\sJ^\ell(\cX') \ar[r]\ar[d] & \sL(\widetilde{\sW}(\cX')) \ar[d]\\
\sJ^\ell(\cX) \ar[r] & \sL(\widetilde{\sW}(\cX))
}
\]
and let $\cE', \cE$ be the images of $\cC'$ in $|\sL(\widetilde{\sW}(\cX'))|, |\sL(\widetilde{\sW}(\cX))|$, respectively. Then $\cE'$ is a bounded cylinder by \autoref{theoremTwistedToWarped}. Note that over $\cX^s$, the map $\widetilde{\sW}(\cX') \to \widetilde{\sW}(\cX)$ is the isomorphism $\cU' \to \cX^s$. Then since $\cE' \subset |\sL(\widetilde{\sW}(\cX'))| \setminus |\sL(\widetilde{\sW}(\cX') \setminus \cU')|$, \cite[Proposition 3.5]{SatrianoUsatine5} implies that $\cE$ is a bounded cylinder. Then $\cC$ is a (bounded) cylinder by \autoref{theoremTwistedToWarped}.
\end{enumerate}
\end{proof}

\begin{lemma}\label{lemmaProperMapToSpaceLiftingArcsToTwistedArcs}
Let $\cX$ be a finite type Artin stack over $k$ with affine diagonal, let $Y$ be a finite type algebraic space over $k$, let $\cX \to Y$ be a tame proper morphism, let $V \hookrightarrow Y$ be an open-sub-algebraic space, and let $\cU$ be the preimage of $V$ in $\cX$. Assume that $|\cU|$ is dense in $|\cX|$ and that the map $\cU \to V$ induced by $\cX \to Y$ is an isomorphism. Let $E \subset |\sL(Y)| \setminus |\sL(Y \setminus V)|$, and let $\cC$ be the preimage of $E$ in $|\sJ(\cX)|$.
\begin{enumerate}[label=(\alph*)]

\item For any field extension $k'$ of $k$, the map
\[
	\overline{\cC}(k') \to E(k')
\]
is bijective.

\item If $E$ is a cylinder, then $\cC$ is a cylinder.

\end{enumerate}
\end{lemma}

\begin{proof}
\begin{enumerate}[label=(\alph*)]

\item Noting that $\cD^\ell_{k'} \to \Spec(k'\llbracket t \rrbracket)=D_{k'}$ is the $\ell$th root of the special point of $D_{k'}$, the desired result follows immediately from \cite[Theorem 3.1]{BrescianiVistoli} applied to $\cX\times_Y D_{k'}\to D_{k'}$.

\item If $E$ is a cylinder then it is the preimage of some constructible subset $E_n \subset |\sL_n(Y)|$. Letting $\cC_n$ be the preimage of $E_n$ in $|\sJ_n(\cX)|$, we have that $\cC_n$ is constructible and its preimage $\cC$ in $|\sJ(\cX)|$ is therefore a cylinder.

\end{enumerate}
\end{proof}

\begin{proposition}\label{propPreimageIntersectedWithERLocus}
Let $\cX$ be an equidimensional smooth finite type Artin stack over $k$ with affine diagonal, let $Y$ be a finite type algebraic space over $k$, let $\cX \to Y$ be a map that factors as a good moduli space morphism followed by a tame proper morphism, let $V \hookrightarrow Y$ be an open sub-algebraic-space, and let $\cU$ be the preimage of $V$ in $\cX$. Assume that $|V|$ is dense in $|Y|$, that $|\cU|$ is dense in $|\cX|$, and that the map $\cU \to V$ induced by $\cX \to Y$ is an isomorphism. Let $E \subset |\sL(Y)| \setminus |\sL(Y \setminus V)|$, and let $\cC$ be the preimage of $E$ in $|\sJ(\cX)|$.
\begin{enumerate}[label=(\alph*)]

\item For any field extension $k'$ of $k$, the map
\[
	\overline{(\cC \cap \cA_{\cX})}(k') \to E(k')
\]
is bijective.

\item If $E$ is a cylinder, then $\cC \cap \cA_\cX$ is a bounded cylinder.

\end{enumerate}
\end{proposition}

\begin{remark}
\autoref{propositionCrepantResolutionImpliesStableGMS} implies that $\cX$ admits a properly stable good moduli space, so $\cA_\cX$ is well defined.
\end{remark}

\begin{proof}
By \autoref{propositionCrepantResolutionImpliesStableGMS}, $\cX$ has a properly stable good moduli space that is proper over $Y$, and the stable locus $\cX^s$ of $\cX$ contains $\cU$. Let $\cX' \to \cX$ be the canonical reduction of stabilizers of Edidin-Rydh, let $\cU'$ be the preimage of $\cU$ in $\cX'$, and let $\cC'$ be the preimage of $\cC$ in $|\sJ(\cX')|$. By \cite[Theorem 2.11(2c) and (4c)]{EdidinRydh}, the composition $\cX' \to \cX \to Y$ is tame and proper. By \cite[Remark 2.12 and Definition 4.1]{EdidinRydh}, the map $\cX' \to \cX$ is finite type, separated, and representable, so $\cX'$ is finite type and has affine diagonal over $k$. By \cite[Theorem 2.11(4b)]{EdidinRydh} and the fact that $|\cU|$ is dense in $|\cX^s|$, we have $|\cU'|$ is dense in $|\cX'|$. By \cite[Remark 2.12]{EdidinRydh}, the map $\cU' \to \cU$ induced by $\cX' \to \cX$ is an isomorphism. Therefore \autoref{lemmaProperMapToSpaceLiftingArcsToTwistedArcs} implies that $\overline{\cC'}(k') \to E(k')$ is bijective for any field extension $k'$ of $k$ and that $\cC'$ is a cylinder if $E$ is a cylinder. Noting that $\cC \cap \cA_\cX$ is the image of $\cC'$ in $|\sJ(\cX)|$ by the definition of $\cA_\cX$, the proposition follows from \autoref{lemmaLiftingToEdidinRydh}.
\end{proof}

\begin{corollary}\label{corollaryLogTerminalConvergenceExpressionForGorensteinWithCanonical}
Let $\cX$ be a smooth irreducible finite type Artin stack over $k$ with affine diagonal, let $Y$ be a $\Q$-Gorenstein irreducible finite type scheme over $k$, let $\cX \to Y$ be a strongly birational map that factors as a good moduli space morphism followed by a tame proper morphism. Then $Y$ has log-terminal singularities if and only if $\bL^{-\ord_{K_{\cX/Y}} - \wt_\cX}$ is integrable on $\cA_\cX$, and in that case
\[
	\mu^\Gor_Y(\sL(Y)) = \int_{\cA_\cX} \bL^{-\ord_{K_{\cX/Y}} - \wt_\cX} \diff\nu_\cX.
\]
\end{corollary}

\begin{proof}
Because $\cX \to Y$ is strongly birational, there exists an nonempty open subscheme $V \subset Y$ such that $\cX \to Y$ is an isomorphism over $V$. Let $\cU$ be the preimage of $V$ in $\cX$, and let $Z$ be $Y \setminus V$ with its reduced subscheme structure. For any $n \in \Z_{\geq 0}$ set
\[
	E^n = \sL(Y) \setminus \theta_{Y,n}^{-1}(Z),
\]
and let $\cC^n$ be the preimage of $E^n$ in $\sJ(\cX)$. Note that $E^n$ is a cylinder, $E_n \subset E_{n+1}$, and $\bigcup_{n \in \Z_{\geq 0}} E_n = \sL(Y) \setminus \sL(Z)$. Noting that $V$ is smooth, each $\mu^\Gor_Y(E^n)$ is well defined. Therefore $Y$ has log-terminal singularities if and only if the sequence $\{\mu^\Gor_Y(E^n)\}_{n \in \Z_{\geq 0}}$ converges (see e.g., \cite[Chapter 7 Proposition 3.4.2]{ChambertLoirNicaiseSebag}), and in that case
\[
	\mu^\Gor_Y(\sL(Y)) = \mu^\Gor(\sL(Y) \setminus \sL(Z)) = \lim_{n \to \infty} \mu^\Gor_Y(E^n).
\]
By \autoref{propPreimageIntersectedWithERLocus}, each $\cC^n \cap \cA_\cX$ is a bounded cylinder and $\overline{(\cC^n \cap \cA_\cX)}(k') \to E^n(k')$ is a bijection for any field extension $k'$ of $k$. Note that $\cX$ has a good moduli space by \autoref{propositionCrepantResolutionImpliesStableGMS}. Thus \autoref{corollaryTwistedChangeOfVariables} and \autoref{propComparingRelativeCanonicalAndHeight} imply that
\[
	\mu^\Gor_Y(E^n) = \int_{\cC^n \cap \cA_\cX} \bL^{-\ord_{K_{\cX/Y}} - \wt_\cX} \diff\nu_\cX.
\]
Since $\bigcup_{n \in \Z_{\geq 0}} (\cC^n \cap \cA_\cX) = \cA_\cX \setminus |\sJ(\cX \setminus \cU)|$ and $\nu_\cX(|\sJ(\cX \setminus \cU)|) = 0$ by \autoref{theoremThinSubsetOfTwistedJets}, $\bL^{-\ord_{K_{\cX/Y}} - \wt_\cX}$ is integrable on $\cA_\cX$ if and only if the sequence
\[
	\left\{ \int_{\cC^n \cap \cA_\cX} \bL^{-\ord_{K_{\cX/Y}} - \wt_\cX} \diff\nu_\cX \right\}_{n \in \Z_{\geq 0}}
\]
converges, and in that case
\begin{align*}
	\int_{\cA_\cX}  \bL^{-\ord_{K_{\cX/Y}} - \wt_\cX} \diff\nu_\cX &= \int_{\cA_\cX \setminus |\sJ(\cX \setminus \cU)|}  \bL^{-\ord_{K_{\cX/Y}} - \wt_\cX} \diff\nu_\cX\\
	&= \lim_{n \to \infty} \int_{\cC^n \cap \cA_\cX}  \bL^{-\ord_{K_{\cX/Y}} - \wt_\cX} \diff\nu_\cX.
\end{align*}
\end{proof}

Finally, we obtain some consequences of \autoref{corollaryLogTerminalConvergenceExpressionForGorensteinWithCanonical} and complete our proofs of the main theorems of this paper.

\begin{corollary}\label{corollaryLogTerminalWithEffectiveDMResolution}
Let $Y$ be a $\Q$-Gorenstein irreducible finite type scheme over $k$, let $\cX$ be a smooth irreducible finite type Artin stack over $k$, and let $\cX \to Y$ be a tame proper birational map. If $K_{\cX/Y} \geq 0$, then $Y$ has log-terminal singularities.
\end{corollary}

\begin{proof}
Because having log-terminal singularities is Zariski-local, we may assume that $Y$ has affine diagonal. Then since tame proper maps have finite diagonal, $\cX$ has affine diagonal. Since $\cX \to Y$ is tame, $\cX$ is a tame stack. Thus $\cA_\cX = |\sJ(\cX)|$ is a bounded cylinder since $\sJ^\ell(\cX)$ is empty for all but finitely many $\ell$. Let $\cZ$ be the support of $K_{\cX/Y}$. By \autoref{theoremThinSubsetOfTwistedJets}, $\bL^{-\ord_{K_{\cX/Y}} - \wt_\cX}$ is integrable on $\cA_\cX$ if and only if it is integrable on the measurable set $\cA_\cX \setminus |\sJ(\cX \setminus \cZ)|$, and we note that $\ord_{K_{\cX/Y}}$ takes rational (i.e., not infinity) values on the latter. Therefore by \autoref{propositionWeightFunctionIsLocallyConstant} and the fact that $K_{\cX/Y} \geq 0$, the function $\bL^{-\ord_{K_{\cX/Y}} - \wt_\cX}$ is integrable on $\cA_\cX$. Thus \autoref{corollaryLogTerminalConvergenceExpressionForGorensteinWithCanonical} implies that $Y$ has log-terminal singularities.
\end{proof}

\begin{theorem}\label{theoremEffectiveResolutionImpiesLogTerminal}
Let $Y$ be a $\Q$-Gorenstein irreducible finite type scheme over $k$, let $\cX$ be a smooth irreducible finite type Artin stack over $k$, and let $\cX \to Y$ be a strongly birational map that factors as a good moduli space morphism followed by a tame proper morphism. If $K_{\cX/Y} \geq 0$, then $Y$ has log-terminal singularities.
\end{theorem}

\begin{proof}
\autoref{propositionCrepantResolutionImpliesStableGMS} implies that $\cX$ has a properly stable good moduli space $X$ and the induced morphism $X \to Y$ is proper. Let $\cX' \to \cX$ be the canonical reduction of stabilizers of Edidin-Rydh. Then $\cX'$ is smooth over $k$ by \cite[Theorem 2.11(1c)]{EdidinRydh}, and the composition $\cX' \to \cX \to Y$ is tame and proper by \cite[Theorem 2.11(2c) and (4c)]{EdidinRydh} and birational by \cite[Theorem 2.11(4b)]{EdidinRydh}. By \cite[Theorem 2.11(1c), Remark 2.12, and Definition 4.1]{EdidinRydh} and \cite[Lemma 8.2(1)]{SatrianoUsatine3}, we have that $K_{\cX' / \cX} \geq 0$. Then $K_{\cX'/Y} \geq 0$ by \cite[Lemma 8.2(4)]{SatrianoUsatine3}, so $Y$ has log-terminal singularities by \autoref{corollaryLogTerminalWithEffectiveDMResolution}.
\end{proof}

\begin{proof}[Proof of \autoref{mainTheoremCrepantResolutionImpliesLogTerminal}]
That (1) implies (2) is a special case of \autoref{theoremEffectiveResolutionImpiesLogTerminal}. That (2) implies (1) is \autoref{citedTheoremLogTerminalImpliesCrepantExists}.
\end{proof}

\begin{proof}[Proof of \autoref{maintheorem}]
By \autoref{mainTheoremCrepantResolutionImpliesLogTerminal}, $Y$ has log-terminal singularities. Thus \autoref{corollaryLogTerminalConvergenceExpressionForGorensteinWithCanonical} implies that $\bL^{-\wt_\cX}$ is integrable on $\cA_\cX$ and
\[
	\mu^\Gor_Y(\sL(Y)) = \int_{\cA_\cX} \bL^{-\wt_\cX}\diff\nu_\cX.
\]
For the last claim, first note that $\wt_\cX$ takes only finitely many values on each $|\sJ^\ell(\cX)|$ by \autoref{propositionWeightFunctionIsLocallyConstant}. Now Let $\cX' \to \cX$ be the canonical reduction of stabilizers of Edidin-Rydh. Since $\cX'$ is tame, the cyclotomic inertia stack $I_{\mu_\ell}(\cX)$ is empty for all but finitely many $\ell$. Therefore $\cA_\cX$ is contained in a finite union of the $|\sJ^\ell(\cX)|$, so $\wt_\cX$ takes only finitely many values on $\cA_\cX$.
\end{proof}

\bibliographystyle{alpha}
\bibliography{SHNCRAS}

\end{document}